\documentclass[12pt]{amsart}

\usepackage{fullpage}

\usepackage{amssymb,amsfonts,amsmath}

\usepackage{graphicx}
\usepackage{amssymb}
\usepackage{epstopdf}

\usepackage[all,cmtip]{xy}

\usepackage[T1]{fontenc}
\usepackage[utf8]{inputenc}
\usepackage{mathtools} 

\usepackage{subfig}

\usepackage[bbgreekl]{mathbbol}

\usepackage{tikz}

\usepackage{color}

\renewcommand{\int}{\operatorname{int}} 
 
\newcommand{\un}{\operatorname{u}}

\newcommand{\id}{\operatorname{id}}

\newcommand{\sx}{{\sf{x}}}

\newcommand{\Z}{\mathbb{Z}} 
\newcommand{\N}{\mathbb{N}} 
\newcommand{\R}{\mathbb{R}} 

\newcommand{\F}{\mathbb{F}}

\newcommand{\cA}{\mathcal{A}} 
 
\newcommand{\cC}{\mathcal{C}}

\newcommand{\cR}{\mathcal{R}}

\newcommand{\cV}{\mathcal{V}} 
\newcommand{\cW}{\mathcal{W}} 

\newcommand{\cZ}{\mathcal{Z}}

\newcommand{\fq}{\operatorname{fq}}

\newcommand{\e}{\epsilon} 
\newcommand{\hra}{\hookrightarrow} 
\newcommand{\imra}{\looparrowright} 
\newcommand{\sra}{\twoheadrightarrow} 
 
\newcommand{\ra}{\longrightarrow}

\newtheorem*{theorem*}{Theorem}
\newtheorem*{LBT}{4D-Light Bulb Theorem}
\newtheorem*{Thm10}{Theorem 10.5(2)}

\theoremstyle{definition}

\theoremstyle{plain}
\newtheorem{thm}{Theorem}[section]
\newtheorem{cor}[thm]{Corollary}
\newtheorem{lem}[thm]{Lemma}

\newtheorem{prop}[thm]{Proposition}
\theoremstyle{definition}
\newtheorem{defn}[thm]{Definition}
\newtheorem{rem}[thm]{Remark}

\title{Homotopy versus isotopy:\\ spheres with duals in 4--manifolds}

\author[R. Schneiderman]{Rob Schneiderman}
\email{robert.schneiderman@lehman.cuny.edu}
\address{Dept. of Mathematics, Lehman College, City University of New York, Bronx, NY}

\author[P. Teichner]{Peter Teichner}
 \email{teichner@mac.com}
\address{Max-Planck-Institut f\"ur Mathematik, Bonn, Germany}

\begin{document}

\maketitle

\begin{abstract}
Dave Gabai recently proved a smooth $4$-dimensional ``Light Bulb Theorem'' in the absence of 2-torsion in the fundamental group. 
We extend his result to 4--manifolds with arbitrary fundamental group by showing that an invariant of Mike Freedman and Frank Quinn gives the complete obstruction to ``homotopy implies isotopy'' for embedded 2--spheres which have a common geometric dual. The invariant takes values in an $\F_2$-vector space generated by elements of order 2 in the fundamental group and has applications to unknotting numbers and pseudo-isotopy classes of self-diffeomorphisms. Our methods also give an alternative approach to Gabai's theorem using various maneuvers with Whitney disks and a fundamental isotopy between surgeries along dual circles in an orientable surface. 
\end{abstract}



\section{Introduction and results}
We work in the category of smooth manifolds. Our starting point is Gabai's $4$-dimensional LBT \cite[Thm.1.2]{Gab}:
\begin{LBT} Let $M$ be an orientable $4$--manifold such that $\pi_1M$ has no elements of order~$2$.
If $R, R':S^2 \hra M$ are embedded spheres in $M$ which are homotopic and have the same geometric dual, then $R$ is isotopic to $R'$.
\end{LBT}

Here a \emph{geometric dual} to a map $R:S^2\to M$ is an embedded sphere with trivial normal bundle which intersects $R$ transversely and in a single point. 
The necessity of the $\pi_1$-condition was shown by Hannah Schwartz in \cite{Schwartz} and also follows from Theorem~\ref{thm:4d-light-bulb}.

In this paper we extend the above LBT to a version  for arbitrary fundamental groups. Fix  a connected orientable $4$--manifold $M$ and a map $f:S^2\to M$ with geometric dual $G:S^2\hra M$. Consider the following set, measuring ``homotopy modulo isotopy'':
\[
\cR^G_{[f]}:= \{ R: S^2 \hra M \mid R \text{ is homotopic to } f  \text{ and has } G \text{ as geometric dual} \}\slash  \text{isotopies of }R.
\]

Let $\F_2T_M$ be the $\F_2$-vector space with basis $T_M:=\{g\in\pi_1M \mid g^2=1\neq g\}$, the elements of order two (2-torsion). It turns out that the self-intersection invariant for maps $S^3 \imra M^4\times \R^2$ with transverse double points gives a homomorphism $\mu_3:\pi_3M\to \F_2T_M$ (see Lemma~\ref{lem:Wall}).

\begin{thm}\label{thm:4d-light-bulb}
The abelian group $\F_2T_M$ acts transitively on $\cR^G_{[f]}$, and $\cR^G_{[f]}\neq\emptyset$ if and only if Wall's reduced self-intersection invariant $\widetilde\mu(f)$ vanishes.
 The stabilizer of $R\in\cR^G_{[f]}$ is the subgroup $\mu_3(\pi_3M)\leq \F_2T_M$.
 If $R,R':S^2\hra M$ represent the same element in $\cR^G_{[f]}$ and agree near $G$ then they are  isotopic by an isotopy supported away from $G$. 
\end{thm}
Note that the group action applied to $R\in\cR^G_{[f]}$ 
gives a bijection $\cR^G_{[f]}\longleftrightarrow \F_2T_M/\mu_3(\pi_3M)$. 
As a consequence Gabai's LBT follows:
If $\pi_1M$ contains no 2-torsion then $\F_2T_M=\{0\}$ and hence $ \cR^G_{[f]}$ consists of a single isotopy class. In fact, in the second version of his paper Gabai strengthens his result to a ``normal form'' \cite[Thm.1.3]{Gab} which in our language translates to saying that there is a surjection $\F_2T_M \twoheadrightarrow  \cR^G_{[f]}$. Examples where this projection is not injective were given in \cite{Schwartz}, providing 4-manifolds $M$ for which $\mu_3$ is non-trivial.

\begin{rem}\label{rem:common}
Hannah Schwartz pointed out examples showing that the geometric dual $G$ needs to be common to both spheres.
Consider closed, 1-connected 4-manifolds $M_0, M_1$ 
such that
\[
\varphi: M_0 \# S^2 \times S^2\overset{\cong}{\ra} M_1\# S^2 \times S^2
\]
is a diffeomorphism which preserves the $S^2 \times S^2$-summands homotopically. 
See for instance the examples in \cite{AKMRS}.
Consider the spheres $R:= S^2 \times p$ and $R':=\varphi(S^2 \times p)$ in $M_1\# S^2 \times S^2$ with geometric duals $p \times S^2$ and $\varphi(p \times S^2)$. Then $R$ and $R'$ are homotopic but can't be isotopic, otherwise the ambient isotopy theorem would lead to a diffeomorphism $M_0\cong M_1$. 
\end{rem}

\subsection{Consequences of Theorem~\ref{thm:4d-light-bulb} and its proof}\label{subsec:corollaries-of-LBT}

\begin{cor}\label{cor:infinitely-many}
There exist  $4$--manifolds $M$ and $f:S^2\imra M$ with infinitely many free isotopy classes of embedded spheres homotopic to $f$ (and with common geometric dual). These manifolds also admit infinitely many distinct pseudo-isotopy classes of self-diffeomorphisms.
\end{cor}
These self-diffeomorphisms (carrying one sphere to the other) will be constructed in Lemma~\ref{lem:pseudo-isotopy}.
For example, let $M'$ be any $4$--manifold obtained by attaching 2-handles to a boundary connected sum of copies of $S^1\times D^3$ such that $\Z/2\ast\Z/2\leq\pi_1M'$. This infinite dihedral group contains infinitely many distinct reflections (which are of order~$2$). It follows from Theorem~\ref{thm:4d-light-bulb} that there exist infinitely many isotopy classes of spheres homotopic to $p \times S^2$ in $M:=M' \#(S^2\times S^2)$, all with the same geometric dual $S^2 \times p$ and all related by diffeomorphisms. 

Under the hypotheses of Theorem~\ref{thm:4d-light-bulb} we also have the following results.
\begin{cor}\label{cor:concordance}
Concordance implies isotopy for spheres with a common geometric dual. 
\end{cor}

\begin{cor}\label{cor:5d}
If $R, R':S^2\hra M^4$ have a common geometric dual in $M$ then $R$ and $R'$ are isotopic in  $M \times \R$ if and only if they are isotopic in $M$. 
\end{cor}

The two results are actually ``scholia'', i.e.\  corollaries to our proof of Theorem~\ref{thm:4d-light-bulb}. Namely, we show that the bijections in our theorem are induced by a based concordance invariant $\fq(R,R')\in\F_2T_M/\mu_3(\pi_3M) $ used by Mike Freedman and Frank Quinn in \cite[Thm.10.5(2)]{FQ} and later named $\fq$ by Richard Stong \cite[p.2]{Stong}.

As explained in Section~\ref{sec:FQ}, Freedman--Quinn actually use the self-intersection invariant $\mu_3(H)\in\F_2T_M$ of a map $S^2 \times I\imra M \times \R \times I$ with transverse double points obtained by perturbing the track of a based homotopy $H$ between $R$ and $R'$ in $M \times \R$, explaining Corollary~\ref{cor:5d}. Stong states that in the quotient $\F_2T_M/\mu_3(\pi_3M)$, the choice of  $H$ disappears and gives $\fq(R,R')$. This will be proven in Section~\ref{sec:homotopies} for any two spheres $R,R'$ in $M$ that are based homotopic, as is the case for $R,R' \in \cR^G_{[f]}$. 

\begin{cor}\label{cor:fq}
If $R, R':S^2\hra M^4$ have a common geometric dual and are homotopic, then they are isotopic if and only if $\fq(R,R')=0$.
\end{cor}
This follows from the relation $\fq(t\cdot R,R)=t$ for all $t\in\F_2T_M/\mu_3(\pi_3M)$, between our action and the Freedman--Quinn invariant, see Section~\ref{subsec:geo-action-def}. 

For the next scholium we consider a ``relative unknotting number'' $\un(R,R')$ for homotopic spheres $R,R': S^2\hra M$: By assumption, there is a sequence of finger moves and Whitney moves that lead from $R$ to $R'$, compare Section~\ref{subsec:regular-homotopy}. Let $\un(R,R')\in \N_0$ denote the minimal number of finger moves required in any such homotopy. 

In general, this is an extremely difficult invariant to compute, even though we'll see in Lemma~\ref{lem:compute-mu-by-crossing-changes} that one always has the estimate $\un(R,R') \geq |\fq(R,R')|$.
Here the {\em support} $|t|$ is the number of non-zero coefficients in $t\in \F_2T_M$, and for an equivalence class $[t]\in \F_2T_M/\mu_3(\pi_3M)$ we let $|[t]|$ be the minimum support of all representatives. 
Michael Klug pointed out that in the presence of a common geometric dual this estimate becomes an equality (see the last part of Section~\ref{sec:infinitely-many}):
\begin{cor}\label{cor:unknotting number}
For $R,R'\in \cR^G_{[f]}$, the relative unknotting number equals the support of the Freedman-Quinn invariant: $\un(R,R') = |\fq(R,R')|$.
\end{cor}

Using the 4-manifold $M$ introduced below Corollary~\ref{cor:infinitely-many}, we see that any (arbitrary large) number is realized as the relative unknotting number between spheres in $M$. See \cite{JKRS} for results on unknotting numbers of $2$-spheres in $S^4$ relative to the unknot.

\subsection{An isotopy invariant statement}\label{subsec:more-general-LBT}
Even though the original LBT's in $S^2 \times S^1$ and $S^2 \times S^2$ are extremely well motivated, see \cite{Gab}, readers may find it confusing that our set $\cR^G_{[f]}$ is  {\em not isotopy invariant}: If we do a finger move on $R\in \cR^G_{[f]}$ that introduces two additional intersection points with the dual $G$, the resulting embedded sphere is isotopic to $R$ but not in $\cR^G_{[f]}$ any more. In other words, if one wants isotopy invariant statements, one should not fix a sphere $G$ as in the LBT. This problem can be addressed as follows.
\begin{defn}\label{def:dual-pairs}
For fixed $R:S^2\hra M^4$ with fixed geometric dual $G$, consider pairs of embeddings $R', G':S^2\hra M$ such that:
\begin{itemize}
\item $R'$ is {\em homotopic} to $R$ via $R_s:S^2\to M$,
\item $G'$ is {\em isotopic} to $G$ via $G_s:S^2\hra M$, and
\item $G_s$ is a geometric dual to $R_s$  for each  $s\in I$.
\end{itemize}
Denote by $\cR_{R,G}$ the set of isotopy classes of such pairs $(R',G')$, where an \emph{isotopy of a pair} is a pair $(R_s,G_s)$ as above where $R_s$ is in addition an isotopy.
\end{defn}
Then an isotopy $R_s$ can be embedded into an ambient isotopy $\varphi_s:M\overset{\cong}{\ra} M$ and hence leads to pairs $(R_s, \varphi_s(G))$ that are all equal in $\cR_{R,G}$.

\begin{thm}\label{thm:4d-light-bulb-isotopy}
The group $\F_2T_M$ acts transitively on $ \cR_{R,G}$, with stabilizers $\mu_3(\pi_3M) $. 
The action on the basepoint $(R,G)$ hence gives a bijection $\F_2T_M/\mu_3(\pi_3M)\longrightarrow \cR_{R,G}$. 
The inverse of this bijection is given by
the Freedman-Quinn invariant 
$\fq(R,\cdot):\cR_{R,G}\longrightarrow \F_2T_M/\mu_3(\pi_3M)$.  
\end{thm}

It turns out that this result is equivalent to Theorem~\ref{thm:4d-light-bulb} above, and as a consequence we won't follow up on it in this paper. For example, to derive Theorem~\ref{thm:4d-light-bulb} we can use Lemma~\ref{lem:based} to turn any homotopy $R_s$ into one that satisfies the last condition in the above definition (with $G_s=G$ a constant isotopy).

\subsection{Outline of the proof of Theorem~\ref{thm:4d-light-bulb}}\label{sec:outline-proof}
\begin{figure}[ht!]
        \centerline{\includegraphics[scale=.32]{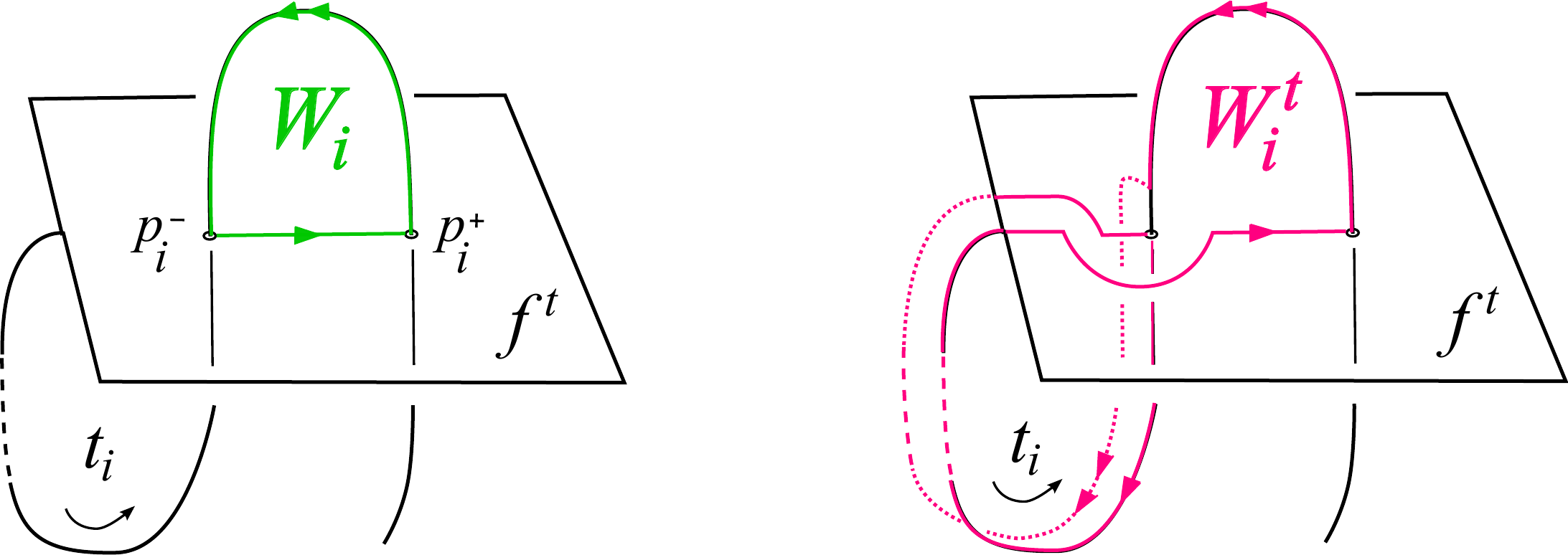}}
        \caption{The Whitney disks $W_i$ and $W_i^t$ pair the same self-intersections $p_i^{\pm}$ of $f^t$. 
        On the boundary $\partial W_i^t$ differs from $\partial W_i$ as it departs and approaches from the negative self-intersection $p_i^-$ in different local sheets of $f^t$. }
        \label{fig:sheet-change-w-disks-1}
\end{figure}

Our action of $t=t_1+\cdots +t_n \in {\F_2T_M}$ on $R\in \cR^G_{[f]}$ will be defined as follows. First create a generic map $f^t:S^2\imra M$ by doing $n$ finger moves on $R$ along arcs representing $t_i\in T_M$. 
There is a collection of $n$ Whitney disks $\cW\subset M\smallsetminus G$ for $f^t$ which are ``inverse'' to the finger moves, i.e.\ the result of doing the Whitney moves along the Whitney disks in $\cW$ is isotopic to $R$. 

Since $t_i^2=1$, we can use $G$ to find a different collection $\cW^t\subset M\smallsetminus G$ of $n$ Whitney disks on $f^t$ which induce the same pairings of self-intersections of $f^t$ as $\cW$ but which induce different sheet choices for the preimages of the self-intersections,  see Figure~\ref{fig:sheet-change-w-disks-1} and Lemma~\ref{lem:choice-of-disks-exists}. 
The result of doing the Whitney moves on the Whitney disks in $\cW^t$ is an embedded sphere denoted by $t\cdot R$, which by construction is homotopic to $R$ and has geometric dual $G$. We'll show in Section~\ref{sec:proof-main} that $t\cdot R$ is isotopic to $R$ if and only if $t\in{\mu_3(\pi_3M) }$.

The isotopy class of $t\cdot R$ can also be described explicitly without knowing the Whitney disk collection $\cW^t$ by the following {\it Norman sphere}, built from $f^t$ and $G$ (see Section~\ref{sec:norman}). Instead of doing Whitney moves on $\cW^t$, the \emph{Norman trick} \cite{Nor} can be applied to eliminate the self-intersections of $f^t$ by tubing into the dual sphere $G$ along arcs in $f^t$. This operation also involves a choice of local sheets for each self-intersection,
and we will show in Section~\ref{sec:choices-of-arcs-and-pairings} that $t\cdot R$ is isotopic to the result of applying the Norman trick using the \emph{opposite} local sheets at each negative self-intersection compared to the original finger moves.

Gabai's proof of his LBT in \cite{Gab} introduces a notion of ``shadowing a homotopy by a tubed surface'', which uses careful manipulations of several types of tubes and their guiding arcs to control the isotopy class of the result of a homotopy between embeddings.
In addition to using the Norman trick, Gabai also works with tubes along framed arcs that extend into the ambient $4$--manifold, including the guiding arcs for finger moves.

Our proof of Theorem~\ref{thm:4d-light-bulb}, which implies Gabai's LBT, takes a different viewpoint by focusing on the generic sphere $f$ which is the middle level of a homotopy between embeddings $R$ and $R'$, given by finger moves and then Whitney moves. By reversing the finger moves, we see that both these embeddings are obtained from $f$ by sequences of Whitney moves along two collections of Whitney disks. We analyze all choices involved in such collections of Whitney disks and show how they are related to the Freedman--Quinn invariant $\fq(R,R')$.

\begin{figure}[ht!]
        \centerline{\includegraphics[scale=.3]{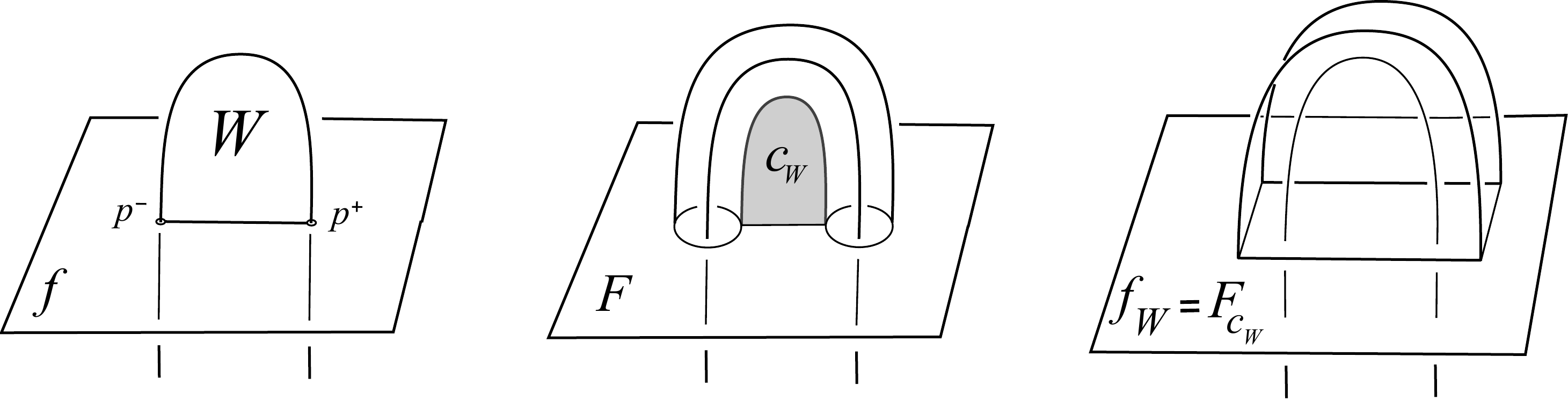}}
        \caption{Left: $W$ pairing self-intersections $p^\pm$ of $f$. Center: The surface $F$ obtained by tubing $f$ to itself admits a cap $c_W$. Right: The result $f_W$ of doing the $W$-move is isotopic to the result $F_{c_W}$ of surgery on $c_W$.}        
        \label{fig:tube-caps-0}
\end{figure}

Our key tool is the relationship between Whitney moves and surgeries on surfaces as shown in Figure~\ref{fig:tube-caps-0}.
Note that the dual curve to $\partial c_W$ on $F$ is a meridional circle to $F$ which bounds a meridional disk. This meridional disk $d$ is usually of not much use since it intersects $F$ (exactly once). However, in the presence of the dual $G$ to $F$, we can tube $d$ into $G$, removing this intersection and obtaining a cap $c_G$ which is disjoint from $F$ and $c_W$, see Figure~\ref{fig:tube-caps-2}. 

Now we can apply the fundamental isotopy between surgeries along dual curves in an orientable surface
showing that surgery on $c_W$ is isotopic to surgery on $c_G$ (see Section~\ref{sec:capped-surface-w-move}). 
And if $W'$ is any other Whitney disk for $f$ having the same Whitney circle $\partial W'=\partial W$, then we see that surgery on $c_G$ is also isotopic to surgery on $c_{W'}$. All together, this implies that the Whitney moves on $f$ along $W$, respectively $W'$, give isotopic results $R$, respectively $R'$! 

This outline finishes our proof in the very simple case that our collections contain only one Whitney disk and the Whitney circles agree. Multiple Whitney disks in our collections correspond to higher genus capped surfaces and the remaining steps in the argument are ``only'' about showing independence of Whitney circles.  Section~\ref{sec:choices-of-arcs-and-pairings} consists of a sequence of lemmas that reduce this dependence only to the choices of sheets at self-intersections whose group elements are of order 2. Fortunately, these are exactly detected by the Freedman--Quinn invariant, finishing our proof.

\subsection{Isotopy of disks with duals}\label{subsec:isotopy-of-disks}

Using the last sentence of Theorem~\ref{thm:4d-light-bulb} one sees that the bijection described by the theorem is equivalent to the following result. In one direction one works in the complement of a tubular neighborhood of the dual sphere $G$, and in the other direction one attaches $D^2 \times S^2$ to the component $S^1 \times S^2\subseteq \partial M$.

\begin{thm}\label{thm:disks}
Let $M$ be a connected oriented manifold with $S^1 \times  S^2\subseteq \partial M$ and let $R:(D^2,S^1) \hra (M,\partial M)$ be an embedding with $\partial R = S^1 \times  p\subset S^1 \times  S^2$. Then the set
\[
\cR_{R}:= \{ R': D^2 \hra M \mid R' \text{ is homotopic rel boundary to } R\} \slash \text{isotopy of $R'$ rel boundary}
\]
is in bijection with $\F_2T_M/\mu_3(\pi_3M)$ via the Freedman--Quinn invariant $\fq(R,\cdot)$.
\end{thm}
Note that $G=q \times S^2\subset\partial M$ is a geometric dual sphere to any $R'\in \cR_{R}$ in the sense that it has trivial normal bundle and intersects $\partial R'$ transversely in the single point $q\times p$. Hence the question arises whether such a theorem can continue to hold for embedded disks (rel boundary) with geometric dual $G:S^2\hra\partial M$ that is \emph{not} assumed to lie in a component $S^1 \times  S^2\subseteq \partial M$.  It's important to point out that our key Lemma~\ref{lem:choice-of-disks-exists} still works in this setting.

The first version of our paper contained a flawed argument that would have shown exactly such a generalization of our theorem. We are thankful to Dave Gabai for pointing out that our technique of sliding Whitney disks over each other may fail to preserve the relevant isotopy class in the case of ``self-slides'' in the setting of disks $R:D^2\hra M$; see the last paragraph in Section~\ref{sec:w-disk-slides}. 

In fact, Gabai recently posted a paper \cite{G2} constructing a generalization of the Freedman--Quinn invariant using work that goes back to Dax \cite{D}. He showed that a ``self-feeding'' construction gives homotopic neat embeddings $R,R':D^2\hra M$ with common dual in $\partial M$ and vanishing Freedman-Quinn invariant but which are {\it not} isotopic rel boundary, as detected by the Dax invariant. Here $M$ is the boundary connected sum of $S^1 \times D^3$ and $S^2 \times  D^2$, so $\partial M$ does not contain $S^1 \times S^2$.

It turns out that our Whitney disk self-sliding does indeed preserve the relevant isotopy class in the setting of spheres $R:S^2\hra M$ with a geometric dual, as is carefully discussed
in this current version of our paper; see Lemmas~\ref{lem:slide-w-disk-across-D-z} and
Lemma~\ref{lem:w-disk-self-slide}. Looking at the proof of Lemma~\ref{lem:slide-w-disk-across-D-z}, we see that it also applies to disks $R:D^2\hra M$ with geometric dual $G:S^2\hra\partial M$ exactly if there is an $S^1$-family of dual spheres running along the boundary of the disk. In other words, the component of $\partial M$ that contains the geometric dual $G$ must be a copy of $S^1 \times S^2$ with $G=q \times S^2$. This is consistent with Theorem~\ref{thm:disks} above.

Danica Kosanovi\'c and the second author recently showed that the Dax invariant classifies isotopy classes of embeddings $D^2\hra M$ with fixed boundary and dual sphere $G:S^2 \hra \partial M$. In particular, this contains the complete classification of Gabai's new examples. The case where the dual sphere $G$ is not contained in the boundary of $M$ is very interesting and completely open.

\subsection{Other proofs}\label{subsec:intro-other-proofs}
In Section~\ref{sec:ambient-morse-pi1-embedding-thm} we provide an alternative proof of Corollary~\ref{cor:fq}, sufficient to conclude Gabai's LBT. We start with the uniqueness part of Freedman and Quinn's Theorem~10.5 \cite{FQ,Stong} which gives a concordance $C: S^2 \times I \hra M\times I$ from $R$ to $R'$ if and only if $\fq(R,R')=0$. By analyzing the handle decomposition on $S^2 \times I$ related to the composition $p_2\circ C$, we show directly that in the presence of a geometric dual ``ambient handles can be cancelled'' and $C$ can be replaced by an isotopy. 
This argument gives a third proof of Gabai's LBT, and a second proof of our generalization. 

It's interesting to note that Gabai also provided a different argument for the special case of $M=S^2 \times S^2$, which was then exposited by Bob Edwards in \cite{Edwards}. 
An important step is ``codimension~2 embedded Morse theory'', sometimes also referred to as ``ambient Morse theory'' applied to a map of a surface to a 4-manifold. This is one dimension below the same technique for 3-manifolds in 5-manifolds used in our Section~\ref{sec:ambient-morse-pi1-embedding-thm}.

\vspace{3mm}
\noindent

{\em Acknowledgements:}  It is a pleasure to thank David Gabai and Daniel Kasprowski for helpful discussions. Thanks also to the referees for careful readings and helpful comments. The first author was supported by a Simons Foundation \emph{Collaboration Grant for Mathematicians}, and both authors thank the Max Planck Institute for Mathematics in Bonn, where this work was carried out.

\tableofcontents


 \section{Preliminaries on surfaces in 4-manifolds}\label{sec:preliminaries} 

We work in the smooth category throughout. 
Smoothing of corners will be assumed without mention during cut-and-paste operations on surfaces.
Orientations will usually be assumed and suppressed, as will choices of basepoints and whiskers.

In the smooth category, a {\em generic} map, written $f:\Sigma^2\imra M^4$,  is a smooth map which is an embedding, except for a finite number of transverse double points. This is the same as a generic immersion and means that there are coordinates on $\Sigma$ and $M$ such that $f$ looks locally like the inclusion $\R^2  \times  \{0\}\subset \R^4$ or like a transverse double point $\R^2 \times \{0\} \cup \{0\} \times \R^2 \subset \R^4$.

\subsection{Homotopy classes of surfaces}\label{subsec:regular-homotopy} 

We will use the following fact (see \cite[Sec.3]{PRT}) about homotopy classes $ [S^2,M]$ of maps $f:S^2\to M$ when $M$ is oriented: The inclusion of generic maps into all smooth (or even all continuous) maps induces a bijection
\[
\{f: S^2 \imra M \mid \mu_1(f) = 0\} /  \{\text{isotopies, finger moves, Whitney moves}\} \longleftrightarrow [S^2,M] 
\]
where $\mu_1(f)\in\Z$ denotes the coefficient at the trivial group element in the self-intersection invariant $\mu(f)$, see Section~\ref{sec:mu}. Note that $\mu_1(f)$ can be changed arbitrarily by (non-regular) cusp homotopies and in the following, we'll always tacitly assume that this has been done such that $\mu_1(f)=0$.  

An isotopy of generic maps is a map $H:S^2 \times I \to M$ such that $H(\cdot,t)$ is generic for all $t\in I$. Note that in a finger move or Whitney move this is true for all but one time $t$.

In the setting of the LBT, finger moves in a generic homotopy from $R$ to $R'$ having common geometric dual $G$ may be assumed to be disjoint from $G$ since finger moves are supported near their guiding arcs.
By the following lemma, the Whitney moves in such a homotopy may also be assumed to be disjoint from $G$ because one easily finds a preliminary isotopy that makes $R$ and $R'$ agree near $G$. This is also \cite[Lem.6.1]{Gab} where the 3D-LBT is used in the proof. For the convenience of the reader, and for completeness, we give an elementary argument.
 
\begin{lem}\label{lem:based}
If $R,R':S^2\hra M$ agree near a common geometric dual $G$ and are homotopic in $M$ then there exists a finite sequence of isotopies, finger moves and Whitney moves in $M\smallsetminus G$ leading from $R$ to $R'$.
\end{lem}
\begin{proof}
We first show that $R,R'$ are base point preserving homotopic, noting that they both send a base-point $z_0\in S^2$ to $z=R\cap G=R'\cap G$ and hence represent elements $[R],[R']\in\pi_2(M,z)$. Any free homotopy $H$ from $R$ to $R'$ identifies $[R']$ with $g\cdot [R]$, where the loop $H(z_0\times I)$ represents $g\in \pi_1(M,z)$ and we use the $\pi_1$-action on $\pi_2$. 

Now take a free homotopy $H$ that is transverse to $G\subset M$ and consider the submanifold $L:=H^{-1}(G)\subset S^2 \times  I$. $L$ is a 1-manifold with boundary $z_0 \times \{0,1\}$ since $R$ and $R'$ intersect $G$ exactly in $z\in M$. This implies that $L$ has a component $L_0$ which is  homotopic (in $S^2 \times I$) to $z_0 \times  I$ rel endpoints. As a consequence, the above group element $g$ is also represented by $H(L_0)\subset G \cong S^2$ and hence $[R]=[R']$. 

Removing an open normal bundle $\nu G$ of $G$ leads to a 4-manifold $W:= M \smallsetminus \nu G$ with a new boundary component $\partial_0 W\cong S^2 \times S^1$. $W$ contains two embedded disks $D$ and $D'$ with the same boundary in $\partial_0 W$. These disks complete to the spheres $R$ and $R'$ when adding $\nu G$ back into the 4-manifold. 

We claim that $D$ and $D'$ are homotopic rel boundary in $W$ by the homological argument below. Granted this fact, we see from the above discussion that there is also a regular homotopy rel boundary from $D$ to $D'$ in $W$. Approximating it by a generic map we obtain the desired type of homotopy in $M\smallsetminus \nu G$.

To show that $D$ and $D'$ are homotopic rel boundary in $W$, it suffices to show that the glued up sphere $S:=D\cup_{\partial} D'$ is null homotopic in $W$. Since $R$ intersects $G$ in a single point, it follows from Seifert-van Kampen that the inclusion induces an isomorphism $\pi_1W \cong \pi_1M$, with base-points taken on $\partial_0 W$.
The long exact sequence of the pair $(M,W)$ for homology with coefficients in $\Z\pi_1W$ gives exactness for
\[
H_3(M,W;\Z\pi_1W) \ra H_2(W;\Z\pi_1W) \ra H_2(M;\Z\pi_1M).
\]
The Hurewicz isomorphism identifies the map on the right hand side with $\pi_2W\ra \pi_2M$ which sends $S$ to zero by our conclusion on $R,R'$ being based homotopic. By excision and Lefschetz duality,
\[
H_3(M,W;\Z\pi_1W) \cong H_3(S^2 \times D^2, S^2 \times S^1;\Z\pi_1W)\cong H^1(S^2 \times D^2;\Z\pi_1W) =0
\]
which implies that $[S]=0$.
\end{proof}
We note that Lemma~\ref{lem:based} is the reason why free (versus based) homotopy and isotopy agree in the presence of a common dual, and in particular, why we don't have to divide out by the conjugation action of $\pi_1M$ in Theorem~\ref{thm:4d-light-bulb}.

In the rest of the paper, we will turn a sequence of finger moves and Whitney moves as in Lemma~\ref{lem:based} into an isotopy, provided the Freedman--Quinn invariant vanishes. If $f:S^2\imra M$ is the {\em middle level} of such a sequence, i.e.\ the result of all finger moves on $R$, then there are two {\em clean collections of Whitney disks} for $f$ in $M$: One  collection $\cW$ reverses all the finger moves and leads back to $R=f_\cW$, and the other collection $\cW'$ does the interesting Whitney moves to arrive at $R'=f_{\cW'}$. 

Thus the triple $(f,\cW,\cW')$ represents the entire homotopy from $R$ to $R'$ up to isotopy. By construction, each of the two collections of Whitney disks is {\em clean} in the sense of Definition~\ref{def:clean} which formalizes the above discussion. In particular, since the result of Lemma~\ref{lem:based} is a homotopy in the complement of $G$, the notion of clean Whitney disk will include disjointness from $G$.
Then all our maneuvers will stay in the complement of $G$, explaining the last sentence in Theorem~\ref{thm:4d-light-bulb}.

\subsection{Self-intersection invariants}\label{sec:mu}
Let $M$ be a smooth oriented $4$--manifold and let $f:S^2\imra M$ be a generic sphere with a whisker from the base point of $M$ to $f$. A loop in $f(S^2)$ that changes sheets exactly at one self-intersection $p$ is called a {\it double point loop} at $p$. After choosing an orientation of the double point loop, it determines an element $g\in\pi_1M$ associated to $p$.
The orientation of a double point loop corresponds to a \emph{choice of sheets} at $p$, i.e.~a choice of a point $x\in f^{-1}(p)$ that is the starting point of the preimage of the loop.

The self-intersection invariant $\mu(f)\in \Z[\pi_1M]/ \langle g-g^{-1} \rangle $ is defined by summing the group elements represented by double point loops of $f$, with the coefficients coming from the usual signs determined by the orientation of $M$.
The relations $g-g^{-1}=0$ in the integral group ring account for the above choices of sheets.

Then $\mu(f)$ is invariant under regular homotopies of $f$ and changes by $\pm 1$ under a cusp homotopy. 
Therefore, taking $\mu(f)$ in a further quotient that also sets the identity element $1\in\pi_1M$ equal to $0$ makes the resulting {\it reduced} self-intersection invariant $\widetilde\mu(f)$ invariant under arbitrary based homotopies of $f$. The vanishing of $\widetilde \mu(f)$ in fact only depends on the unbased homotopy class of $[f]$.

The analogous reduced self-intersection invariant defined for generic $3$-spheres in $6$--manifolds will  be relevant in Section~\ref{sec:FQ}.

\subsection{Whitney disks and Whitney moves}\label{sec:clean-w-disks-moves}
Suppose that a pair $p^\pm$ of oppositely-signed self-intersection points of $f:S^2\imra M$ have equal group elements for some choices of sheets at $p^+$ and $p^-$. 
Then the pair $p^\pm$ admits an embedded null-homotopic \emph{Whitney circle}  $\alpha\cup\beta=f(a)\cup f(b)$ for disjointly embedded arcs $a$ and $b$ joining the preimages $x^+,y^+$ and $x^-,y^-$ of $p^+$ and $p^-$, as in Figure~\ref{fig:w-arcs-disk-move-group-element-1}. 
Such $\alpha$ and $\beta$ are called \emph{Whitney arcs}.
\begin{figure}[ht!]
        \centerline{\includegraphics[scale=.3]{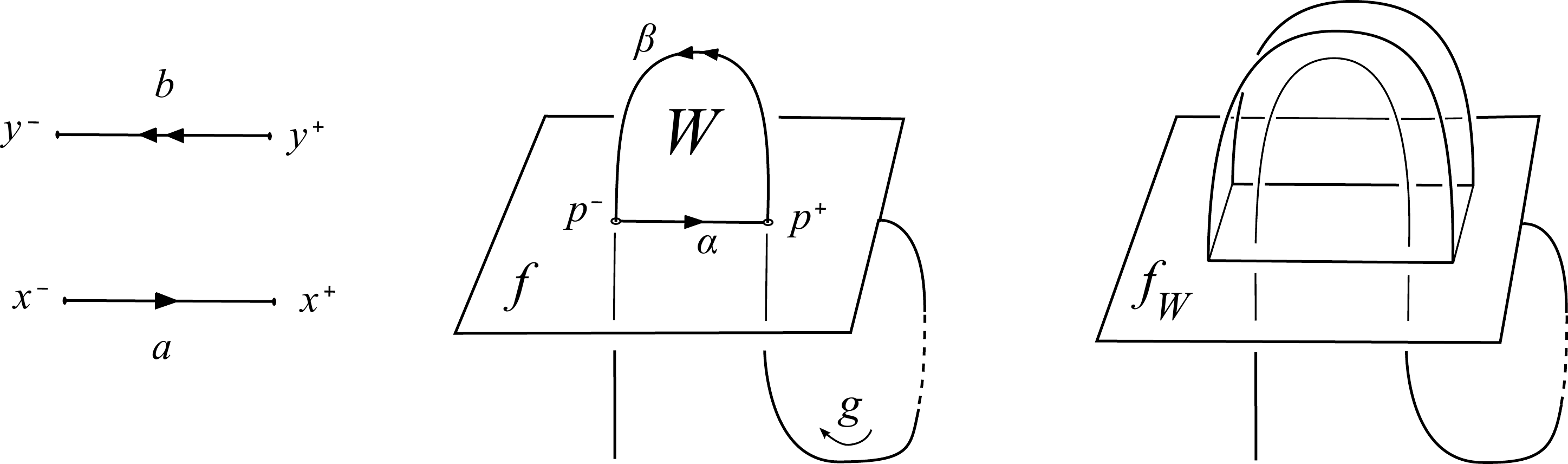}}
        \caption{Left: In the domain of $f$. Center: The horizontal sheet of $f$ appears in the `present' as does the Whitney disk $W$, and the other sheet of $f$ appears as an arc which is understood to extend into `past and future', with the dashed part indicating where $f$ extends outside the pictured $4$--ball neighborhood of $W$ in $M$. Right: After the Whitney move guided by $W$.}
        \label{fig:w-arcs-disk-move-group-element-1}
\end{figure}

The center of Figure~\ref{fig:w-arcs-disk-move-group-element-1} also shows a \emph{Whitney disk} $W$ with boundary $\partial W=\alpha\cup\beta$ pairing self-intersections $p^\pm$ with group element $g\in\pi_1M$. 
The right side of Figure~\ref{fig:w-arcs-disk-move-group-element-1} shows the result $f_W$ of doing a \emph{Whitney move} on $f$ guided by $W$,
which is
an isotopy of one sheet of $f$, supported in a regular neighborhood of $W$, that eliminates the pair $p^\pm$.
Combinatorially, $f_W$ is constructed from $f$ by replacing a regular neighborhood of one arc of $\partial W$ with a \emph{Whitney bubble} over that arc. This Whitney bubble is formed from two parallel copies of $W$ connected by a curved strip which is normal to a neighborhood in $f$ of the other arc. Figure~\ref{fig:w-arcs-disk-move-group-element-1} shows $f_W$
using a Whitney bubble over $\alpha$. Although both these descriptions of $f_W$ involve a choice of arc of $\partial W$, up to isotopy $f_W$ is independent of this choice. 

The construction of an embedded Whitney bubble requires that $W$ is \emph{framed} (so that the two parallel copies used above do not intersect each other), and Whitney disks which do not satisfy the framing condition are called \emph{twisted} (see eg.~\cite[Sec.7A]{ST1}).

\subsection{Sliding Whitney disks}\label{sec:w-disk-slides}
We describe here an operation that ``slides'' Whitney disks over each other. 
This maneuver changes the Whitney arcs while preserving the isotopy class of the results of the Whitney moves, and will be used in the proof of the key Proposition~\ref{prop:common-arcs-w-moves}.
\begin{figure}[ht!]
        \centerline{\includegraphics[scale=.24]{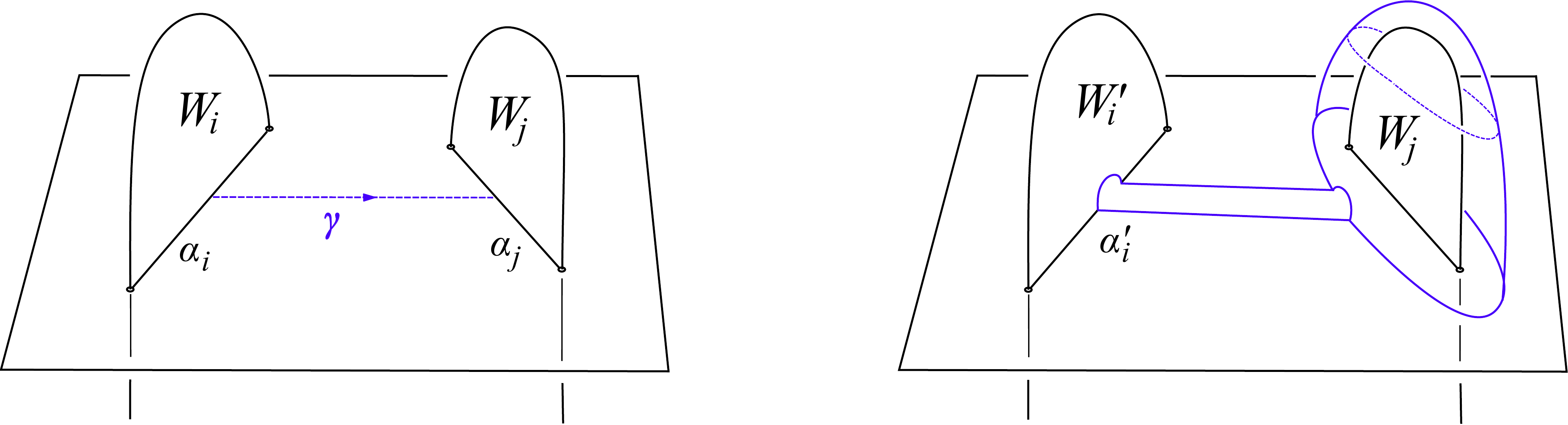}}
        \caption{Left: A path $\gamma$ guiding a slide of $W_i$ over $W_j$. Right: The result $W_i'$ of sliding $W_i$ contains the (blue) Whitney bubble $B_{\alpha_j}$ over $\alpha_j$.}
        \label{fig:w-disk-slide-1}
\end{figure}

Let $W_i$ and $W_j$ be two Whitney disks on $f$, and let $\gamma$ be an embedded path in $f$ from $\alpha_i\subset \partial W_i$ to
$\alpha_j\subset \partial W_j$ such that the interior of $\gamma$ is disjoint from any self-intersection of $f$ or Whitney arcs on $f$.
Denote by $W_i'$ the result of boundary-band-summing $W_i$ into a Whitney bubble $B_{\alpha_j}$ over $\alpha_j$ by a half-tube
along $\gamma$ as in Figure~\ref{fig:w-disk-slide-1}. 
We say that $W_i'$ is the result of \emph{sliding $W_i$ over $W_j$}.

To see that $f_{\{W_i',W_j\}}$ is isotopic to $f_{\{W_i,W_j\}}$, just observe that $W_i'$ becomes isotopic to $W_i$ after doing the $W_j$-Whitney move.
To see this in the coordinates of Figure~\ref{fig:w-disk-slide-1}, note that doing the $W_j$-Whitney move would either replace a horizontal disk of $f$ inside $B_{\alpha_j}\subset W_i'$ by a smaller Whitney bubble over $\alpha_j$, or would leave the same horizontal disk free of intersections by adding a Whitney bubble over $\beta_j$ to the other sheet of $f$.
So $W'_i$ isotopes back to $W_i$ across the smaller bubble or the horizontal disk.

Either of $\alpha_i$ and $\beta_i$ can be slid over either of $\alpha_j$ or $\beta_j$, and the isotopy class of the results of Whitney moves will be preserved as long as $i\neq j$.  
This sliding operation can be iterated:
\begin{lem}\label{lem:w-disk-slide-isotopy}
If a collection $\cW'$ of Whitney disks on $f$ is the result of performing finitely many slides ($i\neq j$) on a collection $\cW$,
then $f_{\cW'}$ is isotopic to $f_\cW$.\hfill$\square$
\end{lem}

Regarding the $i=j$ case, one can indeed slide $W_i$ over itself using a band from $\alpha_i\subset\partial W_i$ to the boundary of a Whitney bubble $B_{\beta_i}$ over $\beta_i$, and the result will still be a clean Whitney disk.
We don't believe that such a {\em self-slide} will preserve the isotopy class of $f_\cW$ in general (as it does in Lemma~\ref{lem:w-disk-slide-isotopy}). 
However, it will follow from Lemma~\ref{lem:w-disk-self-slide} that this self-sliding does indeed preserve the isotopy class of the result of the Whitney move in our current setting where $f$ is a sphere with a geometric dual.

\subsection{Tubing into the dual sphere}\label{sec:tubing-into-G}
For $G$ a geometric dual to $f$, a transverse intersection point $r$ between $f$ and a surface $D$ can be eliminated by tubing $D$ into $G$. This is known as the \emph{Norman trick} \cite{Nor} and is the main reason why dual spheres are so useful.
Here ``tubing $D$ into $G$'' means taking an ambient connected sum of $D$ with a parallel copy $G'$ of $G$ via a tube (an annulus) of normal circles over an embedded arc in $f$ that joins $r$ with an intersection point between $f$ and $G'$, see Figure~\ref{fig:tube-into-G-1}. 
Note that in the case that $D=f$ this operation involves a choice of which local sheet of $r$ to connect into.
\begin{figure}[ht!]
        \centerline{\includegraphics[scale=.25]{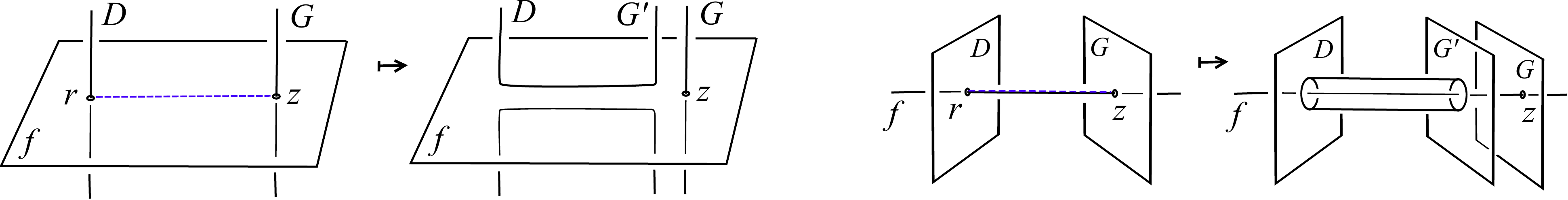}}
        \caption{Two views of the `tubing into $G$' operation to eliminate $r\in f\pitchfork D$, guided by a (blue dashed) path from $r$ to $z=f\cap G$.}
        \label{fig:tube-into-G-1}
\end{figure}

There are infinitely many pairwise disjoint copies of $G$ intersecting a small neighborhood around $z=f\cap G$ in $f$, so this procedure can be applied to eliminate any number of such intersections without creating new ones as long as appropriate guiding arcs for the tubes can be found. 
If a guiding arc intersects the boundary of a Whitney disk on $f$ then the corresponding tube around the arc will have an interior intersection with the Whitney disk, so we will always need to find guiding arcs that are disjoint from existing Whitney disk boundaries.

By varying the radii of the tubes, the guiding arcs can be allowed to intersect while keeping the tubes disjointly embedded.

\subsection{Clean collections of Whitney disks}\label{subsec:existence-of-disks}
Recall that for $f:S^2\imra M^4$, the vanishing of the self-intersection invariant
\[
\mu(f)=0\in\Z[\pi_1M]/\langle g-g^{-1} \rangle
\]
 is equivalent to the existence of choices of sheets so that all double points of $f$ can be arranged in pairs admitting null-homotopic Whitney circles (this statement is independent of the chosen whisker for $f$). 

\begin{defn}\label{def:clean}
A {\em clean collection of Whitney disks} for $f:S^2\imra M$ is a collection of Whitney disks that pair all double points of $f$ and are
framed, disjointly embedded, with interiors disjoint from $f$. In the presence of a dual sphere $G$ for $f$, this notion of a clean collection also includes the disjointness of the Whitney disks from $G$.
\end{defn}
Each Whitney disk in a clean collection is called a \emph{clean Whitney disk}. 

\begin{lem}\label{lem:choice-of-disks-exists}
If $f:S^2\imra M$ admits a geometric dual $G$, any collection of disjointly embedded Whitney circles that are null-homotopic in $M$ extends to a clean collection of Whitney disks.
\end{lem}

\begin{proof}
Start with a collection of generic disks $W_i$ bounded by the given null-homotopic Whitney circles that may intersect $G$, may be twisted, and may have interior intersections with $f$ and each other.

Note that the complement in $S^2$ of the union of the preimages $\{a_i,b_i\}$ of the Whitney circles is connected, and that there exist disjointly embedded tube-guiding paths in the complement of the Whitney circles between any number of isolated points and points near~$z$. 

We describe how to modify the $W_i$ relative their boundaries, without renaming them as changes are made: 

First of all, each $W_i$ can be made disjoint from $G$ by tubing $W_i$ into parallel copies of $f$ along disjoint arcs in $G$. 
Since $f$ is immersed with possibly non-trivial normal bundle, this tubing operation is in general more traumatic than the ``tubing into $G$'' operation described in Section~\ref{sec:tubing-into-G} and creates interior intersections between the $W_i$ and $f$, as well as intersections among the $W_i$.

Next, the intersections and self-intersections among the $W_i$ can be eliminated by pushing each such point down into $f$ by a finger move, and boundary-twists make the $W_i$ framed \cite[Chap.1.3]{FQ}, both at the cost of only creating more interior intersections between Whitney disks and $f$. 

Finally, the interiors of the $W_i$ can be made disjoint from $f$ by tubing the $W_i$ into $G$ along disjoint paths in $f$. Since $G$ is embedded and has trivial normal bundle the $W_i$ are still framed and disjoint from $G$, i.e.\ they form a clean collection of Whitney disks $\cW$ bounded by the original Whitney circles.
\end{proof}
\begin{rem}\label{rem:partial-collection}
The proof of Lemma~\ref{lem:choice-of-disks-exists} shows that if any subcollection of Whitney circles bound clean Whitney disks, then these same Whitney disks can be extended to a clean collection of Whitney disks by applying the construction to the remaining Whitney circles. 
\end{rem}

For any given collection $\cW$ of clean Whitney disks we
denote by $D_z\subset f(S^2)$ a small embedded disk around $z=f(S^2)\cap G$ such that each point in $D_z$ intersects a parallel of $G$ disjoint from $\cW$. 
The radius of $D_z$ is less than 
the minimum of the radii of the finitely many normal tubes around arcs in $G$
used in the first step of the proof of Lemma~\ref{lem:choice-of-disks-exists}, but our modifications of Whitney disk collections will only use the existence of $D_z$ not its diameter.

The minimum of the radii of the finitely many normal tubes around arcs in $f$
used in the last step of the proof of Lemma~\ref{lem:choice-of-disks-exists} gives a uniform lower bound on the distance between $f$ and the complements of small boundary collars of all Whitney disks in $\cW$.
Subsequent modifications of $\cW$ by tubing into $G$ along $f$ will always be assumed to use tubes of radius less than this bound, so as long as tubes are away from Whitney disk boundaries the tubes' interiors will be disjoint from Whitney disk interiors.

 \begin{figure}[ht!]
        \centerline{\includegraphics[scale=.28]{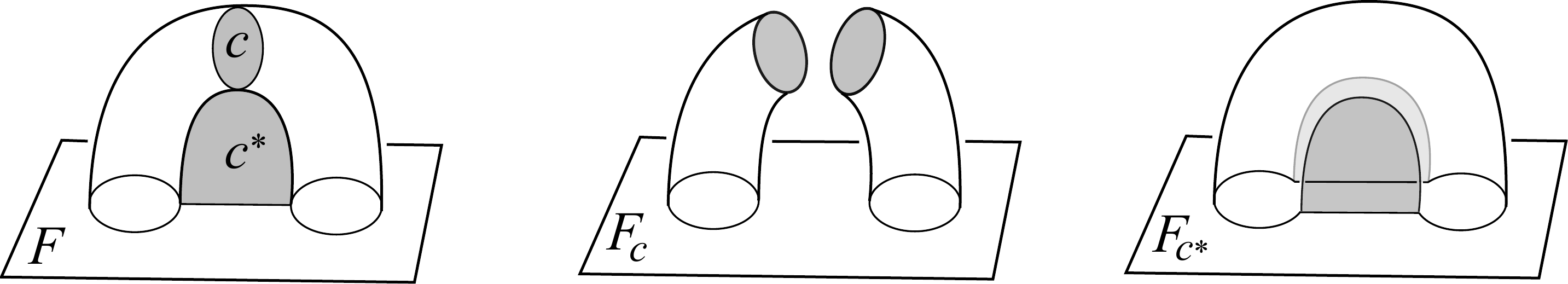}}
        \caption{A small neighborhood in $\R^3$ of $F\cup c\cup c^\ast$ on the left is diffeomorphic to $D^2 \times I$, in a way that $D^2 \times \{0\}\cong F_c$ and $D^2 \times \{1\} \cong F_{c^\ast}$. Hence the two surgeries in the center and right are isotopic.}
        \label{fig:capped-surface-surgery-1}
\end{figure}

\subsection{Capped surfaces and Whitney moves}\label{sec:capped-surface-w-move}
A \emph{cap} on a generic orientable surface $F$ in $M$ is a $0$-framed embedded disk $c$ such that the boundary $\partial c$ is the image of a non-separating simple closed curve in the domain of $F$, and the interior of $c$ is disjoint from $F$. Here ``$0$-framed'' means that a parallel copy $\partial c'\subset F$ of $\partial c$ bounds a cap $c'$ such that $c'\cap c=\emptyset$.  

Two caps on $F$ are \emph{dual} if their boundaries intersect in a single point and their interiors are disjoint. 
A collection $\cC$ of pairwise disjoint caps on $F$ is \emph{dual} to another collection $\cC^\ast$ of pairwise disjoint caps on $F$ if the interiors of all caps $\cC\cup \cC^\ast$ are pairwise disjoint, and the set of all cap boundaries is a geometric symplectic basis for the first homology of $F$ (the cap boundaries intersect geometrically $\delta_{ij}$ in $F$).

For a collection 
$\cC$ of disjoint caps on $F$, we denote by $F_\cC$ the result of surgering $F$ using all the caps of $\cC$. 

The following lemma can be proved by considering an isotopy of a standard model in $3$-space that passes through the symmetric surgery on both sets of caps (see Figure~\ref{fig:capped-surface-surgery-1} and \cite[Sec.2.3]{FQ}): 
\begin{lem}\label{lem:capped-surface-isotopy}
If $\cC$ and $\cC^\ast$ are dual collections of caps on $F$ then $F_\cC$ is isotopic to $F_{\cC^\ast}$ by
an isotopy supported near the union $F\cup\cC\cup\cC^\ast$.
\end{lem}
 
Lemma~\ref{lem:capped-surface-isotopy}, together with the presence of the geometric dual $G$, yields the following simple but useful correspondence between Whitney moves and surgeries:

Let $W$ be a clean Whitney disk on $f$ with $\partial W=\alpha\cup\beta$ (possibly one of a collection $\cW$ of Whitney disks on $f$), and let $F:T^2 \imra M$ be the result of tubing $f$ to itself along $\beta$.
Observe that a cap $c_W$ on $F$ can be constructed from $W$ by deleting a small boundary collar near $\beta$, and 
$F_{c_W}$ is isotopic to $f_W$ (Figure~\ref{fig:tube-caps-1}).
 \begin{figure}[ht!]
        \centerline{\includegraphics[scale=.31]{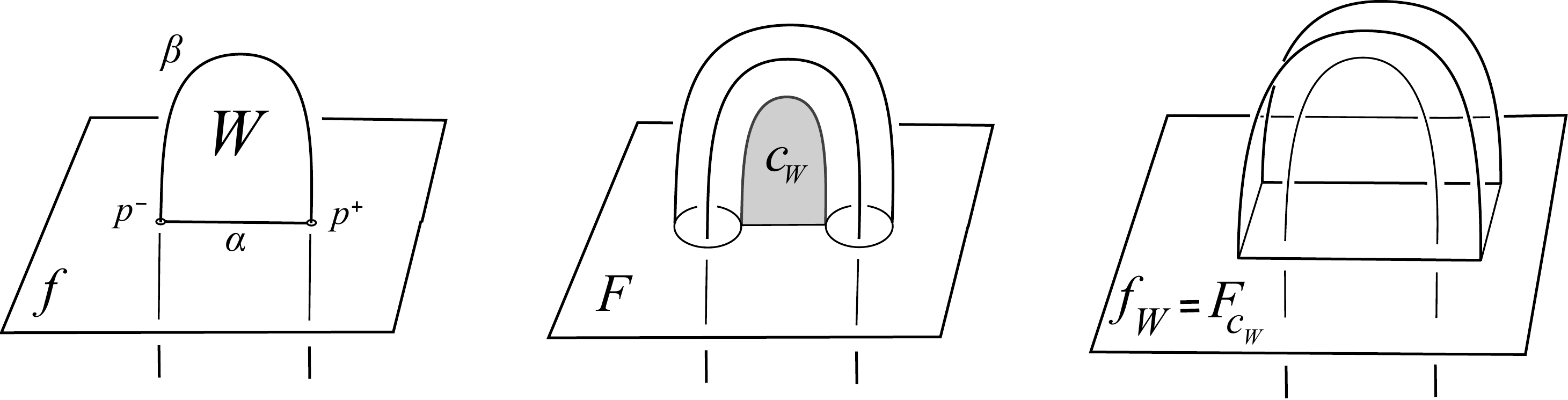}}
        \caption{}
        \label{fig:tube-caps-1}
\end{figure}
 \begin{figure}[ht!]
        \centerline{\includegraphics[scale=.31]{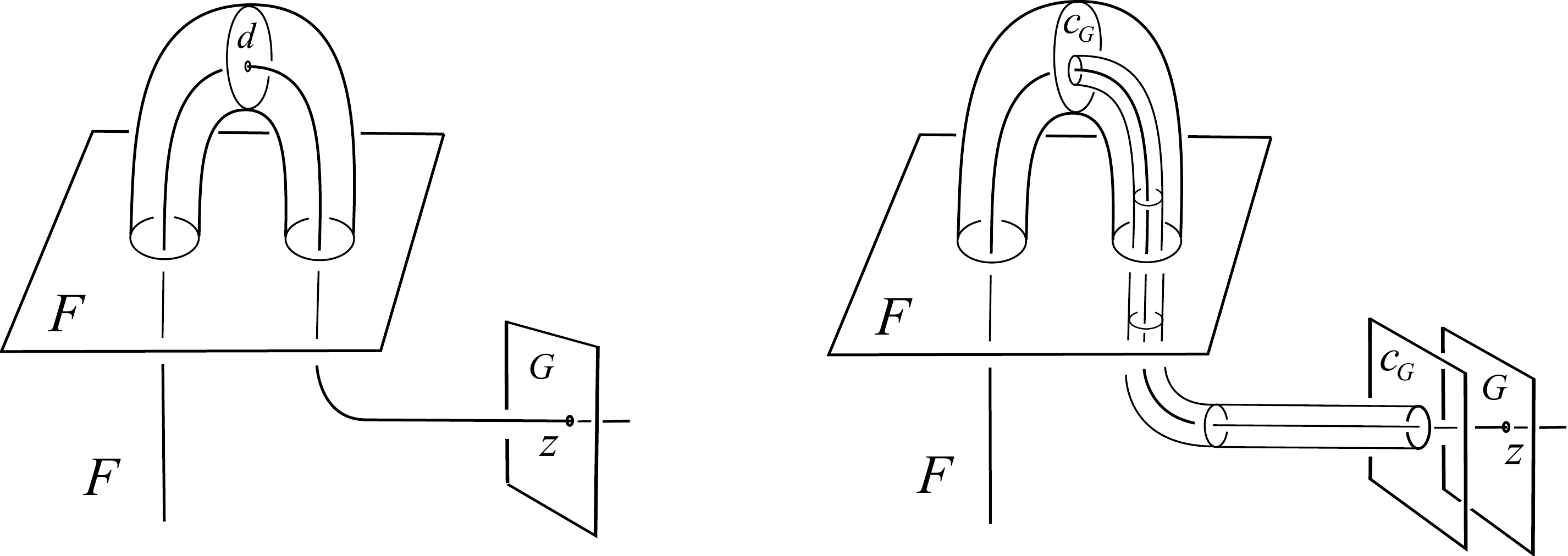}}
        \caption{}
        \label{fig:tube-caps-2}
       \end{figure}

Now we construct a cap $c_G$ on $F$ which is dual to $c_W$.
Start with a meridional disk $d$ to $F$ which has a single transverse intersection $r=d\pitchfork F\in\beta$ and $\partial d\subset F$ (Figure~\ref{fig:tube-caps-2} left). Note that $G$ is a geometric dual to $F$.  Then $c_G$ is the result of eliminating $r$ by tubing $d$ into $G$ along an embedded arc in $F$, disjoint from $c_W$ and $\partial d$ (and any other Whitney disks), running from $r$ to a point where a parallel copy of $G$ intersects $F$, see right Figure~\ref{fig:tube-caps-2}. Such an embedded arc exists since the complement of $\partial W$ is connected (as is the complement of $\partial \cW$). Since $c_G$ and $c_W$ are dual caps, Lemma~\ref{lem:capped-surface-isotopy} gives:

\begin{lem}\label{lem:W-move-equals-surgery-on-tube-cap}
If $F$ is the result of tubing $f$ to itself along one Whitney arc of a clean Whitney disk $W$, and $c_G$ is a cap on $F$ gotten by tubing a meridional disk dual to the Whitney arc into $G$ as above, then $f_W$ is isotopic to $F_{c_G}$.$\hfill\square$
\end{lem} 
So if two Whitney disks $W$ and $W'$ on $f$ have equal Whitney circles $\partial W=\partial W'$, then $f_W$ is isotopic to $f_{W'}$ since each is isotopic to surgery $F_{c_G}$ on a common dual cap $c_G$ to both of the caps $c_W$ and $c_{W'}$ as in Lemma~\ref{lem:W-move-equals-surgery-on-tube-cap}. And since the complement in $f$ of the Whitney circles of a clean collection of Whitney disks is connected we have:
\begin{lem}\label{lem:choice-of-disks}
If $\cW$ and $\cW'$ are clean collections of Whitney disks for the self-intersections of $f$ such that $\partial\cW=\partial\cW'$, then $f_\cW$ is isotopic to $f_{\cW'}$.$\hfill\square$
\end{lem}

\begin{lem}\label{lem:slide-w-disk-across-D-z}
For the Whitney circles $\cA=\partial\cW$ of a clean collection $\cW=\cup_iW_i$ of Whitney disks as in Definition~\ref{def:clean}, consider $\cA'$ which is the result of band summing a Whitney arc $\alpha_i\subset\partial W_i$ into a parallel of $\partial D_z$ along an arc $\gamma$ with interior disjoint from $\cA$ as in the left-most and right-most pictures in Figure~\ref{fig:slide-w-over-z}.
Then there exists a clean collection of Whitney disks $\cW'$ with $\partial\cW'=\cA'$ and such that $f_\cW$ is isotopic to $f_{\cW'}$.
\end{lem}

\begin{proof}
We break up the band sum operation into the three steps illustrated in Figure~\ref{fig:slide-w-over-z}:
Guided by $\gamma$, modify $\alpha_i$ by pushing a subarc slightly across $\partial D_z$, and extend this isotopy to a collar of $W_i$. 
The isotopy class of $f_\cW$ is unchanged since the collection $\cW$ changes by isotopy due to the disjointness of $\gamma$ from $\cA$.

Now delete from $\alpha_i$ the small (dashed) arc which is the intersection of $\alpha_i$ with the interior of $D_z$, and eliminate the oppositely signed self-intersections of $f$ that were paired by $W_i$ by tubing $f$ along the resulting pair of arcs into two oppositely oriented copies of $G$ which intersect $\partial D_z$ at the arcs' endpoints. See the second picture from the left in
Figure~\ref{fig:slide-w-over-z}.

\begin{figure}[ht!]
        \centerline{\includegraphics[scale=.25]{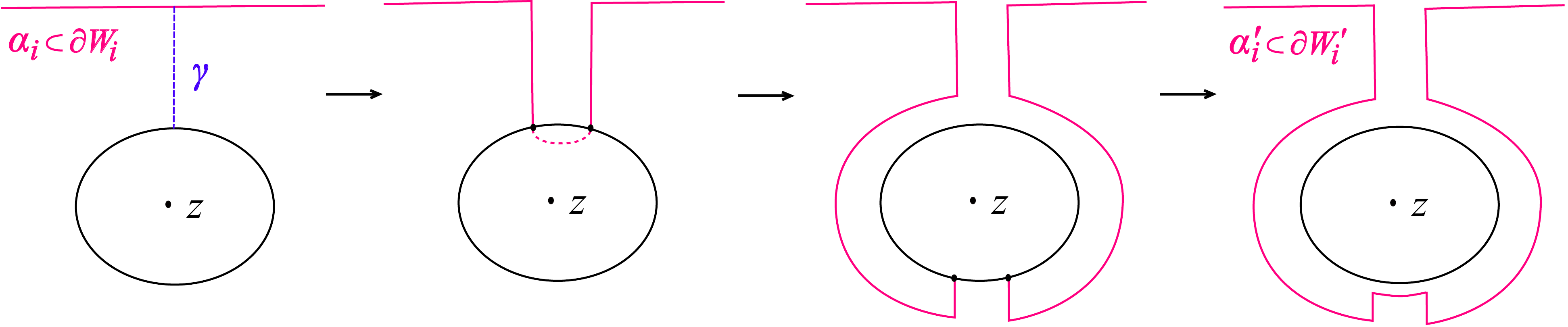}}
        \caption{}
        \label{fig:slide-w-over-z}
\end{figure} 

This yields an immersed sphere $f^G_\gamma$ which 
admits the clean collection of Whitney disks $\cV:=\cW \smallsetminus W_i$.
Note that by construction $f^G_\gamma$ is also the result of tubing $f$ to itself along the $\alpha_i$ that had been pushed into $D_z$ and then surgering the tube along a cap formed from a parallel copy of $G$
near where $\gamma$ meets $\partial D_z$. It follows from Lemma~\ref{lem:W-move-equals-surgery-on-tube-cap} that $f_{W_i}$ is isotopic to $f^G_\gamma$. 
Hence, $f_\cW$ is isotopic to $(f^G_\gamma)_{\cV}$.

Next, change $f^G_\gamma$ by an isotopy which moves the two tubes and the two parallels of $G$ contained in $f^G_\gamma$ in opposite directions around $\partial D_z$ as shown in the third picture from the left in
Figure~\ref{fig:slide-w-over-z}. 
Since we may assume that the two tubes have been chosen to have radii smaller than any previous tubes used to construct Whitney disks in $\cV$ this isotopy does not create any new intersections (see the second paragraph after Remark~\ref{rem:partial-collection}).
After this isotopy $f^G_\gamma$ still 
admits $\cV$, and the isotopy class of $(f^G_\gamma)_{\cV}$ is unchanged.

Now (re)connect the endpoints of the two guiding arcs of the tubes near the short subarc of $\partial D_z$ between the endpoints to get a single arc $\alpha'_i:=\alpha_i+_\gamma\partial D_z$ which is isotopic to the result of taking the band sum of $\alpha_i$ with $\partial D_z$ along $\gamma$ (see the right-most picture in
Figure~\ref{fig:slide-w-over-z}).
The resulting embedded Whitney circle $\alpha'_i\cup\beta_i$ is null-homotopic and disjoint from $\partial \cV$, so by Lemma~\ref{lem:choice-of-disks-exists} there exists a collection $\cW'$ of Whitney disks with $\partial \cW'=\alpha'_i\cup\beta_i\cup\partial \cV$.
As per Remark~\ref{rem:partial-collection}, the proof of Lemma~\ref{lem:choice-of-disks-exists} fixes $\cV$ while constructing a clean Whitney disk $W_i'$ bounded by $\alpha'_i\cup\beta_i$  
in the complement of $\cV$, so we have $\cW'=W'_i\cup\cV$.

It follows again by Lemma~\ref{lem:W-move-equals-surgery-on-tube-cap} that $f^G_\gamma$ is isotopic to $f_{W'_i}$, since $f^G_\gamma$ is isotopic to the result of tubing $f$ to itself along $\alpha'_i$ and then surgering a cap formed from a copy of $G$ near where the guiding arcs were reconnected.
Hence $f_{\cW'}$ is isotopic to $(f^G_\gamma)_{\cV}$, and we see that $f_\cW$ and $f_{\cW'}$ are isotopic. 
\end{proof}

\begin{figure}[ht!]
        \centerline{\includegraphics[scale=.2]{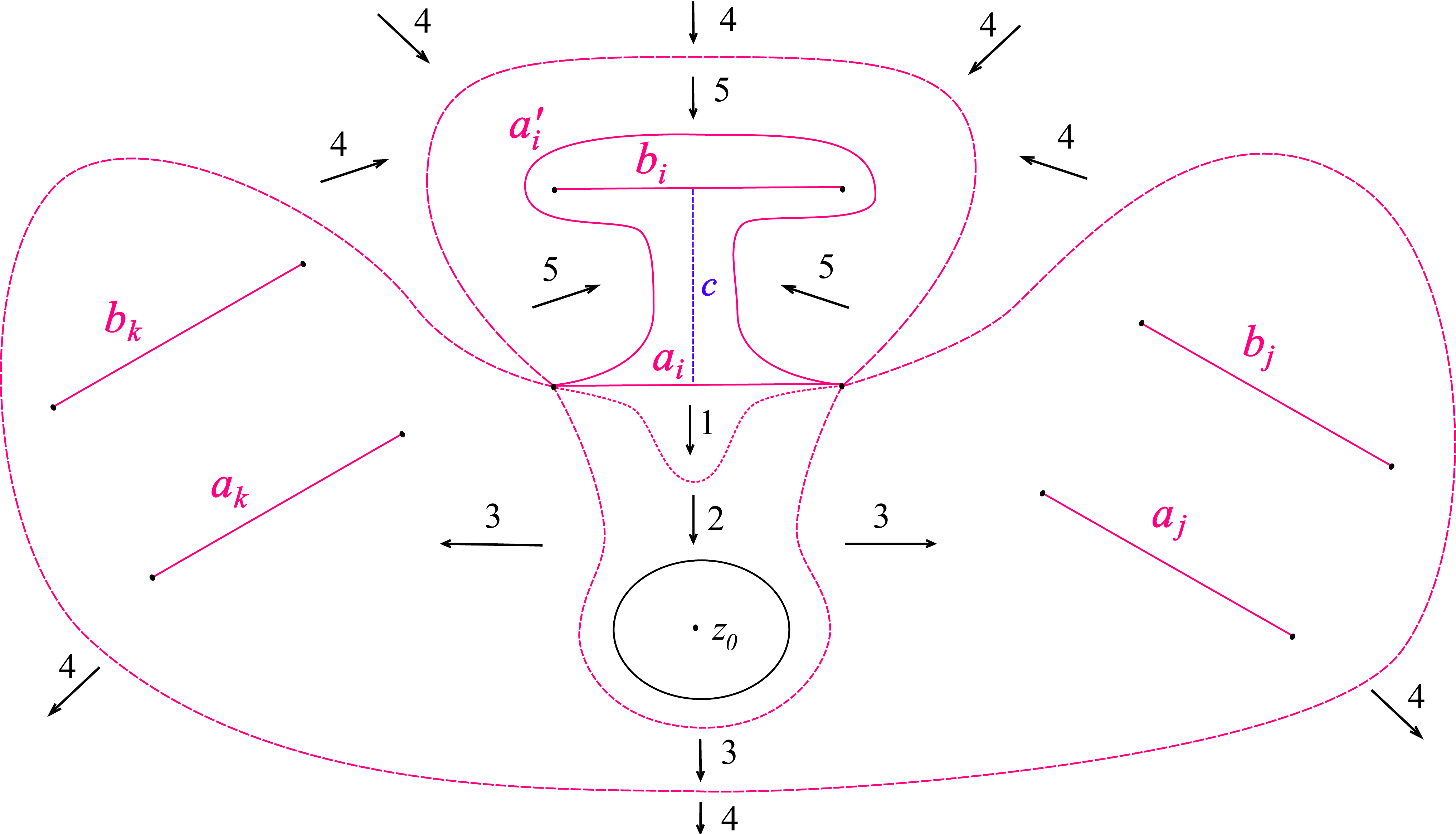}}
        \caption{}
        \label{fig:self-w-slide-dashed-arcs}
       \end{figure}
\begin{lem}\label{lem:w-disk-self-slide}
For a clean Whitney disk collection $\cW$ on $f:S^2\imra M^4$ with geometric dual $G$, if $\cW'$ is gotten from $\cW$ by sliding a Whitney disk over itself then $f_{\cW'}$ is isotopic to $f_\cW$.
\end{lem}

\begin{proof}
Let $\alpha_i$ be the Whitney arc of $\partial W_i=\alpha_i\cup\beta_i$ that is slid over $\beta_i$ to become $\alpha'_i\subset \partial W'_i=\alpha'_i\cup\beta_i$.
Referring to Figure~\ref{fig:self-w-slide-dashed-arcs}, consider the following five steps (indicated by the arrows in the figure) describing in the domain an isotopy of $\alpha_i=f(a_i)$ to $\alpha'_i=f(a'_i)$: 

Step~1 and Step~2 isotope $\alpha_i$ towards and then across $D_z=f(D_{z_0})$, as in Lemma~\ref{lem:slide-w-disk-across-D-z}. After these first two steps of the isotopy the union of the resulting new arc $\alpha^2_i$ with the original $\beta_i$ admits a clean Whitney disk $W^2_i$, and replacing $W_i$ by $W^2_i$ in $\cW$ yields a clean collection $\cW^2$ such that $f_{\cW^2}$ is isotopic to $f_\cW$ by Lemma~\ref{lem:slide-w-disk-across-D-z}. 

Step~3 then uses the Whitney disk sliding operation of Section~\ref{sec:w-disk-slides} to 
push
$\alpha^2_i$ across all the $\alpha_j$ and $\beta_j$ Whitney arcs of the Whitney disks $W_j$ for $j\neq i$ by sliding $W^2_i$ twice over each of these Whitney disks (once each for $\alpha_j$ and $\beta_j$). Taking the resulting Whitney disk $W_i^3$ as a replacement for $W^2_i$ in $\cW^2$ yields $\cW^3$, with $f_{\cW^3}$ isotopic to $f_{\cW^2}$ by Lemma~\ref{lem:w-disk-slide-isotopy}.

Finally, Steps~4 and 5 isotope a collar of $W_i^3$ around the $2$-sphere until the Whitney disk boundary arc ends up as the band sum $\alpha'_i$ of the original $\alpha_i$ with the boundary of a Whitney bubble over $\beta_i$. This 5-step construction yields $\cW^5$ with $W^5_i\in\cW^5$ having boundary $\alpha'_i\cup\beta_i$ and $f_\cW$ isotopic to $f_{\cW^5}$.
Now form $\cW'$ from $\cW^5$ by replacing the Whitney disk $W^5_i$ resulting from this construction with the Whitney disk $W'_i$ gotten by sliding $\alpha_i$ across $\beta_i$ which has the same boundary.
By Lemma~\ref{lem:choice-of-disks} we get that $f_\cW$ is isotopic to $f_{\cW'}$.  
\end{proof}
 
We come to our most useful geometric result for $f:S^2\imra M$ with geometric dual $G$:
\begin{prop}\label{prop:common-arcs-w-moves}
If $\cW$ and $\cW'$ are clean collections of Whitney disks on $f$ such that for each $i$, $W_i\in\cW$ and $W_i'\in\cW'$ share at least one common Whitney arc $\beta_i = \beta_i'$, then $f_\cW$ is isotopic to~$f_{\cW'}$.
\end{prop}

\begin{proof}
We first prove the simplest case of the statement: If $W$ and $W'$ are Whitney disks on $f$ which share a common Whitney arc $\beta=\beta'$, then $f_W$ is isotopic to $f_{W'}$.

The proof will proceed as in the setting of Lemma~\ref{lem:W-move-equals-surgery-on-tube-cap}, but because here we have two Whitney disks with possibly $\alpha\neq\alpha'$ we may need to apply the sliding maneuver of Section~\ref{sec:w-disk-slides} to create a tube-guiding arc to $z$ for cleaning up the meridional cap.

Let $F$ be the surface resulting from tubing $f$ to itself along the common Whitney arc $\beta=\beta'$ of $\partial W$ and $\partial W'$.  
Deleting small boundary collars of $W$ and $W'$ near $\beta$
yields caps $c_W$ and $c_{W'}$ for $F$ as in Figure~\ref{fig:tube-caps-1}, but with $\partial c_{W'}$ wandering off into the ``horizontal'' part of $F$ corresponding to $\alpha'\neq\alpha$.
By Lemma~\ref{lem:W-move-equals-surgery-on-tube-cap}, $F_{c_W}$ is isotopic to $f_W$, and $F_{c_{W'}}$ is isotopic to $f_{W'}$.

As in the setting of Lemma~\ref{lem:W-move-equals-surgery-on-tube-cap}, we want to construct a cap $c_G$ for $F$ such that $c_G$ is dual to both $c_W$ and $c_{W'}$. Then by Lemma~\ref{lem:W-move-equals-surgery-on-tube-cap} it will follow that each of $f_W$ and $f_{W'}$ is isotopic to $F_{c_G}$.

The construction of $c_G$ starts as in Figure~\ref{fig:tube-caps-2}: We want to clean up a meridional disk $d$ to $F$ which has a single transverse intersection $r=d\pitchfork F\in\beta$ and $\partial d\subset F$ by tubing $d$ into $G$.  
But now we have to find an embedded path from $r$ to $z=G\cap F$ that is disjoint from \emph{both} $\partial c_W$ and $\partial c_{W'}$.

If $r$ and $z$ lie in the same connected component of $F\smallsetminus(\partial c_W\cup \partial c_{W'})$ then there is no problem. We can eliminate $r$ by tubing $d$ into $G$ along an embedded path in $F$ running from $r$ to a point near $z$ where a parallel copy of $G$ intersects $F$, and the resulting cap $c_G$ for $F$ is dual to both $c_W$ and $c_{W'}$.  
\begin{figure}[ht!]
        \centerline{\includegraphics[scale=.19]{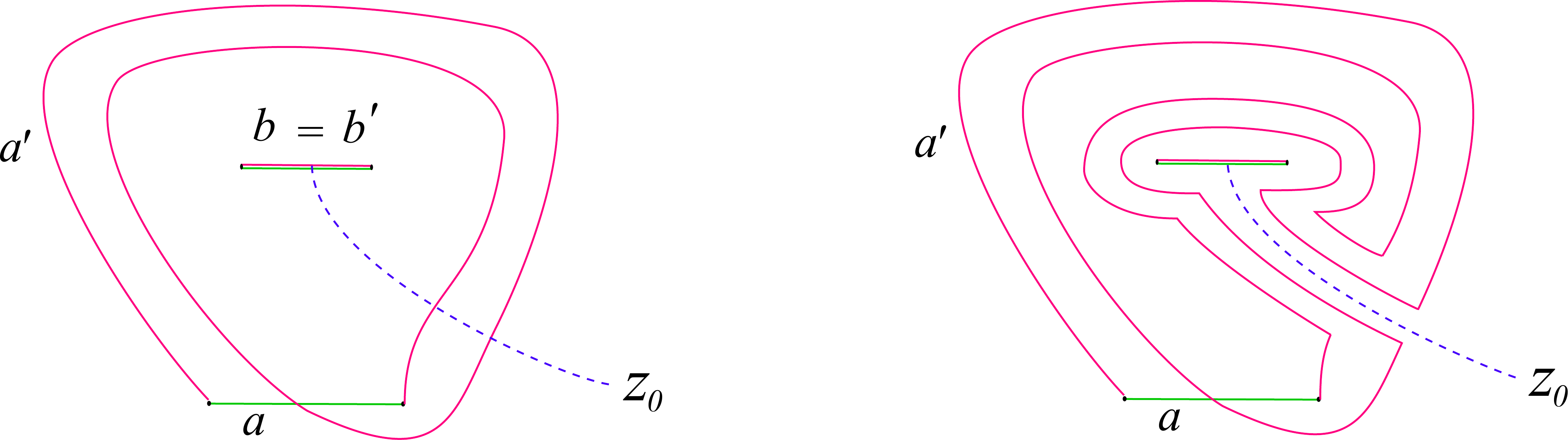}}
        \caption{In the domain $S^2$: The case of one pair of Whitney disks, with $f(z_0)=z=f\cap G$. Slides are done in the order starting closest to $f(b)=\beta=\beta'=f(b')$. Here $a$ and $a'$ are the preimages of $\alpha$ and $\alpha'$, respectively.}
        \label{fig:alpha-slides}
\end{figure}

Now consider the case that $r$ and $z=G\cap F$ do not lie in the same connected component of $F\smallsetminus(\partial c_W\cup \partial c_{W'})$, and observe that
this means that $\beta$ and $z$ do not lie in the same component of the complement in $f$ of the immersed loop $\alpha\cup \alpha'$ (see the left side of Figure~\ref{fig:alpha-slides} for the preimage).
In this case we can modify the original Whitney disk $W'$ before constructing $F$ using the sliding maneuver of Section~\ref{sec:w-disk-slides} to arrange that $\beta$ and $z$ do lie in the same component of $f\smallsetminus(\alpha\cup \alpha')$: 

Since $f \smallsetminus \partial W$ is connected, there is an embedded arc $\gamma$ from $z$ to $r\in\beta'=\beta$ such that $\gamma$ is disjoint from $\alpha$ (the preimage of $\gamma$ is the dashed blue arc in Figure~\ref{fig:alpha-slides}). Eliminate the intersections between $\gamma$ and $\alpha'$ by sliding $W'$ over itself from $\alpha'$ to $\beta'$ guided by $\gamma$ as in Section~\ref{sec:w-disk-slides} (right side of Figure~\ref{fig:alpha-slides}). By Lemma~\ref{lem:w-disk-self-slide} this does not change the isotopy class of $f_{W'}$, and now the construction of the cap $c_G$ for $F$ goes through as desired. 

\begin{figure}[ht!]
        \centerline{\includegraphics[scale=.2]{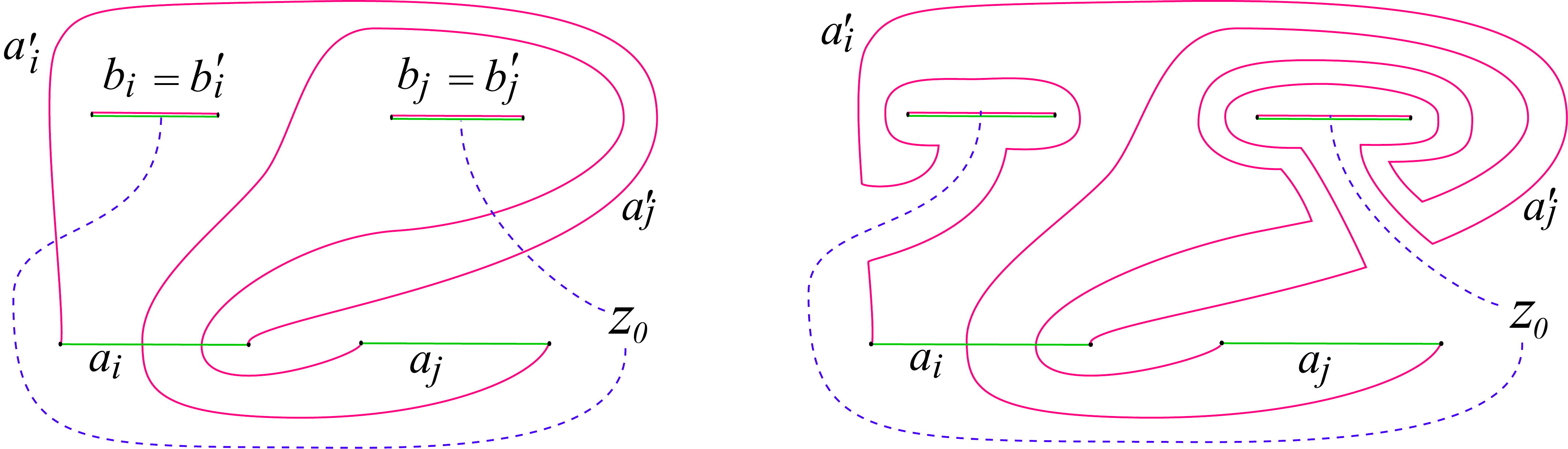}}
        \caption{In the domain $S^2$: The general multiple Whitney disk case of Figure~\ref{fig:alpha-slides}. Again the arcs $\gamma_i$ (dashed blue preimages) lead to our slides of the $\alpha'_i$ (preimages $a'_i$) algorithmically, starting closest to the $\beta_i$-arcs (preimages $b_i=b_i'$).}
        \label{fig:multiple-alpha-slides}
\end{figure}
For the general statement, apply the same construction to each of the pairs of Whitney disks $W_i$ and $W_i'$ in $\cW$ and $\cW'$.
Start with disjointly embedded  arcs $\gamma_i$ in $f \smallsetminus \partial \cW$ from the common arcs $\beta_i$ to $z$. The only new complication is that making these arcs disjoint from $\partial\cW'$
may involve more Whitney disk slides as shown in Figure~\ref{fig:multiple-alpha-slides}. By Lemma~\ref{lem:w-disk-slide-isotopy} and Lemma~\ref{lem:w-disk-self-slide} these sides preserve the isotopy class of $f_{\cW'}$, and by applying Lemma~\ref{lem:W-move-equals-surgery-on-tube-cap} to each pair $W_i, W'_i$ we have that $f_\cW$ is isotopic to $f_{\cW'}$.
\end{proof}

 \section{New Proof of Gabai's LBT}\label{sec:proof-of-theorem}
Let $M$ be a smooth orientable $4$--manifold and  $f:S^2\imra M$ a generic smooth map with 
$0=\mu(f)\in\Z[\pi_1M]/\langle g-g^{-1} \rangle$
and with geometric dual $G$. 
Recall that $ \cR^G_{[f]}$ denotes the set of isotopy classes of embedded spheres which are homotopic to $f$ and have $G$ as a geometric dual.

\textbf{Outline of our proof of Gabai's LBT:} We will show that $ \cR^G_{[f]}$ contains a unique element if $\pi_1M$ does not contain 2-torsion. As explained in Section~\ref{subsec:regular-homotopy}, any two embedded spheres in $\cR^G_{[f]}$ are related via a finite sequence of isotopies, finger moves and Whitney moves, all away from $G$. By general position it can be arranged that the finger moves occur before the Whitney moves (see eg.~\cite[Lem.8]{BT}). Denoting the result of the finger moves by $f$, we will consider all possible collections of Whitney disks on $f$ in $M\smallsetminus G$ and show that all the resulting embeddings are isotopic.
As a first step, Section~\ref{sec:choices} describes precisely the various types of choices involved in constructing a collection $\cW$ of clean Whitney disks on $f$ such that the result $f_\cW$ of doing the Whitney moves in $\cW$ on $f$ is an embedding. 
In Sections~\ref{subsec:existence-and-choices-of-disks}--\ref{sec:choice-of-sheets-trivial-element} 
we prove that the isotopy class of $f_\cW$ does not depend on any of these choices.

\subsection{Choices of sheets, pairings, W-arcs and W-disks}\label{sec:choices}
We'll discuss the four types of choices $\textsf{C}_{\textrm{sheets}}$, $\textsf{C}_{\textrm{pairings}}$, $\textsf{C}_{\textrm{W-arcs}}$ and 
$\textsf{C}_{\textrm{W-disks}}$ that determine a clean collection $\cW$ of Whitney disks on $f:S^2\imra M$ and hence a generic homotopy from $f$ to an embedding $f_\cW$ (with geometric dual $G$). In the following, each step will depend on having made all previous choices. Moreover, each later choice lets us reconstruct the previous choices. 

Denote the set of transverse self-intersections of $f$ by $\{p_1,\dots,p_{2n}\}\subset f(S^2)$, where the ordering of the $p_i$ is an artifact of the notation and will never be used; and fix a whisker for $f$ from the basepoint of $M$. 
(The condition $0=\mu(f)\in\Z[\pi_1M]/\langle g-g^{-1} \rangle$ implies that $f$ has an even number of self-intersections.)

\begin{enumerate}
\item[$\textsf{C}_{\textrm{sheets}}$:] \label{item:choice-of-sheets}
A \emph{choice of sheets} $\{x_1,\dots, x_{2n}\}\in\textsf{C}_{\textrm{sheets}}$ consists of choices $x_i\in f^{-1}(p_i)\subset S^2$, subject to the following requirement:
By Section~\ref{sec:clean-w-disks-moves}, each $x_i$ orients a double point loop at $p_i$ by the convention that the loop is the image of a path {\it starting from} $x_i$. Via the whisker for $f$ we get a well-defined group element $g(x_i)\in\pi_1M$.

Then our choice of sheets is required to satisfy
\begin{equation}\tag{$*$}
0=\sum_{i=1}^{2n}\,\epsilon_i\cdot g(x_i)\in\Z[\pi_1M],   \text{ where $\epsilon_i\in\{\pm 1\}$ is the sign of $p_i$.}
\end{equation}
A different choice of whisker for $f$ would change each $g(x_i)$ to a conjugate $g(x_i)^h$ for some fixed $h\in \pi_1M$, hence our requirement $(*)$ is independent of the whisker. Moreover, switching the preimage choice $x_i$ at $p_i$ has the effect of inverting the group element $g(x_i)$, 
so choices of sheets exist since $0=\mu(f)\in\Z[\pi_1M]/\langle g-g^{-1} \rangle $.

\item[$\textsf{C}_{\textrm{pairings}}$:] \label{item:choice-of-pairings}
For $\{x_1,\dots, x_{2n}\}\in\textsf{C}_{\textrm{sheets}}$, a \emph{compatible choice of pairings} $\{x_1^\pm,\dots, x_n^\pm\}\in\textsf{C}_{\textrm{pairings}}$ consists of $n$ distinct pairs $x_i^\pm:=(x^+_i,x^-_i)=(x_{j_i}, x_{k_i})$  with $\epsilon_{j_i}=+1=-\epsilon_{k_i}$ and $g(x_{j_i})=g(x_{k_i})$.
A choice of pairings exists by our requirement $(*)$ on $\{x_1,\dots, x_{2n}\}$ and it induces pairings $(p_i^+, p_i^-)$ of the self-intersections of $f$.

\item[$\textsf{C}_{\textrm{W-arcs}}$:] \label{item:choice-of-arcs}
For $\{x_1^\pm,\dots, x_n^\pm\}\in\textsf{C}_{\textrm{pairings}}$, a \emph{compatible choice of Whitney arcs} 
$\{\alpha_1,\beta_1, \dots, \alpha_n, \beta_n\}\in\textsf{C}_{\textrm{W-arcs}}$ are the images under $f$ of disjointly embedded arcs $a_i \subset S^2$ joining $x^+_i$ and $x^-_i$, and arcs $b_i \subset S^2$  joining $y^+_i$ and $y^-_i$ for $i=1,\dots, n$, where $f^{-1}(p^\pm_k) = \{x^\pm_k, y^\pm_k\}$. 
Here $\alpha_i:= f(a_i)$ and $\beta_i:=f(b_i)$ are disjoint, except that $\partial \alpha_i = \{p_i^+, p_i^-\} = \partial \beta_i$. Note that $\alpha_i\subset f(S^2)$ determines $a_i\subset S^2$ and hence the original choice of pairings is determined by $\{\alpha_1,\dots,\alpha_n\}$ alone. A choice of Whitney arcs always exists since the complement in $S^2$ of finitely many disjointly embedded arcs and points is connected.

\item[$\textsf{C}_{\textrm{W-disks}}$:] \label{item:choice-of-disks}
Given a choice of Whitney arcs $\{\alpha_1,\beta_1, \dots, \alpha_n, \beta_n\}\in\textsf{C}_{\textrm{W-arcs}}$, a \emph{compatible choice of Whitney disks} $\{W_1,\dots, W_n\}\in \textsf{C}_{\textrm{W-disks}}$ is a clean collection of Whitney disks $W_i$ whose boundaries are equal to the circles $\alpha_i\cup \beta_i\subset M$. Recall that \emph{clean} means the $W_i$ are framed, disjointly embedded, have interiors disjoint from $f$, and are disjoint from $G$.
The existence of a choice of Whitney disks for any choice of Whitney arcs follows from Lemma~\ref{lem:choice-of-disks-exists}. To reconstruct $\alpha_i$ from $W_i:D^2\hra M$, we also require that $\alpha_i = W_i(S^1_-)$, where $S^1_-\subset S^1=\partial D^2\subset D^2\subset \R^2$ is the lower semi-circle. 
\end{enumerate}
In the following, we will  abbreviate our choices by 
\[
\sx:=\{x_1,\dots, x_{2n}\}, \,\,\sx^\pm:=\{x^\pm_1,\dots, x^\pm_{n}\},\,\,  \cA:=\{\alpha_1,\beta_1, \dots, \alpha_n, \beta_n\}  \text{ and }  \cW:=\{W_1,\dots, W_{n}\}.
\]
The meaning of $\partial \cW = \cA$ should be clear from our conventions.
The embedded sphere obtained from $f$ by doing Whitney moves guided by the Whitney disks in $\cW$ is denoted $f_\cW$.

\subsection{Existence and choices of Whitney disks}\label{subsec:existence-and-choices-of-disks}
For future reference we observe here that the existence of a compatible $\cW\in\textsf{C}_{\textrm{W-disks}}$ for any given $\cA\in\textsf{C}_{\textrm{W-arcs}}$ guaranteed by Lemma~\ref{lem:choice-of-disks-exists}, together with the definitions of pairing choices and sheet choices in Section~\ref{sec:choices}, imply the following:
\begin{lem}\label{lem:disks-exist-for-all-choices}
Given  $\sx^\pm\in\textsf{C}_{\textrm{pairings}}$, there exists $\cW\in\textsf{C}_{\textrm{W-disks}}$ compatible with $\sx^\pm$.

Given $\sx\in\textsf{C}_{\textrm{sheets}}$, there exists $\cW\in\textsf{C}_{\textrm{W-disks}}$ compatible with $\sx$.
$\hfill\square$
\end{lem}

From Lemma~\ref{lem:choice-of-disks}, the isotopy class of $f_\cW$ is independent of the interiors of the Whitney disks in $\cW$, i.e.\ $f_\cW$ only depends on $\cA$.

We next introduce Norman spheres, which will play a key role in showing that the isotopy class of $f_\cW$ is also independent of choices of arcs and pairings for any given sheet choice.

\subsection{Norman spheres}\label{sec:norman}

Fix a choice of sheets $\sx=\{x_1,x_2,\ldots,x_{2n}\}\in\textsf{C}_{\textrm{sheets}}$ for $f$. We need yet another type of choice to define a Norman sphere (whose isotopy class will ultimately only depend on $\sx$). Recall that $D_z\subset f$ denotes a small disk around $z=f\cap G$ such that each point in $D_z$ intersects a parallel of $G$ which is geometrically dual to $f$.

\begin{enumerate}
\item[$\textsf{C}_{\textrm{N-arcs}}$:] \label{item:choice-of-N-arcs}
A {\it compatible choice of Norman arcs} $\cZ:=\{\sigma_1,\dots,\sigma_{2n}\} \in \textsf{C}_{\textrm{N-arcs}}$ for $\sx\in\textsf{C}_{\textrm{sheets}}$ is the image under $f$ of disjointly embedded arcs $s_i\subset S^2$ starting at $x_i$ and ending in $f^{-1}(\partial D_z)$. Then $\sigma_i:=f(s_i)\subset f(S^2)$ are disjointly embedded arcs starting at $p_i$ and ending in $\partial D_z$; they determine the arcs $s_i$ uniquely.
\end{enumerate}

\begin{defn} \label{def:Norman}
The \emph{Norman sphere} $f^G_\cZ :S^2\hra M$ is obtained from $f$, $G$ and $\cZ$ by
eliminating all the self-intersections $p_i\in f\pitchfork f$ by tubing $f$ into parallel copies of $G$ along the $\sigma_i$. 
Precisely, these tubing operations replace the image of a small disk around each $y_i\in S^2$ by a normal tube along $\sigma_i$ together with a parallel copy $G_i$ of $G$ with a small normal disk to $f$ removed at $G_i\cap f$. 
Here $f^{-1}(p_i)=\{x_i,y_i\}$ with $x_i\in\sx$, and the $y_i$-sheet of $f$ at $p_i$ is deleted by the tubing operation since the $y_i$-sheet is normal to $\sigma_i$ at $p_i$.
\end{defn}

By construction, the Norman sphere $f^G_\cZ$ is embedded and has $G$ as a geometric dual.
Also, $f^G_\cZ$ is homotopic to $f$ since the copies of $G$ in the connected sum with $f$ come in oppositely oriented pairs having the same group element by our requirement $(*)$ in Section~\ref{sec:choices} on the sheet choice $\sx$.
Hence $f^G_\cZ\in\cR^G_{[f]}$.

Surprisingly, we will show in Lemma~\ref{lem:Norman-independence-of-arcs} that the isotopy class of $f^G_\cZ $ only depends on $\sx$ and not at all on $\cZ$.

We remark that the $\sigma_i$ are as in \cite{Gab} which are the simplest of the three types of arcs used by Gabai. The $\sigma_i$ in \cite{Gab} are allowed to intersect but here we require them to be disjointly embedded. 

\begin{lem}\label{lem:Norman-for-every-Whitney}
For any given choice of sheets $\sx$, if $\cW$ is an $\sx$-compatible choice of Whitney disks then there is an $\sx$-compatible choice of Norman arcs $\cZ$ 
such that $f^G_\cZ $ is isotopic to $f_\cW$. 
\end{lem}

\begin{proof}
We apply the first step in the proof of Lemma~\ref{lem:slide-w-disk-across-D-z} simultaneously to all $\alpha_i$:
Let $\cA:=\partial\cW$ be the Whitney arcs, and let $\sx^\pm$ be the choice of pairings determined by $\cA$. To construct the Norman arcs $\cZ$, isotope the Whitney arcs $\alpha_i$ just across $\partial D_z$ and extend this isotopy to an isotopy of $W_i$ in a collar on $\alpha_i$; see Figure~\ref{fig:Norman-arcs-1} where $D_{z_0}:=f^{-1}(D_z)$. This can be done keeping the $\alpha_i$ disjoint from each other and from all $\beta_j$. Deleting the part of the new $\alpha_i$ that lies in the interior of $D_z$ gives two arcs $\sigma_i^\pm$ which start at $x_i^\pm$ and end in $\partial D_z$.
\begin{figure}[ht!]
        \centerline{\includegraphics[scale=.175]{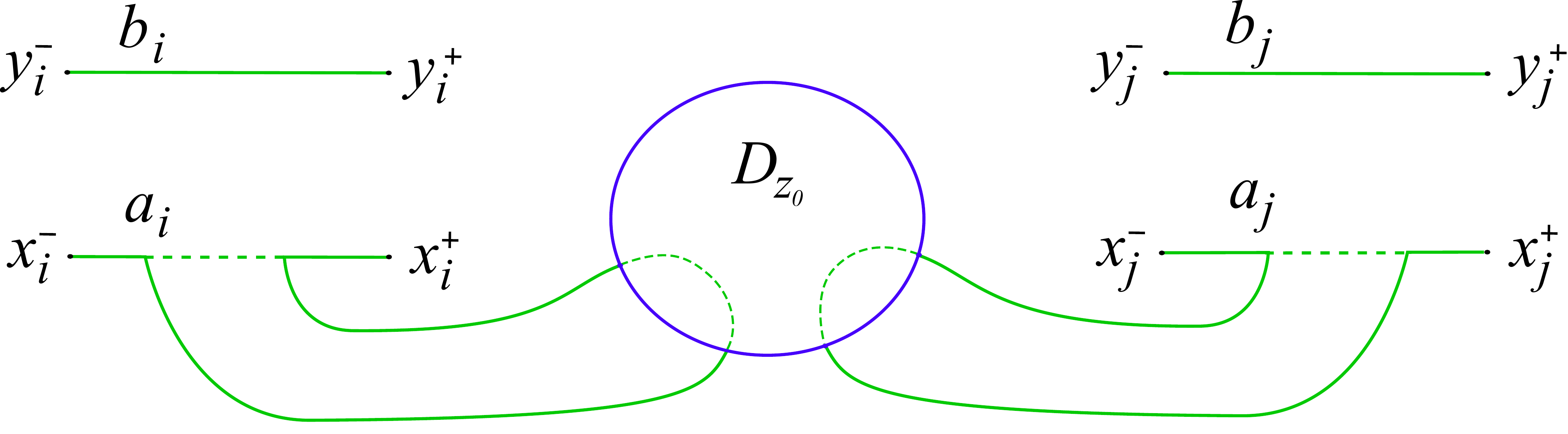}}
        \caption{The preimages $a_i$ and $a_j$ of arcs $\alpha_i$ and $\alpha_j$ after the isotopy.}
        \label{fig:Norman-arcs-1}
\end{figure} 
 
Define $\cZ:=\{\sigma_1^-,\sigma_1^+,\dots, \sigma_n^-,\sigma_n^+\}$ and observe that since the corresponding copies $G_i^\pm$ of $G$ are oppositely oriented, the Norman sphere $f^G_\cZ $ is isotopic to the result of tubing $f$ to itself along each $\alpha_i\subset\partial W_i$, then surgering a meridional cap dual to $\alpha_i$ that has been tubed into $G$ as in Figure~\ref{fig:tube-caps-2}. So $f^G_\cZ $ is isotopic to $f_\cW$ by Lemma~\ref{lem:W-move-equals-surgery-on-tube-cap}.
\end{proof}

In the proofs of the next two lemmas we describe isotopies of Norman spheres using homotopies of Norman arcs by requiring that the radii of the tubes are not equal at any temporarily-created intersection between Norman arcs during a homotopy. Following Gabai, we indicate the tube of smaller radius as an under-crossing of the corresponding Norman arc.

\begin{lem}[Lemma~5.11(ii) of \cite{Gab}]\label{lem:Norman-cyclic-ordering}
Given any $\cZ'\in\textsf{C}_{\textrm{N-arcs}}$ and points $z_1,\dots,z_{2n} \in \partial D_z$, there is a choice of Norman arcs $\cZ=\{\sigma_1,\dots,\sigma_{2n}\}$, compatible with the same $\sx\in\textsf{C}_{\textrm{sheets}}$ as $\cZ'$, such that $\sigma_i$ ends in $z_i$ and the Norman spheres $f^G_{\cZ'} $ and $f^G_\cZ $ are isotopic. 
\end{lem}
\begin{proof}
It suffices to observe that neighboring $z_i$ and $z_j$ in $\partial D_z$ can be exchanged by pushing the tube around $\sigma_j$ across (and inside) the tube around $\sigma_i$, as in Figure~\ref{fig:Norman-arcs-re-order-1} and Figure~\ref{fig:Norman-arcs-re-order-2}.
\end{proof}
\begin{figure}[ht!]
        \centerline{\includegraphics[scale=.15]{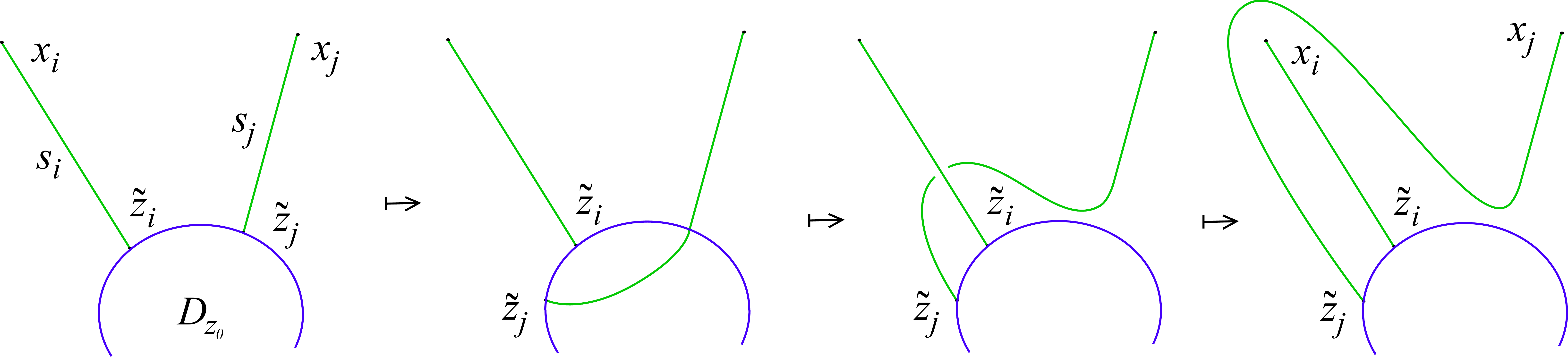}}
        \caption{The indicated homotopy of $s_i$ and $s_j$ corresponds to an isotopy of Norman spheres which slides the tube around $\sigma_j$ inside of the tube around $\sigma_i$. See Figure~\ref{fig:Norman-arcs-re-order-2}.}
        \label{fig:Norman-arcs-re-order-1}
\end{figure}
\begin{figure}[ht!]
        \centerline{\includegraphics[scale=.52]{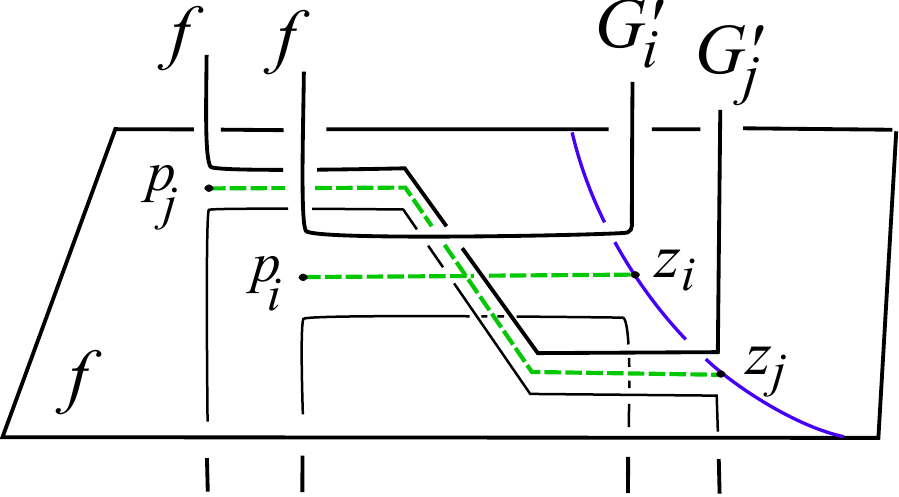}}
        \caption{The image of the third-from-left picture in Figure~\ref{fig:Norman-arcs-re-order-1}. Here the smaller radius of the tube around $\sigma_j$ compared to the tube around $\alpha_i$ corresponds to $s_j$ crossing under $s_i$ in Figure~\ref{fig:Norman-arcs-re-order-1}.}
        \label{fig:Norman-arcs-re-order-2}
\end{figure}

\begin{lem}\label{lem:Norman-independence-of-arcs}
If two choices of Norman arcs $\cZ, \cZ'\in\textsf{C}_{\textrm{N-arcs}}$ are compatible with the same $\sx\in\textsf{C}_{\textrm{sheets}}$ then the Norman spheres $f^G_\cZ $ and $f^G_{\cZ'} $ are isotopic. 
\end{lem}
As a consequence, we get a Norman sphere $f^G_\cZ =: f^G_\sx\in  \cR^G_{[f]}$ for a given choice of sheets $\sx$.
\begin{proof}
Let $\sx^\pm$ be any compatible choice of pairings for $\sx$. By Lemma~\ref{lem:Norman-cyclic-ordering} 
we may assume that $\cZ=\{\sigma_1^-,\sigma_1^+,\dots, \sigma_n^-,\sigma_n^+\}$ induces the cyclic ordering $(z_1^-,z_1^+,z_2^-,z_2^+,\ldots,z_n^-,z_n^+)$ in $\partial D_z$, where $z_i^\pm$ is the end-point of $\sigma^\pm_i$. 

We will first construct a choice of Whitney disks $\cW$ for $f$ such that $f^G_\cZ $ is isotopic to $f_\cW$, by performing essentially the inverse of the steps in the proof of Lemma~\ref{lem:Norman-for-every-Whitney}. 
For each $i$, denote by $\alpha_i$ the union of the embedded arcs $\sigma_i^-$ and $\sigma_i^+$ together with a short arc in $\partial D_z$ that runs between $z_i^+$ and $z_i^-$. 
These $\alpha_i$ then form one half of a collection of Whitney arcs for the choice of pairings $\sx^\pm$.

Choose appropriate arcs $\beta_i$ to complete the half collection $\alpha_i$ to a $\sx^\pm$-compatible choice of Whitney arcs $\cA=\{\alpha_1,\beta_1,\dots,\alpha_n,\beta_n\}$. By Lemma~\ref{lem:choice-of-disks-exists} there exists a collection $\cW\in\textsf{C}_{\textrm{W-disks}}$ with boundary $\cA$. 

It follows that $f^G_\cZ $ is isotopic to $f_\cW$ by Lemma~\ref{lem:W-move-equals-surgery-on-tube-cap}, since $f^G_\cZ $ is isotopic to the result of surgering the capped surface formed by tubing $f$ along the $\alpha_i$ arcs, as observed in the proof of Lemma~\ref{lem:Norman-for-every-Whitney}.
 
By Lemma~\ref{lem:Norman-cyclic-ordering} 
we may assume that $\cZ'=\{{\sigma'_1}^-,{\sigma'_1}^+,\dots, {\sigma'_n}^-,{\sigma'_n}^+\}$ induces the same cyclically ordered points $(z_1^-,z_1^+,z_2^-,z_2^+,\ldots,z_n^-,z_n^+)$ in $\partial D_z$ as $\cZ$, with $z_i^\pm$ the end-point of ${\sigma'_i}^\pm$.  

For each $i$, denote by $\alpha'_i$ the union of the embedded arcs ${\sigma'_i}^-$ and ${\sigma'_i}^+$ together with a short arc in $\partial D_z$ that runs between $z_i^+$ and $z_i^-$. These $\alpha'_i$ form a half collection of Whitney arcs for the choice of pairings $\sx^\pm$.

Now pause to observe that if each $\alpha'_i$ is disjoint from all the previously chosen $\beta_j$, then the unions $\alpha'_i\cup\beta_i$ are Whitney circles for a clean collection $\cW'$ of Whitney disks on $f$ by Lemma~\ref{lem:choice-of-disks-exists}, and $f^G_\cZ $ is isotopic to $f^G_{\cZ'}$.
This is because the collections $\cW$ and $\cW'$ share the common $\beta_i$-arcs so $f_\cW$ would be isotopic to $f_{\cW'}$ by Proposition~\ref{prop:common-arcs-w-moves}.
Then analogously to the above argument that $f^G_{\cZ} $ is isotopic to $f_{\cW}$ we have that $f^G_{\cZ'} $ is isotopic to $f_{\cW'}$, and hence $f^G_{\cZ'} =f_{\cW'}=f_\cW=f^G_{\cZ}\in\cR^G_{[f]}$ completing the proof.

So it just remains to get $\alpha'_i\cap\beta_j=\emptyset$ for all $i,j$.

Since the $\alpha'_i$ are constructed from ${\sigma'_i}^\pm$ by adding short arcs in $\partial D_z$, it suffices to show that we may push all the 
${\sigma'_i}^\pm$ off all the $\beta_j$ in a way that corresponds to an isotopy of the Norman sphere $f^G_{\cZ'} $. 
It will be convenient to describe this pushing-off construction in the domain of $f$, so we want to get ${s'_i}^\pm\cap b_j=\emptyset$,
where $b_j\subset S^2$ is an embedded arc from $y_j^-$ to $y_j^+$ with $f(b_j)=\beta_j$, and ${s'_i}^\pm\subset S^2$ goes from $x_i^\pm$ to $f^{-1}(z_i^\pm)$
with $f({s'_i}^\pm)={\sigma'_i}^\pm$. 

Our construction will work with one $b_j$ at a time, removing intersections with all ${s'_i}^\pm$ in a way that does not create new intersections in any previously cleaned-up $b_k$.
This will be accomplished by describing an isotopy of the Norman sphere tubes induced by pushing (as needed) each ${s'_i}^\pm$ across the endpoints $y_j^\pm$ of $b_j$, using the fact that a disk around $y_j^\pm$ maps to a disk in the Norman sphere consisting of a tube along $\sigma_j^\pm$ into $G^\pm_j$. As observed by Gabai \cite[Rem.5.10]{Gab}, in the case $i=j$ we are \emph{not} able to push 
${s'_j}^\pm$ across $y_j^\pm$, but we \emph{are} able to push ${s'_j}^\pm$ across the opposite-signed $y_j^\mp$.
This is similar to the fact that a handle cannot be slid over itself.

Consider first the case where some $b_j$ only has intersections with a single ${s'_i}^\pm$ (Figure~\ref{fig:beta-push-off-1} left). If $i\neq j$ then these intersections can all be eliminated by an isotopy of ${s'_i}^\pm$ across $y_j^+$ (Figure~\ref{fig:beta-push-off-1} right).
If $i= j$ then $b_j\cap {s'_j}^\pm$ can be eliminated by an isotopy of ${s'_j}^\pm$ across the oppositely-signed $y_j^\mp$.
These isotopies pushing ${s'_j}^\pm$ off $b_j$ can be done without creating any intersections among the parallel strands of ${s'_j}^\pm$.
\begin{figure}[ht!]
        \centerline{\includegraphics[scale=.22]{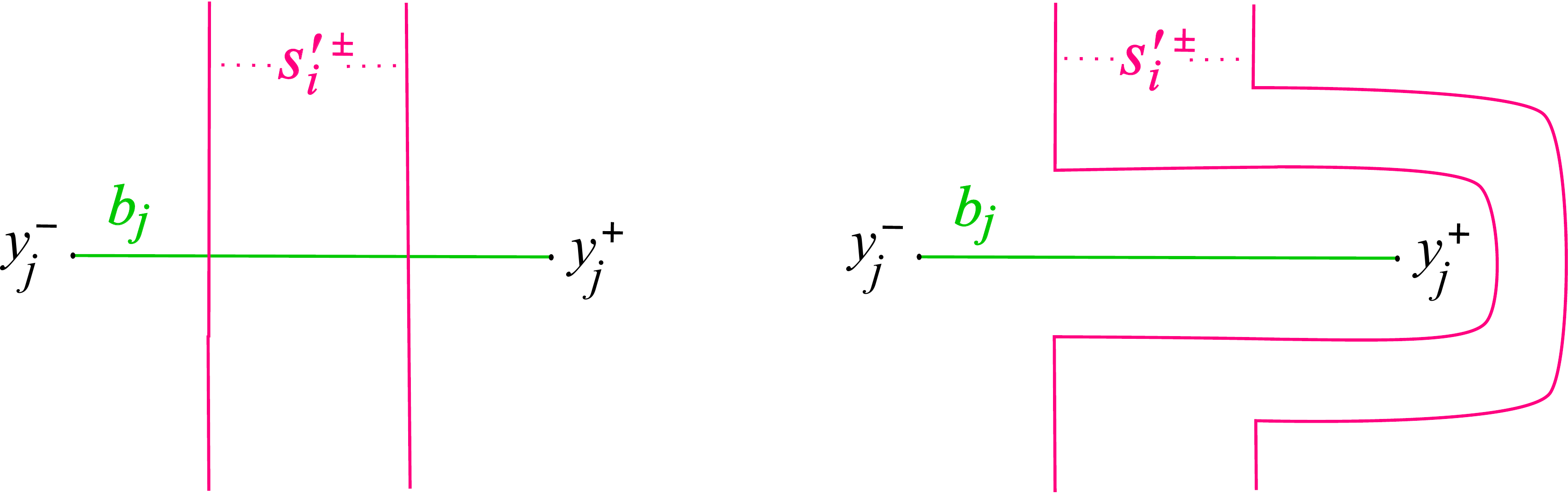}}
        \caption{For $i\neq j$ all strands of ${s'_i}^\pm$ can be pushed off $b_j$ across $y_j^+$.}
        \label{fig:beta-push-off-1}
\end{figure}

Next consider the case where $b_j$ intersects only the two arcs ${s'_j}^+$ and ${s'_j}^-$, each in a single point
$r^+=b_j\cap {s'_j}^+$ and $r^-=b_j\cap {s'_j}^-$. If $r^\pm$ is adjacent to $y_j^\mp$ in $b_j$, then each $r^\pm$ can be eliminated as in the previous case by pushing ${s'_j}^\pm$ across $y_j^\mp$.  
If $r^\pm$ is adjacent to $y_j^\pm$ in $b_j$, then first eliminate $r^-$ by pushing ${s'_j}^-$ across $y_j^+$ and under ${s'_j}^+$, as in Figure~\ref{fig:beta-push-off-2} left. 
Then eliminate $r^+$ by pushing ${s'_j}^+$ across $y_j^-$ and over ${s'_j}^-$, as in Figure~\ref{fig:beta-push-off-2} center. At this point we have 
$b_j\cap {s'_j}^\pm=\emptyset$, but ${s'_j}^+$ intersects ${s'_j}^-$ in two points $q$ and $q'$. Each of $q$ and $q'$ can be eliminated by pushing ${s'_j}^-$ along ${s'_j}^+$ and across (under)
$x_j^+$ as in Figure~\ref{fig:beta-push-off-2} right, since the tube around $\sigma_j^-$ has a smaller radius.
Note that the pushing of ${s'_j}^-$ along ${s'_j}^+$ will create new intersections between ${s'_j}^-$ and any other $b_k$ with $k\neq j$ that intersected ${s'_j}^+$ along the strand of the original ${s'_j}^+$ between $x_j^+$ and $r^+$. But such new intersections only are created in a $b_k$ that has yet to be cleaned up.
\begin{figure}[ht!]
        \centerline{\includegraphics[scale=.2]{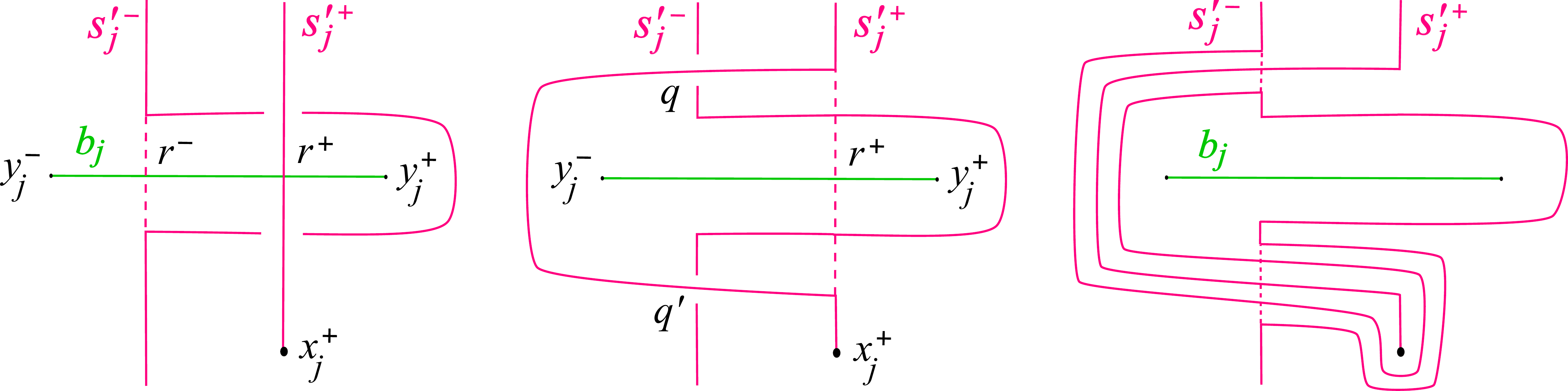}}
        \caption{}
        \label{fig:beta-push-off-2}
\end{figure}

The construction of the previous paragraph can be adapted to the general case where $b_j$ intersects arbitrary strands of ${s'_i}^\pm$ for arbitrary $i$ as follows.
(Picture the ${s'_j}^\pm$-arcs in Figure~\ref{fig:beta-push-off-2} as two among several parallel collections of strands.) First simultaneously push all strands of ${s'_j}^-$ and all strands of any other ${s'_i}^\pm$ with $i\neq j$ under any and all strands of ${s'_j}^+$ and across $y_j^+$.
This can be done in parallel, without creating any intersections among the strands that are being isotoped.
Then simultaneously push any and all strands of ${s'_i}^+$ over all other strands and across $y_j^-$. This can be done in parallel, so that the only resulting intersections between $s'$-arcs are where ${s'_i}^+$ passes over other strands. At this point $b_j$ is disjoint from all ${s'_i}^\pm$, and the intersections among $s'$-arcs can all be eliminated by pushing the under-crossing arcs along ${s'_j}^+$ across (under) $x_j^+$.
\end{proof}

\subsection{Independence of pairings and Whitney arcs}\label{sec:choices-of-arcs-and-pairings}

From Lemmas~\ref{lem:Norman-for-every-Whitney} and~\ref{lem:Norman-independence-of-arcs} we get:

\begin{cor}\label{cor:w-disk-independence-of-arcs-and-pairings}
If two choices of Whitney disks $\cW, \cW'\in\textsf{C}_{\textrm{W-disks}}$ are each compatible with the same choice of sheets $\sx\in\textsf{C}_{\textrm{sheets}}$, then $f_\cW$ is isotopic to $f_{\cW'}$. In particular, $f_\cW\in \cR^G_{[f]}$ is independent of $\sx$-compatible choices of pairings, Whitney arcs and Whitney disks. $\hfill\square$
\end{cor}
As a consequence, $f_\cW\in \cR^G_{[f]}$ only depends on $\sx$ and it's safe to write $f_\cW=:f_\sx\in \cR^G_{[f]}$, where the existence of an
$\sx$-compatible $\cW$ is guaranteed by Lemma~\ref{lem:disks-exist-for-all-choices}.

By the same lemmas we also see that $f_\sx$ is isotopic to the Norman sphere $f^G_\sx$, whose isotopy class therefore only depends on $\sx$ but not  on $G$.

To complete the proof of Gabai's LBT it remains to consider the $\sx$-dependence of $f_\sx$.

\subsection{Double sheet changes}\label{sec:double-sheet-change}
Let $\sx=\{x_1,\dots,x_{2n}\}\in\textsf{C}_{\textrm{sheets}}$ and recall that $g(x_i)\in \pi_1M$ is represented by a double point loop through $p_i$ which is the image of an oriented arc from $x_i$ to $y_i$, where $f^{-1}(p_i)=\{x_i, y_i\}$. Switching the choice $x_i$ to $y_i$ changes $g(x_i)$ to $g(y_i)=g(x_i)^{-1}$ while keeping the sign $\epsilon_i$ of $p_i$. Changing the whisker for $f$ changes all $g(x_i)$ by a fixed conjugation and also keeps the signs.

Assume that for two indices $i,j$ we have $\epsilon_j=-\epsilon_i$ and $g(x_j)=g(x_i)=:g$. 
Then a different choice of sheets $\sx'\in\textsf{C}_{\textrm{sheets}}$ can be defined by replacing $x_i$ by $y_i$ and replacing $x_j$ by $y_j$,
since it satisfies our requirement $(*)$ in Section~\ref{sec:choices} with the canceling terms 
$\epsilon_i\cdot g+\epsilon_j\cdot g=0$ replaced by $\epsilon_i\cdot g^{-1}+\epsilon_j\cdot g^{-1}=0$.

We will refer to such a change of sheet choice as a \emph{double sheet change}.

\begin{lem}\label{lem:double-sheet-change}
If $\sx,\sx'\in\textsf{C}_{\textrm{sheets}}$ differ by a double sheet change,
then $f_{\sx}=f_{\sx'}\in \cR^G_{[f]}$. 
\end{lem}

\begin{proof}
Let $\{x_i,x_j\}\subset\sx$ be the local sheets involved in the double sheet change. There is a choice of pairings $\sx^\pm$ compatible with $\sx$ such that $x_i=x_1^+$ and $x_j=x_1^-$ (or vice versa). 
Moreover, by Lemma~\ref{lem:disks-exist-for-all-choices} there is a choice of Whitney disks $\cW=\{W_1,\dots,W_n\}$ compatible with $\sx^\pm$, i.e.\ $p_i$ and $p_j$ are paired by $W_1$. 

Let $\cW':=\{W'_1, W_2,\dots,W_n\}$ be the choice of Whitney disks where $W_1'$ differs from $W_1$ only by precomposing with a reflection of the domain $D^2$ across the horizontal diameter.
This exchanges the two boundary arcs of $W_1$ but does not change the effect of doing a Whitney move since $W_1$ and $W_1'$ have the same image in $M$. 
Now observe that $\cW'$ is compatible with $\sx'$ and it follows from Corollary~\ref{cor:w-disk-independence-of-arcs-and-pairings} that
$f_\sx=f_{\cW}=f_{\cW'}=f_{\sx'}\in \cR^G_{[f]}$. 
\end{proof}

\subsection{Choice of sheets for double point loops not of order $2$} \label{sec:choice-of-sheets-non-trivial-not-in-T}
Consider a sheet choice $\sx=\{x_1,\dots,x_{2n}\}\in\textsf{C}_{\textrm{sheets}}$ such that for some $i$ we have $g_i:=g(x_i)\in\pi_1M$ with $g_i^2\neq 1$. If $\sx'$ is a different choice of sheets that takes $y_i$ as the preferred preimage instead of $x_i$, then this has the effect of inverting $g_i$. Since $g_i\neq g_i^{-1}$, in order for $\sx'$ to satisfy the requirement $(*)$ in Section~\ref{sec:choices} of a choice of sheets it follows that $\sx'$ must also switch some oppositely-signed $x_j$ to $y_j$, where $g(x_j)=g(x_i)$.
So $\sx'$ differs from $\sx$ by at least one double sheet change, and Lemma~\ref{lem:double-sheet-change} applied finitely many times gives:
\begin{lem}\label{lem:choice-of-sheets-non-trivial-not-in-T}
If choices of sheets $\sx,\sx'\in\textsf{C}_{\textrm{sheets}}$ only differ at self-intersections $p_i$ where the double point loops $g_i$ satisfy  $g_i^2\neq 1$, then $f_\sx=f_{\sx'}\in \cR^G_{[f]}$. \hfill$\square$
\end{lem}
Note that the assumption does not depend on the whisker for $f$.

\subsection{Choice of sheets for trivial double point loops}\label{sec:choice-of-sheets-trivial-element}
Let  $p_i$ be a self-intersection of $f$ with trivial group element $1\in\pi_1M$.
By the same construction as in the proof of Lemma~\ref{lem:choice-of-disks-exists}, $p_i$ admits a clean \emph{accessory disk} $A_i$, i.e.\ $A_i$ is a framed embedded disk
with interior disjoint from $f$ such that the boundary circle $\partial A_i\subset f$ changes sheets just at $p_i$. 
See \cite[Sec.7]{ST1} for details on accessory disks. If $p_i^+$ and $p_i^-$ are oppositely signed with trivial group element, then clean Whitney disks for $p_i^\pm$ can be constructed by banding together two clean accessory disks $A_i^\pm$ as in Figure~\ref{fig:accessory-disk-bands-1}, which shows two choices of bands resulting in Whitney disks $W_i$ and $W'_i$ which induce the possible different sheet choices. 
These Whitney disks are supported in a neighborhood of the union of the two accessory disks together with a generic disk in $f$ containing the accessory circles $\partial A_i^\pm$.
We will show that $W_i$ and $W_i'$ are isotopic via an ambient isotopy supported near one of the accessory disks. Hence $f_{W_i}$ is isotopic to $f_{W_i'}$.
\begin{figure}[ht!]
        \centerline{\includegraphics[scale=.175]{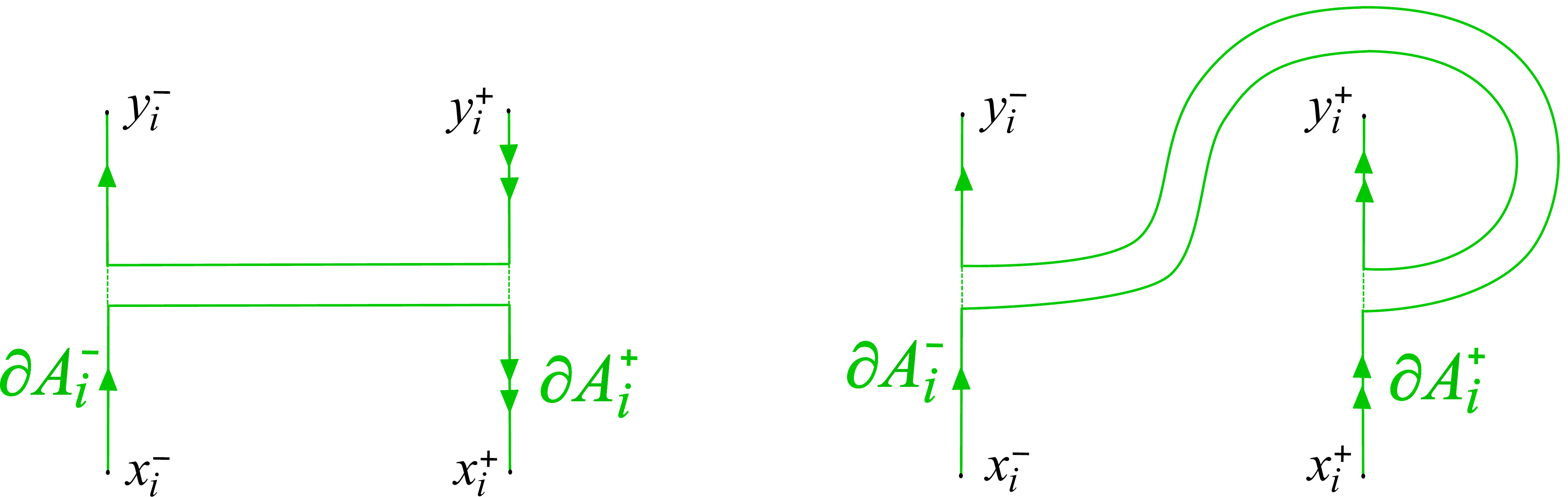}}
        \caption{Preimages of Whitney circles for $W_i$ (left) and $W_i'$ (right) formed by banding together accessory disks $A_i^\pm$ in two different ways, with $W_i$ satisfying the sheet choice $\{x_i^-,x_i^+\}$ and $W'_i$ satisfying the sheet choice $\{x_i^-,y_i^+\}$.  Applying the rotation isotopy of Lemma~\ref{lem:rotate-B4} to $A_i^+$ interchanges $x_i^+$ and $y_i^+$.}
        \label{fig:accessory-disk-bands-1}
\end{figure}

A regular neighborhood of a clean accessory disk is diffeomorphic to a standard model in $4$--space, so we work locally, dropping superscripts and subscripts. 

Let $(\Delta,\partial\Delta)\imra (B^4,S^3)$ be a generic $2$--disk with a single self-intersection $p$ which is the result of applying a cusp homotopy \cite[1.6]{FQ} to a standard $(D^2,S^1)\subset (B^4,S^3)$. Then $p$ admits a clean accessory disk $A$, and the following lemma will be proved:
\begin{lem}\label{lem:rotate-B4}
There is an ambient isotopy $h_s$ of $B^4$ such that
\begin{enumerate}
\item
$h_0$ is the identity,
\item
$h_1(\Delta\cup A)=\Delta\cup A$, 
\item $h_1|_A$ is a reflection of $A$ that fixes the double point of $\Delta$ in $\partial A$, and
\item
$h_1|_{\partial\Delta}$ is a rotation by $180$ degrees.
\end{enumerate}
\end{lem}

Applying Lemma~\ref{lem:rotate-B4} to a $B^4$-neighborhood of $A_i^+$ we see that the two Whitney disks $W_i$ and  $W'_i$ in
Figure~\ref{fig:accessory-disk-bands-1} are isotopic:
Rotating the right accessory arc $\partial A_i^+$ by $180$ degrees drags one band to the other, and hence one Whitney disk to the other.

\begin{proof}
To prove Lemma~\ref{lem:rotate-B4}, consider $\Delta$ as the trace of a null-homotopy of the Whitehead double of the unknot in $S^3 \times \{0\} =\partial B^4$ which pulls apart the clasp in a collar $S^3\times I\subset B^4$, creating the self-intersection $p$ admitting a clean accessory disk $A$, as in Figure~\ref{fig:accessory-disk-rotation-1}. Define the homotopy $h_s$ of $\Delta$ in the coordinates of Figure~\ref{fig:accessory-disk-rotation-1} to be rotation around the vertical axis by $180\cdot s$ degrees in each slice $S^3\times \{t\}: h_s(x,t) = (h_s(x),t)$, with $s$ and $t$ each going from $0$ to $1$. Extend $h_s$ to $B^4$ by tapering the rotation back to zero inside the collar $(x,u), u\in [1,0]$ of a smaller 4-ball that is the complement of the $S^3 \times I$ already used: $h_s(x,u):=h_{u\cdot s}(x)$, with $s$ and $u$ each going from $1$ to $0$ (and $h_s=$ identity outside this second collar). 
\begin{figure}[ht!]
        \centerline{\includegraphics[scale=.2]{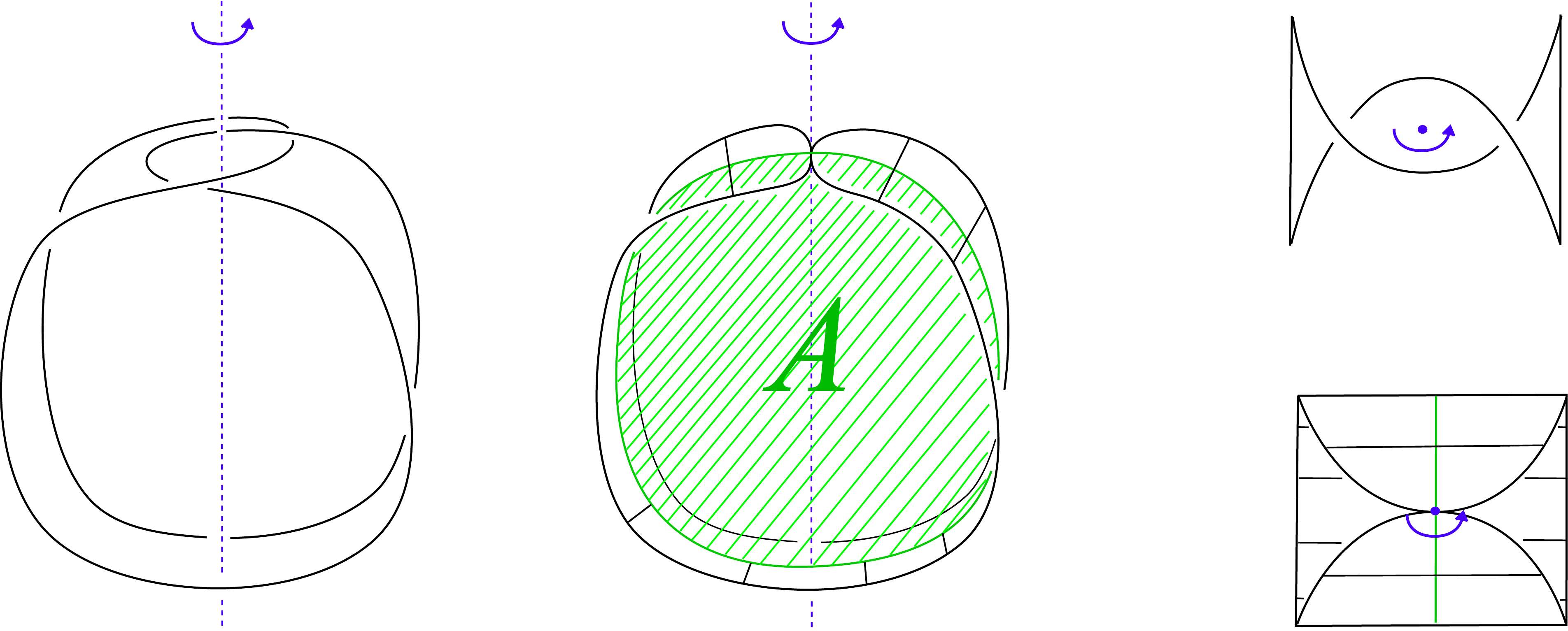}}
        \caption{Left: The Whitehead double of the unknot in $S^3$ is the boundary of $\Delta$. Center: The clean accessory disk $A$ for the self-intersection $p$ of $\Delta$ which corresponds to the clasp singularity. Both $\Delta$ and $A$ have $180$ degree rotational symmetry (top views of left and center on upper and lower right).}
        \label{fig:accessory-disk-rotation-1}
\end{figure}
\end{proof}
By Corollary~\ref{cor:w-disk-independence-of-arcs-and-pairings} we can compute $f_\sx=f_\cW$ by $\cW\in\textsf{C}_{\textrm{W-disks}}$ whose Whitney disks pairing self-intersections with trivial group elements are formed from banding together accessory disks as above.
So in combination with Lemma~\ref{lem:choice-of-sheets-non-trivial-not-in-T} we have:
\begin{cor}\label{cor:choice-of-sheets-non-involution}
If choices of sheets $\sx,\sx'$ only differ at self-intersections $p_i$ whose double point loops don't have order~$2$, then
$f_\sx=f_{\sx'}\in \cR^G_{[f]}$.
\end{cor}

This result completes the proof of Gabai's LBT. To prove our main Theorem~\ref{thm:4d-light-bulb} it remains to understand the $\sx$-dependence of $f_\sx$ in the presence of self-intersections with group elements of order~$2$. In the subsequent Section~\ref{sec:FQ} and Section~\ref{sec:proof-main} we will show that it is completely controlled by the Freedman--Quinn invariant.

\section{The Freedman--Quinn invariant}\label{sec:FQ}
In Section~\ref{sec:3-in-6} we review some relevant aspects of the intersection form on $\pi_3$ of a 6--manifold. In Section~\ref{sec:homotopies} the Freedman--Quinn invariant is defined using the self-intersection invariant applied to the track of a homotopy between spheres in $M^4\times \R$, which is a map of a 3--manifold to a 6--manifold rel boundary.

\subsection{3--manifolds in 6--manifolds}\label{sec:3-in-6}

Recall that for a smooth oriented 6--manifold $P^6$, the intersection and self-intersection invariants give maps
\[
\lambda_3: \pi_3P \times \pi_3P \to  \Z\pi_1P \quad \text{ and } \quad \mu_3: \pi_3P  \to  \Z\pi_1P / \langle  g+g^{-1}, 1 \rangle.
\]
The intersection invariant $\lambda_3$ can be computed geometrically by representing the two homotopy classes by transverse based maps $S^3 \to P$ and counting their intersection points with signs and group elements. Similarly, for the self-intersection invariant $\mu_3$ one represents the homotopy class by a generic based map $a:S^3\imra P$ and counts transverse self-intersections, again with signs and group elements: 
\[
\mu_3(a):= [\sum_p \e_p \cdot  g_p] \in \Z\pi_1P / \langle  g+g^{-1}, 1 \rangle
\]
We note that in this dimension, switching the choice of sheets at a double point $p$ changes $g_p\in \pi_1P$ to $g_p^{-1}$ (as in dimension 4) but the signs change from $\e_p$ to $-\e_p$, explaining the relation $g+g^{-1} =0$ in the range of $\mu_3$ (as a opposed to $g-g^{-1} =0$ in the range of $\mu_2$ in dimension 4). The relation $1=0$ is important to make $\mu_3(a)$ only depend on the homotopy class of $a$ since a cusp homotopy introduces a double point with arbitrary sign and trivial group element (as in dimension 4). In this dimension we will not be using the ``tilde'' notation for this reduced self-intersection invariant in the interest of streamlining notation, and the relation $1=0$ will always be assumed in the target of $\mu_3$. Changing the whisker for $a$ changes $\mu_3(a)$ by a conjugation with the corresponding group element. 
The homotopy invariance of $\mu_3$ follows from the fact that a generic homotopy is isotopic to a sequence of cusps, finger moves and Whitney moves, none of which changes the invariant.

Using the involution $\bar g:= g^{-1}$ on $\Z\pi_1P$, the ``quadratic form'' $(\lambda_3, \mu_3)$ satisfies the  formulas
\begin{equation}\tag{$**$}
\mu_3(a+b)= \mu_3(a) + \mu_3(b) +[\lambda_3(a,b)]  \quad\text{ and } \quad\lambda_3(a,a) = \mu_3(a) -\overline{\mu_3(a)} \in \Z\pi_1P
\end{equation}
where the second formula has no content for the coefficient at the trivial element in $\pi_1P$: Since $\lambda_3$ is skew-hermitian, it vanishes on the left hand side, whereas it's automatically zero on the right hand side that is defined by picking a representative of $\mu(a)\in\Z\pi_1P$ and then applying the involution to that specific choice.

The case $N=M\times \R$ of the following lemma describes the homomorphism used in Theorem~\ref{thm:4d-light-bulb} and will be used in the definition of the Freedman--Quinn invariant given in Section~\ref{sec:homotopies}.
Recall that $T_N$ denotes the 2-torsion in $\pi_1N$.
\begin{lem}\label{lem:Wall}
If $P^6 = N^5 \times I$, then $\mu_3:\pi_3N \to \F_2T_N\leq  \Z\pi_1N / \langle  g+g^{-1}, 1 \rangle$ is a homomorphism.
\end{lem}
\begin{proof}
First note that the intersection pairing $\lambda_3$ vanishes identically, since one can represent $a,b\in \pi_3(N \times I)$ disjointly (and hence transversely without intersections) in $N \times 0$ respectively $N \times 1$.
So from the second formula in $(**)$ above, together with the observation that $\F_2T_N \leq \Z\pi_1N/ \langle g+g^{-1} , 1 \rangle$ is the subgroup generated by $\{ \zeta\in \Z\pi_1N   \mid \zeta = \bar \zeta \neq 1\}$, we see that $\mu_3(a)$ lies exactly in 
$\F_2T_N$. And from the first formula in $(**)$ it follows that $\mu_3:\pi_3N \to \F_2T_N$ is a homomorphism.
\end{proof}

The next lemma will be used in the proof of Corollary~\ref{cor:infinitely-many} given in Section~\ref{sec:infinitely-many}.
\begin{lem}\label{lem:Hurewicz}
$\mu_3$ factors through the Hurewicz homomorphism $\pi_3P\twoheadrightarrow H_3 \widetilde P$. 
\end{lem}
 
\begin{proof}
We will use Whitehead's exact sequence $\Gamma(\pi_2P) \to \pi_3P\twoheadrightarrow H_3 \widetilde P$ from \cite{Wh}, 
where the first map is induced by the quadratic map $\eta: \pi_2P \to \pi_3P$ which is pre-composition by the Hopf map $h:S^3\to S^2$. The second map is surjective by Hurewicz's theorem applied to the 1-connected space $\widetilde P$. 
We need to show that 
\[
\mu_3(a_3+ \eta(a_2))= \mu_3(a_3) \ \forall \ a_i\in\pi_iP,
\]
By the quadratic property of $\mu_3$ given by the first formula in $(**)$, we get
\[
\mu_3(a_3+ \eta(a_2))= \mu_3(a_3) + \mu_3(\eta(a_2)) + \lambda_3(a_3,\eta(a_2))
\]
and so we want to show that the last two terms on the right vanish. Representing $a_2$ by an embedding $b_2:S^2 \hra P^6$, we see that $\eta(a_2)=b_2 \circ h$ is supported in the image of $b_2$. As a consequence of working in a 6--manifold, we can find a representative of $a_3$ in the complement of this 2--manifold and hence their intersection invariant $\lambda_3$ vanishes. Similarly, there is a generic representative of $\eta(a_2)$ which has support in the normal bundle of $b_2$, a simply-connected 6--manifold. Therefore, $\mu_3(\eta(a_2))=0$ since the trivial group element is divided out in the range of $\mu_3$.
\end{proof}

\begin{rem}\label{rem:homological}
Even though we obtain a map $\mu_3: H_3 \widetilde P\to \Z\pi_1P / \langle  g+g^{-1}, 1 \rangle$, it is not clear to us whether $\mu_3$ can be computed in a ``homological way'', i.e.\ without representing homology classes by generic maps and counting double points. This can be done for $\lambda_3$ but the second formula in $(**)$ shows that $\lambda_3(a,a)$ does {\em not} determine $\mu_3(a)$ at group elements of order~$2$.
\end{rem}

\subsection{The self-intersection invariant for homotopies of 2--spheres in 5--manifolds}\label{sec:homotopies}
The above description of $\mu_3$ can also be applied to define self-intersection invariants of properly immersed simply-connected $3$--manifolds in a $6$--manifold. In this setting $\mu_3$ is computed just as above, by summing signed double point group elements, and is invariant under homotopies that restrict to isotopies on the boundary. 

Now fix a smooth oriented 5--manifold $N$. For any homotopy $H:S^2 \times I \to N^5$ between embedded spheres in $N$ we define the self-intersection invariant of $H$
\[
\mu_3(H)\in\Z\pi_1N / \langle  g+g^{-1}, 1 \rangle
\]
to be the self-intersection invariant $\mu_3$ of a generic track $S^2 \times I \imra N^5 \times  I$ for $H$ (with fixed boundary and based at the sphere $H_0$). The invariant $\mu_3(H)$ is independent of the choice of generic track since any two choices of perturbations to make $S^2 \times I \imra N^5 \times  I$ generic differ at most by a homotopy rel boundary.

Definition~\ref{def:fq} of the Freedman--Quinn invariant below involves the case where $N^5=M^4 \times \R$ and $H_0, H_1$ are embeddings $S^2\hra M \times 0$. In this case one has that $\mu_3(H)\in \F_2 T_M$, as in Lemma~\ref{lem:Wall}.
The next lemma characterizes the dependence of $\mu_3(H)$ on the choice of homotopy $H$ only in this case.
\begin{lem}\label{lem:self-homotopies}
If $J:S^2 \times I \imra M \times \R \times I$ is a generic track of a based self-homotopy of $R:S^2\hra M \times 0$,
then $\mu_3(J) \in \F_2 T_M$ lies in the image of the homomorphism $\mu_3: \pi_3M \to \F_2T_M$.
\end{lem}
It follows that for any two based homotopies $H, H': S^2 \times  I\to M^4 \times \R$ between embedded spheres $H_0=H'_0$ and $H_1=H'_1$ in $M \times 0$, 
the difference $\mu_3(H) - \mu_3(H')\in \F_2 T_M$ lies in the image of $\mu_3: \pi_3M \to \F_2T_M$, since
stacking the two homotopies gives a based self-homotopy $J=H\cup -H'$ such that $\mu_3(J) = \mu_3(H) - \mu_3(H')$.
\begin{proof}
By assumption, $J$ agrees with the track $R \times I$ of the product self-homotopy on the 2-skeleton $S^2 \times \{0,1\} \cup z_0 \times I$ of $S^2 \times I$. So they only differ on the 3-cell where $R\times I$ is represented by $R(D^2)\times I$ (here $D^2$ is the complement in $S^2$ of a small disk around $z_0$) and
$J$ is represented by a generic $3$--ball $B:D^3\imra (M\times \R \times I) \smallsetminus \nu(z \times I)$.
Here $z$ denotes the image of the basepoint $z_0\in S^2$, and by construction the boundaries of these $3$--balls are parallel copies of an embedded 2--sphere in the boundary of a small neighborhood of $R \times \{0,1\}\cup(z \times I)$.
Gluing $B$ and $R(D^2)\times I$ together along a small cylinder $S^2 \times I$ between their boundaries yields a map of a $3$--sphere
$b:= B\cup R(D^2) \times I:S^3 \to M \times \R \times I$.
To prove the lemma we will show that $\mu_3(J)=\mu_3(b)\in\mu_3(\pi_3(M))$.  

First note that all contributions to $\mu_3(J)$ come from double point loops in $B$. 
There are two types of self-intersections that contribute to $\mu_3(b)$, namely the self-intersections of the immersed $3$--ball $B$ and the intersections between $B$ and the embedded $3$--ball $R(D^2)\times I$. Observe that $B\pitchfork R(D^2)\times I=J\pitchfork R\times I$, with the corresponding loops based at $z\in R\times 0$ determining the same group elements contributing to both of $\mu_3(b)$ and $\lambda_3(J,R\times I)$.   

Now note that $\lambda_3(J,R\times I)=0$, 
since $R\times I\subset M\times 0\times I$ can be made
disjoint from a homotopic (rel boundary) copy of $J$ in $M\times 1\times I$.
So $B\pitchfork R(D^2)\times I$ contributes trivially to $\mu_3(b)$, and it follows that $\mu_3(b)=\mu_3(J)$
since both are determined by double point loops in $B$.
\end{proof}

\begin{defn}\label{def:fq}
Given embeddings $R,R':S^2\hra M^4$ which are based homotopic, their Freedman--Quinn invariant is given by:
\[
\fq(R,R'):= [\mu_3(H)] \in \F_2T_M/\mu_3(\pi_3M)
\]
for any choice of based homotopy $H$ from $R\times 0$ to $R'\times 0$ in $M \times \R$.  
\end{defn}
Recall from the beginning of the proof of Lemma~\ref{lem:based} that a common dual for $R$ and $R'$ forces any given homotopy in $M$ to be based and hence $\fq(R,R')$ is defined for any pair $R,R'\in \cR^G_{[f]}$. This definition of $\fq(R,R')$ is independent of the choice of $H$ by Lemma~\ref{lem:self-homotopies}.

\subsection{Computing the Freedman--Quinn invariant}\label{subsec:compute-fq}

We show how to compute $\fq(R,R')$ as a ``difference of sheet choices'' for embedded $2$-spheres $R\times 0$ and $R'\times 0$ in $M \times \R$.
 
Consider a homotopy $H$ given by finger moves on $R$ leading to a middle level $f:S^2\imra M$, followed by Whitney moves on $f$ leading to $R'$. The collection of Whitney disks $\cW$ on $f$, inverse to the finger moves, gives $f_\cW=R$ and determines a choice of sheets 
$\sx=(x_1,\ldots,x_{2n})$, and the collection of Whitney disks $\cW'$ such that $f_{\cW'}=R'$ determines a choice of sheets 
$\sx'=(x'_1,\ldots,x'_{2n})$.

We will describe an isotopy in $M \times \R$ from $R \times 0$ to $f \times  b $, where $b:S^2 \to \R$ will be a sum of bump functions that ``resolves'' the double points in $f$. For simplicity of notation, we'll assume that $f$ is the result of just a single finger move, with $\sx=(x_1,x_2)$.

First define for each $x\in S^2$ a smooth family of non-negative bump functions $b^x_s:S^2\to \R$ which are supported in a small neighborhood of $x$ and have maximum $b^x_s(x)=s$. 
There is a homotopy $R_s$, $s\in [0,1]$, describing how the finger grows from $R$ to the self-tangency which introduces an identification of $x,y\in S^2$, where $y$ gives the ``finger tip'' $R_s(y)$ while $R_s(x)$ is fixed for all $s$. It gives an isotopy $R_s \times  b^x_s$ from $R \times 0$ to $R_1 \times  b^x_1$, with the self-tangency avoided by the bump $b^x_1$ having lifted the image of the $x$-sheet above 
 what was the tangency point (see Figure~\ref{fig:2d-finger-move-sheet-resolution} left).

\begin{figure}[ht!]
        \centerline{\includegraphics[scale=.23]{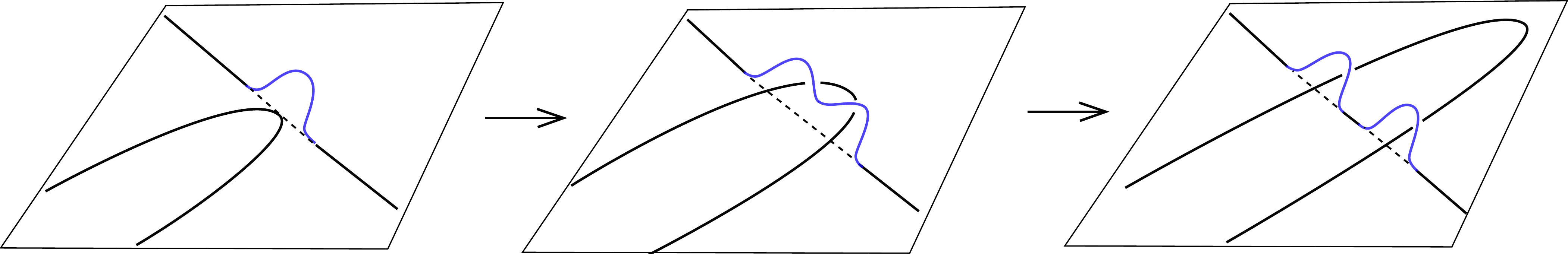}}
        \caption{A single bump splitting into two, along a finger move.}
        \label{fig:2d-finger-move-sheet-resolution}
\end{figure}
We extend this to an isotopy in $M \times \R$ from $R \times 0$ to an embedding $f \times  b $:
As $R_s$ continues to move towards $f$, the self-tangency splits into two transverse intersection points, and we arrange the single bump $b^x_1$ to split into a sum of two bumps which finally arrives at $b:=b^{x_1}_1 + b^{x_2}_1$ when the finger move is done, see Figure~\ref{fig:2d-finger-move-sheet-resolution}.

Note that in this convention, the chosen sheets $x_i\in S^2$ represent ``over-crossings'' of the embedding $f \times b$. The isotopy class of this embedding does not depend on the particulars of $b$ but only on the
choice of sheets $\sx$. In the general case of $n$ finger moves such a $b$ can be defined simultaneously to get a corresponding isotopy.

Turning the homotopy $H$ upside down, we can also consider finger moves leading from $R'$ to $f$ which are inverse to the Whitney moves along Whitney disks in $\cW'$. Apply the same procedure using the choice of sheets 
$\sx'=(x'_1,\ldots,x'_{2n})$ to get an isotopy in $M\times \R$ from $R' \times 0$ to $f \times  b'$.  
If $x_i=x_i'$ we have $b =b'$ near $x_i$, so these two isotopies can be glued together in that neighborhood. 

If $x_i\neq x_i'$ there is a local homotopy $H_i(s) := f \times (b^{x^\prime_i}_{1-s} + b^{x_i}_s)$ that moves $f \times  b'$ to locally coincide with $f \times  b$ by a ``crossing change'' (see Figure~\ref{fig:2d-sheet-resolution}). $H_i$ has a single double point where it identifies $(x_i,1/2)$ with $(x_i',1/2)$. The associated group element is $g(x_i)\in\pi_1M$ associated to the sheet choice $x_i$ of the double point $f(x_i)$.
\begin{figure}[ht!]
        \centerline{\includegraphics[scale=.23]{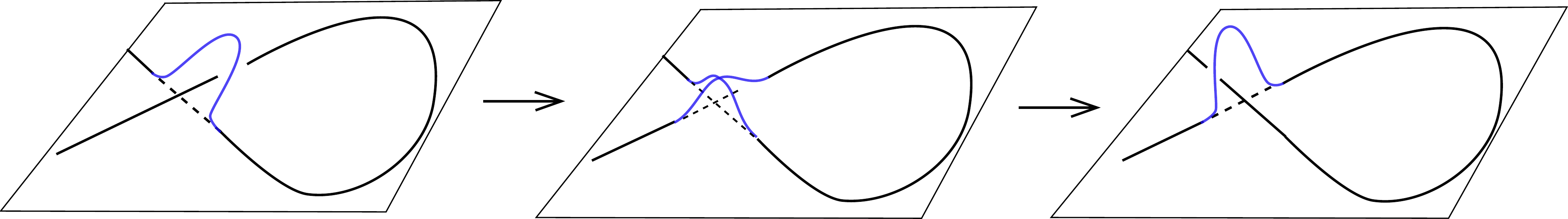}}
        \caption{Two bumps crossing in a single point during a local homotopy $H_i$.}
        \label{fig:2d-sheet-resolution}
\end{figure}

Assembling such local homotopies $H_i $ around all $x_i\neq x_i'$, and then composing with the above isotopies from $R \times 0$ to $f \times b$ and from $f \times b'$ to $R' \times 0$, yields a based homotopy $H_{\cW,\cW'}$. Its isotopy class rel boundary only depends on the sheet choices $\sx,\sx'$ and not on the particulars of the bump functions in the construction.

\begin{lem}\label{lem:compute-mu-by-crossing-changes}
$\mu_3(H_{\cW,\cW'} ) = \sum_i  g(x_i) \in \F_2T_M$,
where the sum is over those double points $p_i$ of $f$ for which $x_i\neq x_i'$. This sum is therefore a representative for $\fq(R,R')\in\F_2T_M/\mu_3(\pi_3(M))$. 
\end{lem}
Recall from Lemma~\ref{lem:Wall} and Section~\ref{sec:homotopies} that the target
of $\mu_3(H_{\sx,\sx'} )$ is indeed the subgroup $\F_2T_M$ of $\Z\pi_1N / \langle  g+g^{-1}, 1 \rangle$, i.e.\ any $g(x_i)$ with $g(x_i)^2\neq 1$ must contribute trivially (and we don't have to worry about signs).

\subsection{Singular circles: The origin of the $\fq$ invariant} \label{sec:singular circles}
The $\fq$-invariant originally appeared in the more general setting of \cite[Chap.10.9]{FQ} as the obstruction to eliminating circles of intersections between the cores of $3$-handles in a $5$-manifold.
For the interested reader we briefly explain the connection with singular circles in our setting. The results of this section will not be used in our paper.

The singular set of a generic track $S^2\times I\imra M\times I$ of a regular homotopy from $R$ to $R'$ consists of circles which are double-covered by circles in $S^2\times I$. The group element associated to a singular circle is determined by a double point loop in the image of $S^2\times I$ that changes sheets exactly at one point on the singular circle, with a choice of first sheet orienting the loop. 
The group element $g(\gamma)$ associated to a circle $\gamma$ with connected double cover satisfies $g(\gamma)^2=1$ since $\gamma$ itself represents $g(\gamma)$ and the double cover bounds a disk in the domain. The singular arcs that appear in \cite[Chap.10.9]{FQ} and start/end at cusps, do not occur in our setting since we work with a regular homotopy.
\begin{lem}\label{lem:mu3-equals-circle-count}
$
\fq(R,R')=[\sum_\gamma g(\gamma)]\in \F_2T_M/\mu_3(\pi_3(M)),
$
where the sum is over all $\gamma$ that have connected double covers in $S^2 \times I$.
\end{lem}
\begin{proof}[Sketch of Proof:] The idea is to resolve the singular circles of a track $H:S^2\times I\imra M\times I$ to (at worst) self-intersection points of $S^2\times I\imra  M\times \R\times I$, and compute $\mu_3$. Using the extra $\R$-factor, the singular circles with disconnected covers can be eliminated by perturbing one sheet into the $\R$-direction. By perturbing the sheets that intersect in a circle $\gamma$ with connected double cover partially into the positive $\R$-direction and partially into the negative $\R$-direction, $\gamma$ can be eliminated except for a single  transverse self-intersection with group element $g(\gamma)$. 
\end{proof}

It is interesting to note that these singular circles in $M \times I$ project to the middle level $f:S^2\times 1/2\imra M\times 1/2$ as follows: They map to the union of the boundary arcs of Whitney disks $W_i$ (inverse to finger moves on $R$) and the boundary arcs of Whitney disks $W_i'$ (guiding Whitney moves towards $R'$). These arcs meet at the self-intersections of $f$, so the union $\cup_i\partial W_i\cup\partial W'_i$ is a map of circles into $f(S^2)$. The number of circles will not in general be the number of self-intersection pairs, because the $W_i$ and $W'_i$ may induce different pairings. 

To see that these Whitney disk boundaries are projections of the singular circles to the middle level $f$, consider first the track of the $i$th finger move:
As the finger first touches the sheets and then pushes through, a single tangential self-intersection is created which then splits into two self-intersections that move apart until coming to rest at the end of the finger's motion. So in each sheet the motion of a single point splitting into two traces out one arc in the boundary of the Whitney disk $W_i$ (inverse to the finger move). In the domain $S^2\times I$ of the homotopy we see neighborhoods of two minima  of singular circles, see Figure~\ref{fig:homotopy-track}.
Turning the homotopy upside down, the same observations explain neighborhoods of the maxima.

\begin{figure}[ht!]
        \centerline{\includegraphics[scale=.28]{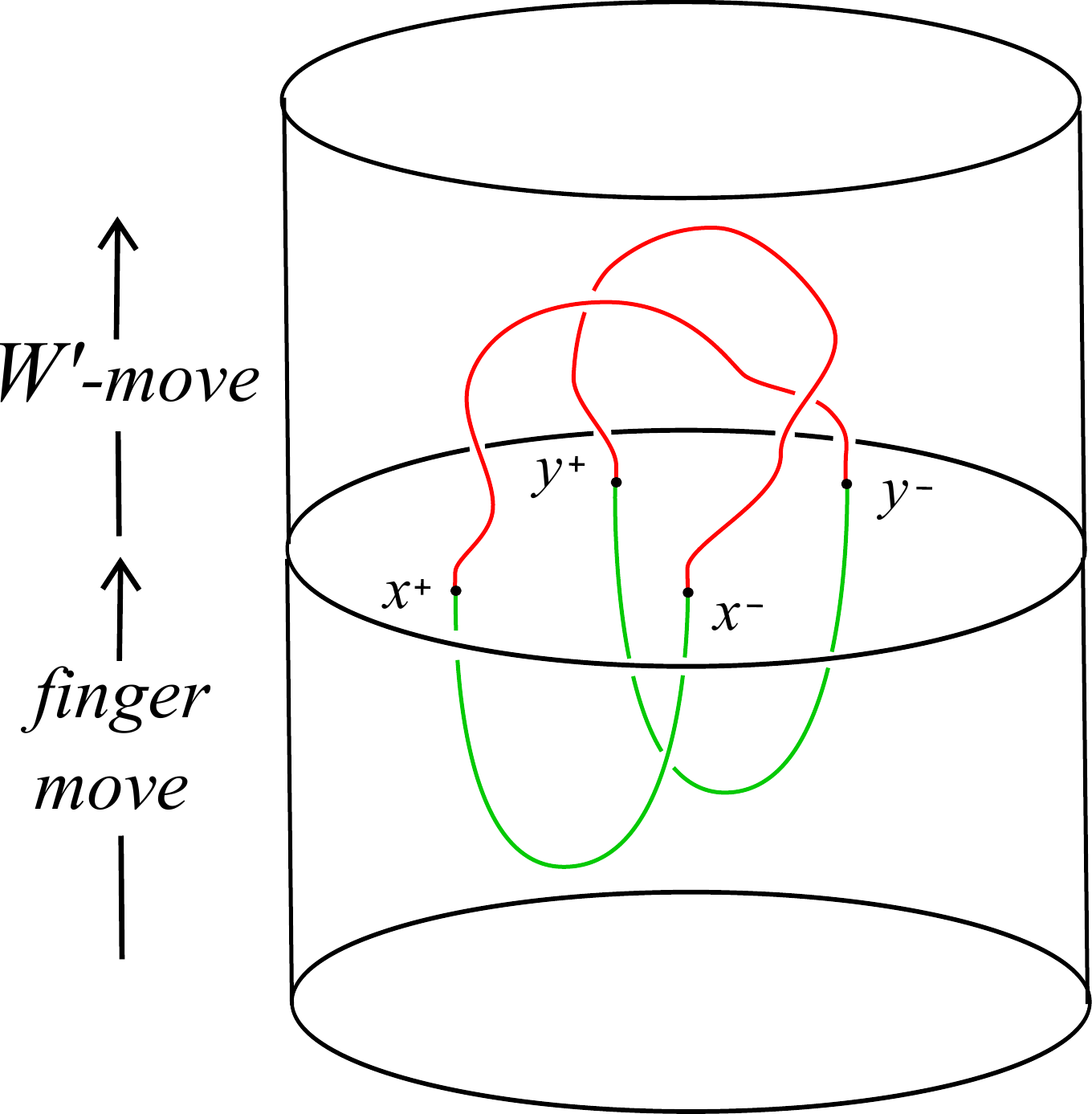}}
        \caption{Singular circles in $S^2\times I$: A connected double cover.}
        \label{fig:homotopy-track}
\end{figure}

Singular circles with connected double covers arise when there are differences in the sheet choices determined by the $W_i$ and $W'_i$ as shown in Figure~\ref{fig:homotopy-track}. This is consistent with our two computations of the Freedman--Quinn invariant in Lemmas~\ref{lem:compute-mu-by-crossing-changes} and~\ref{lem:mu3-equals-circle-count}: Each singular circle with double point loop $g$ corresponds to $n$ finger moves along the same $g$ and $n$ Whitney moves resolving the resulting double points. The number $n$ is the number of minima (and maxima) of the projection $M \times I\to I$ when restricted to the singular circle. The double cover is connected if and only if $g^2=1$ and there is an odd number of sheet changes from the sheet choice determined by the finger moves to the sheet choice of the Whitney moves.

\section{Proof of Theorem~\ref{thm:4d-light-bulb}}\label{sec:proof-main}

The last sentence of Theorem~\ref{thm:4d-light-bulb} follows from the fact that all our constructions, including throughout this section, are supported away from $G$. That $\cR^G_{[f]}\neq\emptyset$ if and only if Wall's reduced self-intersection invariant $\widetilde\mu(f)$ vanishes follows from Lemma~\ref{lem:choice-of-disks-exists}, since the vanishing of
$\widetilde\mu(f)$ is a sufficient condition for the existence of null homotopic Whitney circles for all self intersections of $f$, and is a necessary condition for $f$ to be homotopic to an embedding.
For the rest of Theorem~\ref{thm:4d-light-bulb}, we will proceed with the following steps:

\begin{enumerate}
\item [A.] Define the geometric action of $\F_2T_M$ on $ \cR^G_{[f]}$ and show that 
\[
\fq(t\cdot R,R) = [t]\in \F_2T_M/\mu_3(\pi_3M)\quad \forall R\in \cR^G_{[f]}, t\in \F_2T_M.
\] 
\item[B.] Show that the stabilizers are $\mu_3(\pi_3M)$.
\item[C.] Prove that $R'$ is isotopic to $\fq(R,R')\cdot R$ for all $R,R'\in \cR^G_{[f]}$.
\end{enumerate}
The last item implies the transitivity of the action, so these steps complete the proof of Theorem~\ref{thm:4d-light-bulb}: For a fixed $R\in\cR^G_{[f]}$ the Freedman--Quinn invariant $\fq(R,\cdot)\in\F_2T_M/\mu_3(\pi_3M)$ inverts the $\F_2T_M$-action.  
$\hfill\square$

\subsection{The geometric action on $ \cR^G_{[f]}$}\label{subsec:geo-action-def}
An outline of this construction was given in Section~\ref{sec:outline-proof}.
Given $t=t_1+\cdots +t_n\in\F_2T_M$ and $R\in \cR^G_{[f]}$, we first do $n$ finger moves on $R$, along arcs starting and ending near the base-point in $R$, representing $t_i \in T_M$. The isotopy class of the resulting generic map $f^t:S^2\imra M$ only depends on $R$ and $t$ because $\pi_1(M \smallsetminus R)\cong \pi_1M$ and homotopy implies isotopy for arcs in 4--manifolds.

The second step in the definition of our action is to do Whitney moves on $f^t$ along a collection $\cW^t$ of $n$ Whitney disks to arrive at an embedding denoted by $t\cdot R$, where $\cW^t$ satisfies the following sheet choice condition: Let $\sx=(x_1^+,x_1^-,\dots, x_n^+,x_n^-)$ be a sheet choice such that the collection $\cW$ of Whitney disks $W_i$ which are inverse to the finger moves is $\sx$-compatible and each $W_i$ pairs $f(x_i^\pm)$, i.e.~$\cW$ is also compatible with the pairing choice $\sx^\pm=(x_1^\pm,\ldots,x_n^\pm)$. Then we take $\cW^t$ to be any choice of Whitney disks that is compatible with the sheet choice $\sx^t:=(x_1^+,y_1^-,\ldots,x_n^+,y_n^-)$ which has the sheets of $f^t$ switched at each negative self-intersection $f^t(x^-_i)=f^t(y^-_i)$. Such an $\sx^t$-compatible $\cW^t$ exists by Lemma~\ref{lem:disks-exist-for-all-choices}, and by Corollary~\ref{cor:w-disk-independence-of-arcs-and-pairings} the isotopy class of $f^t_{\cW^t}$ is determined by $\sx^t$,
so $t\cdot R:=f^t_{\cW^t}\in\cR^G_{[f]}$ is well defined. 
Lemma~\ref{lem:compute-mu-by-crossing-changes} implies by construction:

\begin{lem}\label{lem:fq-tdotR-R-equals-t}
$\fq(t\cdot R, R) = [t]$ for all $R\in \cR^G_{[f]}$ and $t=t_1 +\cdots + t_n\in \F_2T_M$. $\hfill\square$
\end{lem}

By Corollary~\ref{cor:choice-of-sheets-non-involution}, sheet choices $\sx$ don't affect the isotopy class of $f_\sx$ at double points whose group element is not 2-torsion. This implies that $t\cdot R$ is unchanged if we perform more finger moves on $R$ along non-2-torsion (and then appropriate Whitney moves to arrive at an embedding). In Lemma~\ref{lem:double-sheet-change} we showed that making double sheet changes doesn't change the isotopy class of $f_\sx$, so only the mod 2 number of finger moves along 2-torsion matters for the isotopy class of $t\cdot R$. This leads to the following result:

\begin{lem}\label{lem:sheets give action}
For $R\in \cR^G_{[f]}$ and $t=t_1 +\cdots + t_n\in \F_2T_M$, $t\cdot R=R'\in \cR^G_{[f]}$ for any $R'$ that is obtained from $R$ by a sequence of finger moves and Whitney moves as long as $\mu_3(H_{\cW,\cW'})=t$. $\square$
\end{lem}
Recall that by Lemma~\ref{lem:compute-mu-by-crossing-changes} $\mu_3(H_{\cW,\cW'})= \sum_{x_i\neq x_i'}  g(x_i)$ only depends on the middle level of the homotopy and the two sheet choices $\sx$ and $\sx'$ (and only at double points whose group elements are 2-torsion and which are counted mod 2).

\subsection{The stabilizer equals $\mu_3(\pi_3M)$}\label{subsec:stabilizer}

\begin{lem}\label{lem:stabilizer-contains-mu-pi3}
If $t\cdot R$ is isotopic to $R$, then $t\in \mu_3(\pi_3M)$, i.e.\ the stabilizer of $R\in \cR^G_{[f]}$ is contained in $\mu_3(\pi_3M)$.
\end{lem}
\begin{proof}
The union of a based homotopy $H^t$ from $R$ to $t\cdot R$ with $\mu_3(H^t)=t$ and a based isotopy $H^0$ from 
$t\cdot R$ to $R$ forms a based self-homotopy $J:=H^t\cup H^0$ of $R$. 
So by Lemma~\ref{lem:self-homotopies}, we have $t=\mu_3(H^t)=\mu_3(J)\in\mu_3(\pi_3M)$.
\end{proof}

\begin{lem}\label{lem:stabilizer}
If $t\in \mu_3(\pi_3M)$ then $t\cdot R$ is isotopic to $R$, i.e.\ $\mu_3(\pi_3M)$ is contained in the stabilizer of any $R\in \cR^G_{[f]}$.
 \end{lem}

\begin{proof}
We first use that a closed tubular neighborhood $\nu(R\cup G)$ has boundary $S^3$ and is homotopy equivalent to $S^2\vee S^2$ (in fact, capping it off with $B^4$ leads to a $S^2$-bundle over $S^2$ with Euler number $R\cdot R$). If $M_0\subset M$ is the closure of the complement of $\nu(R\cup G)$ then the corresponding Mayer-Vietoris sequence (for universal covering spaces) reads as follows:
\[
H_3(\nu(R\cup G);\Z\pi_1M) \oplus H_3(\widetilde M_0) \ra H_3(\widetilde M) \ra H_2(S^3;\Z\pi_1M)
\]
Since the first summand of the first term and the last term are both $0$, we see that the inclusion induces an epimorphism $H_3(\widetilde M_0) \sra H_3(\widetilde M)$. By the surjectivity of Hurewicz maps, this implies that we may assume that $t=\mu_3(a)$ for some $a\in \pi_3M_0$. 

Now represent $a$ by a based generic regular homotopy $F_s:S^2 \times I\to M_0$ from the trivial sphere $F_0=F_1$ in $M_0$ to itself.
By construction, $F_s$ lies in the complement of $R$ at each $s$-level, so we can take a smooth family of ambient connected sums of $F_s$ with $R\times s$
to get a homotopy $H:S^2\times I\to M$ from $R$ to itself with $\mu_3(H)=t$. 
By Lemma~\ref{lem:sheets give action}, this shows that $F_1\# R$ is an admissible representative of our action $t\cdot R$ and therefore, $t\cdot R$ is isotopic to $R$.
\end{proof}

\subsection{The action is transitive}\label{sec:transitive}
This follows directly from:
\begin{lem}\label{lem:fq-inverts-action}
For any $R,R'\in  \cR^G_{[f]}$, we have 
$\fq(R,R')\cdot R = R'.$
\end{lem}
\begin{proof} This is a simple consequence of Lemmas~\ref{lem:compute-mu-by-crossing-changes} and~\ref{lem:sheets give action}.
\end{proof}

\section{Proofs of Corollaries~\ref{cor:infinitely-many} and \ref{cor:unknotting number}}\label{sec:infinitely-many}

Recall the statement of Corollary~\ref{cor:infinitely-many}:
There exist  $4$--manifolds $M$ and $f:S^2\imra M$ with infinitely many free isotopy classes of embedded spheres homotopic to $f$ (and with common geometric dual). These manifolds also admit infinitely many distinct pseudo-isotopy classes of self-diffeomorphisms.

\begin{proof}[Proof of Corollary~~\ref{cor:infinitely-many}:]
We first note that in the example given below Corollary~\ref{cor:infinitely-many},  $\mu_3(\pi_3M) =0$ since $M$ (and hence its universal covering $\widetilde M$) has no 3-handles, and $\mu_3$ factors through the Hurewicz homomorphism $\pi_3(M) \sra H_3(\widetilde M)=0$ by Lemma~\ref{lem:Hurewicz}. So $|\F_2T_M/\mu_3(\pi_3M)|=|\F_2T_M|=\infty$.

The pseudo-isotopy statement of Corollary~\ref{cor:infinitely-many} follows from Lemma~\ref{lem:pseudo-isotopy} below because a diffeomorphism $\varphi:M \times I \stackrel{\cong}{\to} M \times I$ with $\varphi_0=\id$ (the pseudo-isotopy condition) and $\varphi_1(R)=R'$ leads to the concordance $\varphi\circ (R\times\id):S^2 \times I\hra M \times I$ from $R$ to $R'$. This contradicts $\fq(R,R')\neq 0$ by Corollary~\ref{cor:concordance}.
\end{proof}

\begin{lem}\label{lem:pseudo-isotopy}
Let $G:S^2 \hra M$ be framed and fix $n\in\Z$. Then the diffeomorphism group of $M$ acts transitively on embedded spheres $R:S^2 \hra M$ with $G$ as a geometric dual and normal Euler number $e(\nu R)=n$.
\end{lem}

\begin{proof}
Given $G,R$ as above, consider a closed regular neighborhood $\nu(R\cup G)\subset M$. It is diffeomorphic to the $4$-manifold $M_n$ with one 0-handle and two 2-handles attached to the Hopf link, one 0-framed and the other $n$-framed. In particular, the boundary $\partial M_n$ is a 3-sphere which leads to a decomposition
\[
M \cong M_n \cup_{S^3} M_R,
\]
where $M_R$ is the closure of the complement of $M_n$ in $M$. Note that $G:S^2 \hra M_n\subset M$ is the union of the (core of the) 0-framed 2-handle and a disk bounding the $0$-framed component of the Hopf link. As a consequence, surgery on $G$ in $M_n$ leads to the 4-manifold where that 0-framed 2-handle is replaced by a 1-handle. This 1-handle then cancels the $n$-framed 2-handle, showing that surgery on $G$ leads from $M_n$ to $D^4$. It follows that surgery on $G$ also leads from $M$ to $M_R\cup_{S^3}D^4$.

If $R':S^2 \hra M$ with $e(\nu R')=e(\nu R)$ also has $G$ as a geometric dual, then
repeating the same constructions for $R'$ in place of $R$, we get a second decomposition
\[
M \cong M_n \cup_{S^3} M_{R'},
\]
where $M_{R'}\cup_{S^3}D^4$ is diffeomorphic to surgery on $G$ in $M$. But $G$ is a {\em common} dual, so we get an orientation preserving diffeomorphism $M_R \cong M_{R'}$. Since orientation preserving diffeomorphisms of $S^3$ are isotopic to the identity, we can extend this to a self-diffeomorphism of $M$ which carries $R$ to $R'$ and fixes $G$: This just requires to line up the 2-handles of $M_n$ in the obvious way.
\end{proof}

Corollary~\ref{cor:unknotting number} states that
for $R,R'\in \cR^G_{[f]}$, the relative unknotting number equals the support of the Freedman-Quinn invariant: $\un(R,R') = |\fq(R,R')|$.

Here $\un(R,R')\in \N_0$ denotes the minimal number of finger moves required in any regular homotopy between $R$ and $R'$,
and $|\fq(R,R')|$ is the minimum number of non-zero coefficients in any representative of $\fq(R,R')\in \F_2T_M/\mu_3(\pi_3M)$.

\begin{proof} For $t\in\F_2T_M$, the relative unknotting number satisfies $\un(t\cdot R,R) \leq |t|$ because $t\cdot R$ is constructed from $R$ by using $|t|$ finger moves. Moreover, any $R'\in  \cR^G_{[f]}$ is isotopic to some $t\cdot R$, so it suffices to understand those particular numbers. If $[t] = [s]$ then $t\cdot R=[t]\cdot R$ is isotopic to  $s\cdot R$, so $\un(t\cdot R,R) \leq |s|$ holds as well. 

If $u:=\un(t\cdot R,R)$ then there are $u$ finger moves and then $u$ Whitney moves that lead from $R$ to $t\cdot R$. By general position, we may assume that the finger moves are disjoint from $G$ and run along group elements $g_i\in \pi_1M, i=1,\dots, u$. By Lemma~\ref{lem:choice-of-disks-exists} we find Whitney disks with the same sheet choices in the complement of $G$, and by Lemma~\ref{lem:sheets give action} they also lead to $t\cdot R$. This implies that $u$ is at least as large as the number of 2-torsion $s_j$ among the $g_i$ which by itself equals $|s|$ for $s:=\sum_j s_j$. So we get $u\geq|s|$ and together $u=|[t]|$ as claimed. 
\end{proof}

\section{Ambient Morse theory and the $\pi_1$-negligible embedding Theorem}\label{sec:ambient-morse-pi1-embedding-thm}

A third proof of Gabai's LBT arises from ambient Morse theory and the uniqueness part of the $\pi_1$-negligible embedding theorem \cite[Thm.10.5A(2)]{FQ, Stong}. We only state it in the orientable, non $s$-characteristic case that we are going to use because our dual $G$ is framed. Recall that an embedding $h:V\hra W$ is \emph{$\pi_1$-negligible} if the inclusion induces an isomorphism $\pi_1(W \smallsetminus h(V)) \cong \pi_1W$, which is guaranteed by a dual sphere.

\begin{Thm10}\label{thm:FQ-10.5-uniqueness}
Let $(V; \partial_0 V, \partial_1 V)$ be a compact $4$--manifold triad so that $\pi_1(V, \partial_0 V)=\{1\}=\pi_1(V,\partial_1V)$ (all basepoints), each component has nonempty intersection with $\partial_1 V$, and components disjoint from $\partial_0 V$ are 1-connected. 

Suppose $W$ is an oriented $4$--manifold, $h, h':(V,\partial_0 V), \hra (W,\partial W)$ are $\pi_1$-negligible embeddings, both not $s$-characteristic, and $H$ is a homotopy rel $\partial_0 V$. Then there is an obstruction $\fq(H)\in H^2 (V,\partial_0 V; \F_2T_W)$ which vanishes if and only if $H$ is homologous (with $\Z[\pi_1W]$-coefficients) to a $\pi_1$-negligible concordance $V \times I \hra W \times I$ from $h$ to $h'$.$\hfill\square$
\end{Thm10}
Stong extends this theorem to the $s$-characteristic case in \cite[p.2]{Stong} by showing that there is a secondary obstruction, the {\it Kervaire-Milnor invariant}, to finding a concordance. 

Stong also observes on the bottom of \cite[p.2]{Stong} that $\fq$ can be strengthened to be independent of $H$ by taking $\fq(h,h')$ in the quotient of $H^2 (V,\partial_0 V; \F_2T_W)$ by the self-intersection invariant on $\pi_3W$. Note that this is a 5-dimensional result so it holds in the smooth category.

We apply this theorem for $W$ defined to be the manifold $M$, with an open neighborhood of $G$ removed, and $V:= D^2 \times D^2$ with $\partial_0 V= S^1 \times D^2$ and $\partial_1 V = D^2 \times S^1$. Then $R, R'$ can be turned into embeddings $h, h': (V,\partial_0 V) \hra (W,\partial W)$ by using the normal bundles of $R,R'$ and removing a neighborhood of their intersection point with $G$. Note that $R$ may have non-trivial normal bundle (necessarily isomorphic to that of $R'$) but after removing the neighborhood of $G$, it turns into a $D^2$-bundle over $D^2$ which must be trivial.

By Lemma~\ref{lem:based}, the resulting embeddings $h, h'$ are homotopic rel $\partial_0 V$ and the theorem applies. Note that
$(V,\partial_0 V)\simeq (D^2, S^1)$ and hence the invariant $\fq(R,R')=\fq(h, h')$ lies in $\F_2T_W/\mu_3(\pi_3W) $. Note also that Seifert--van Kampen shows that in this case, every concordance is $\pi_1$-negligible (as long as it is on one boundary). 

If $\fq(h,h')=0$ then $h$ and $h'$ are concordant by the above theorem.  We now reverse the above steps of thickening spheres and disks to 4--manifolds with boundary to arrive at a concordance $C: S^2 \times I \hra M \times  I$ between $R$ and $R'$ exactly as in the corollary below. Note that Stong's additional Kervaire-Milnor invariant vanishes in our setting since $R$ is not s-characteristic: The dual sphere $G$ is framed, so that 
\[
R\cdot G \equiv 1 \neq 0 \equiv G\cdot G \mod 2.
\] 

\begin{cor}\label{cor:concordance}
Given embedded spheres $R,R'\in \cR^G_{[f]}$ as in Theorem~\ref{thm:4d-light-bulb}, the obstruction $\fq(R,R')\in \F_2T_M/\mu_3(\pi_3M)$ vanishes if and only if there is a concordance $C: S^2 \times I \hra M \times  I$ between $R$ and $R'$ that has $G$ as a geometric dual in every level: $C^{-1}(G \times \{t\}) = (z_0,t).\hfill\square$ 
\end{cor}
By the following result, which will be proven using Morse theory for the 3-manifold $S^2 \times I$ and only basic lemmas from this paper, the Freedman--Quinn invariant completely detects isotopy in this setting:
\begin{thm}\label{thm:concordance implies isotopy}
Given a concordance $C: S^2 \times I \hra M \times  I$ between $R$ and $R'$ which has $G$ as a geometric dual in every level $t\in I$ as in the above corollary, it follows that $R$ and $R'$ are isotopic.
\end{thm}

\begin{proof}
We now show how to directly turn the concordance $C$ into an isotopy using the geometric duals.
By general position, we may assume that the composition $p_2\circ C: S^2 \times I \to I$ is a Morse function. If it has no critical points then $C$ is the track of an isotopy, so we'll study the critical points of $p_2\circ C$ by Morse theory. Extend a gradient-like vector field on $S^2 \times I$ to $M \times I$ that has no additional critical points and flows downwards in the $I$-direction away from the image of $C$. 
Then the relative local models for $(M \times I, C(S^2 \times I))$ at the critical points, corresponding to the $k$-handles of the 3-manifold $S^2 \times I$, $k=0,1,2,3$, are well known. 
In \cite[Lem.8]{BT} it is proven by a dimension count for ascending and descending manifolds of the vector field that by an ambient isotopy of $M \times  I$ one can order the critical points according to their index. 

Moreover, one can also re-order critical points of the {\it same} index arbitrarily which can be seen as follows, say in the case of 1-handles: The core of a 1-handle (in the 3-manifold $S^2 \times I$) is an arc, whereas the cocore is a 2-disk. If we have two adjacent 1-handles, one just below a level $M:=M \times t$ and the other just above $M$, then we can push the cocore up and the core down into that ``middle'' level $M$. By general position, this 1-manifold and 2-manifold will not intersect in the ambient 4-manifold $M$ and hence we can push the upper 1-handle below the lower one. 

As a consequence, we can assume that our Morse function on $S^2 \times I$ first has $n$ minima (0-handles) which are then abstractly cancelled by $n$ 1-handles: Each 0-handle must be abstractly cancelled eventually and we can slide those cancelling 1-handles below the other 1-handles.  Looking at the top, $m$ maxima arise that are abstractly cancelled by $m$ 2-handles. 
The remaining 1-handles have both their feet on $R$ by construction, similarly for the 2-handles read in the other direction.

The remaining 1- and 2-handles form a third cobordism which must be diffeomorphic to $S^2\times I$ since gluing $S^2 \times I$ to its top and bottom gives the entire cobordism $S^2 \times I$. 

More precisely, we can find two non-critical levels $t_1 < t_2$ in $(0,1)$ such that $C^{-1}(M \times  \{t_i\})$ are spheres which separate the domain $S^2 \times I$ of $C$ into three product cobordisms: 
\[
V_i := C^{-1}(M \times  [t_i, t_{i+1}]), \ i=0,1,2  \text{ and } t_0:= 0, t_3:=1.
\]
Here $V_i\cong S^2 \times I$ consists of the $i$- and $(i+1)$-handles discussed above. The proof of Theorem~\ref{thm:concordance implies isotopy} will be completed by the subsequent lemmas which show that each of the three restrictions of $C$ to $V_i$ can be turned into an isotopy, using the geometric dual $G$.
\end{proof}

For $V_0$, the $t$-parameter gives a movie in $M$ that starts with $R$ and then shows $n$ trivial spheres $S_1,\dots,S_n$ being born in $M$, one for each $0$-handle. Then $n$ tubes form, one for each $1$-handle, that connect $R$ to each $S_i$, making the result a new sphere $R_w$ in $M$. Here $w$ is a collection of $n$ words in the free product $\pi_1M \ast F_n\cong\pi_1(M\setminus\cup_iS_i)$, where $F_n$ is the free group generated by the meridians $m_i$ to $S_i$, and the words $w_k$ in $w$ measure how the core arcs $a_k$ of the $1$-handles hit the cocore $3$--balls $B_i$ of the $0$-handles bounded by the $S_i$ in $M$: Each intersection point with $B_i$ reads out the letter $m_i$, whereas the letters from $\pi_1M$ arise from the arcs that are in between such intersection points. The arcs can be turned into based loops after picking whiskers from each $S_i$ to the base-point of $M$. The argument below does not depend on those choices. 

These cocores and cores originally lie in $M \times [0,t_1]$ but we pushed the cocores up and the cores down into a common middle level $M=M \times t_1/2$. By the above reordering argument, the collection $\cC$ of cocores is embedded disjointly into $M$ and similarly, the collection  $\cC'$ of cores is also embedded disjointly. However, these 3-- and 1--manifolds can intersect each other in the 4-dimensional middle level $M$, so the abstract handle cancellation can a priori not be done ambiently in $M$.

\begin{lem}\label{lem:0-1-handles}
The sphere $R_w$ is isotopic to $R$ in any neighborhood of $R\cup \cC\cup\cC'\cup G$ in $M$.
\end{lem}
\begin{figure}[ht!]
        \centerline{\includegraphics[scale=.24]{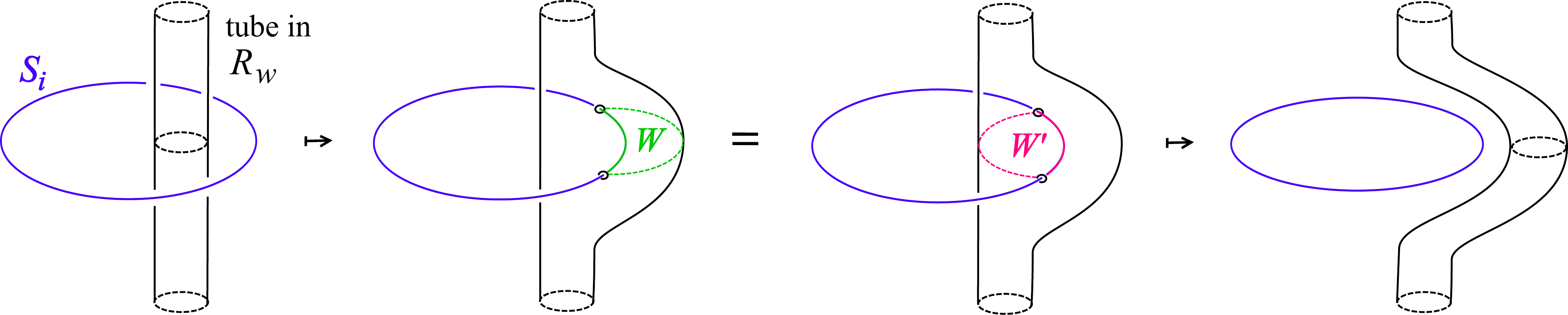}}
        \caption{Pushing core arcs out of cocore 3-balls}
        \label{fig:pullout-tube}
\end{figure}
\begin{proof}
Figure~\ref{fig:pullout-tube} shows how we can reduce the number of occurences of the meridian $m_i$ in $w$. This is a finger move and then a Whitney move on $R_w$, and as usual we see two Whitney disks, $W$ going back to $R_w$ by the inverse of the finger move and $W'$ going forward. These Whitney disks share a boundary arc $\beta$, and by Proposition~\ref{prop:common-arcs-w-moves} it follows that $R_w$ is isotopic to the result $R_{w'}$ of the Whitney move along $W'$, with $w'$ containing one letter $m_i$ fewer then $w$. Iterating this procedure we see that $R_w$ is isotopic to $R_{w_0}$ where $w_0\in\pi_1M$. This means that the 1-handles for $R_{w_0}$ do not intersect the cocore $3$--balls for the $0$-handles. These $3$--balls  and copies of $D^2\times D^1$ inside the tubes then provide the final isotopy from $R_{w_0}$ to $R$. 
\end{proof}

Applying the same arguments of Lemma~\ref{lem:0-1-handles} to $V_2$ turned upside down shows that the restriction of $C$ to $V_2$ can be replaced by an isotopy. So it just remains 
to show that the restriction of $C$ to $V_1$ can be replaced by an isotopy. 

The $t$-parameter movie for $V_1$ starts with the sphere $R$ at $t=t_1$ then $g$ tubes form, one for each remaining 1-handle in $V_1$. We then see a surface $F$ of genus $g$ in the middle level $M$ in which the collection $\cC$ of cocores is also embedded. These are 2--disks, or better, a collection of $g$ disjoint caps (section~\ref{sec:capped-surface-w-move}) attached to a half-basis of disjointly embedded simple closed curves in $F$. The movie continues with $g$ 2-handles being attached to $F$ whose cores form a second collection of caps $\cC'$, again embedded disjointly into the middle level $M$. 

\begin{lem}\label{lem:1-2-handles}
The sphere $R$ is isotopic to $R'$  in any neighborhood of $F\cup\cC\cup\cC'\cup G$ in $M$.
\end{lem}

\begin{proof}
By construction, we have a genus $g$ surface $F\subset M$, together with a collection $\cC$ of $g$ caps such that surgery leads to $R$, and another collection $\cC'$ of $g$ caps for $F$ that surger it to $R'$. The caps in each collection are embedded in $M$, and disjoint from all other caps in the same collection, but caps of different collections may intersect on their boundary (in $F$) as well as in their interiors.  

There are two handle-bodies $Y$ and $Y'$ formed from $F \times [-\epsilon,\epsilon]$ by (abstractly) attaching thickened caps from $\cC$ to 
$F \times -\epsilon$, respectively $\cC'$ to $F \times \epsilon$, and then filling the resulting boundary with two $3$--balls. 
This is a Heegaard decomposition of $S^3$ to which we will next apply some classical $3$--manifold results to simplify the intersection pattern in $F$ between the boundaries of the caps in $\cC$ and those in $\cC'$.

Waldhausen's uniqueness theorem for Heegaard decompositions of $S^3$ \cite{Waldhausen} gives a diffeomorphism of triples (isotopic to the identity -- but we won't use this here) 
\[
(S^3; Y, Y') \cong (S^3; Y_0, Y'_0)
\] 
where the subscript $0$ refers to the standard Heegaard decomposition, stabilized to be of the same genus as $Y$. In the following, we'll need the usual notion of {\it minimal systems of disks}, which are disjointly embedded disks that cut a handlebody into a $3$--ball. For $Y$, respectively $Y'$, such minimal systems are given by the caps in $\cC$, respectively $\cC'$. On the $(Y_0, Y_0')$-side these are standard disks in the sense that their boundaries 
intersect 
$\delta_{ij}$ geometrically. By applying Waldhausen's diffeomorphism, we see that $Y$ and $Y'$ admit minimal systems of disks that also intersect $\delta_{ij}$ geometrically on the boundary.

A result of Reidemeister \cite{Reidemeister} and Singer \cite{Singer} from 1933 asserts that any two minimal systems of disks in a handlebody are slide equivalent. This implies that after finitely many handle slides among the abstract caps in $\cC$, respectively $\cC'$, we may assume that the collections of caps $\cC$ and $\cC'$ intersect $\delta_{ij}$ geometrically on the boundary. These handle-slides can be achieved ambiently in $M$ and we'll assume from now on that this has been done. This has the consequence that the complement in $F$ of the boundaries of the caps in $\cC$ and $\cC'$ is connected. In particular,  in the following arguments we may always find (disjoint) arcs in $F$ from any point in this complement to the intersection point of $F$ and $G$. 

If the interiors of all caps happen to be disjoint then Lemma~\ref{lem:capped-surface-isotopy} shows that the two surgeries $R$ and $R'$ are isotopic in $M$. We will complete our proof of Lemma~\ref{lem:1-2-handles} by showing the following general result.
\end{proof}

\begin{lem}\label{lem:disjoint-caps}
Let $F$ be a surface in a 4--manifold admitting a collection $\cC$ of disjoint caps $c_i$, and also admitting another collection $\cC'$ of disjoint caps $c_i'$, such that the $\partial c_i$ intersect the $\partial c'_i$  geometrically $\delta_{ij}$ in $F$.

If $F$ has a geometric dual $G$ which is disjoint from $\cC\cup \cC'$ then there exists a collection $\cC''$ with the same boundaries as $\cC$, which has no interior intersections with $\cC'$, and such that surgery on $\cC''$ is isotopic to surgery on $\cC$. 
\end{lem}
Recall that by definition (section~\ref{sec:capped-surface-w-move}) the interiors of all caps are embedded in the complement of $F$.
And in our current setting of the proof of Lemma~\ref{lem:1-2-handles} the geometric dual $G$ to $F$ is indeed disjoint from $\cC\cup \cC'$.
Note that Lemma~\ref{lem:capped-surface-isotopy} then implies that surgery on $\cC$ is also isotopic to surgery on $\cC'$, which we wanted to prove.

\begin{proof}
Our construction will eliminate each intersection point $p\in c_i\pitchfork c'_j$ for $c_i\in\cC$ and $c'_j\in\cC'$ by tubing $c_i$ into a dual sphere $S_j$ to $c'_j$. This does not change $F_{\cC'}$ since $\cC'$ is fixed, and it will be checked that the tubing of the $c_i$ into the $S_j$ does not change $F_\cC$ up to isotopy. 

\begin{figure}[ht!]
        \centerline{\includegraphics[scale=.275]{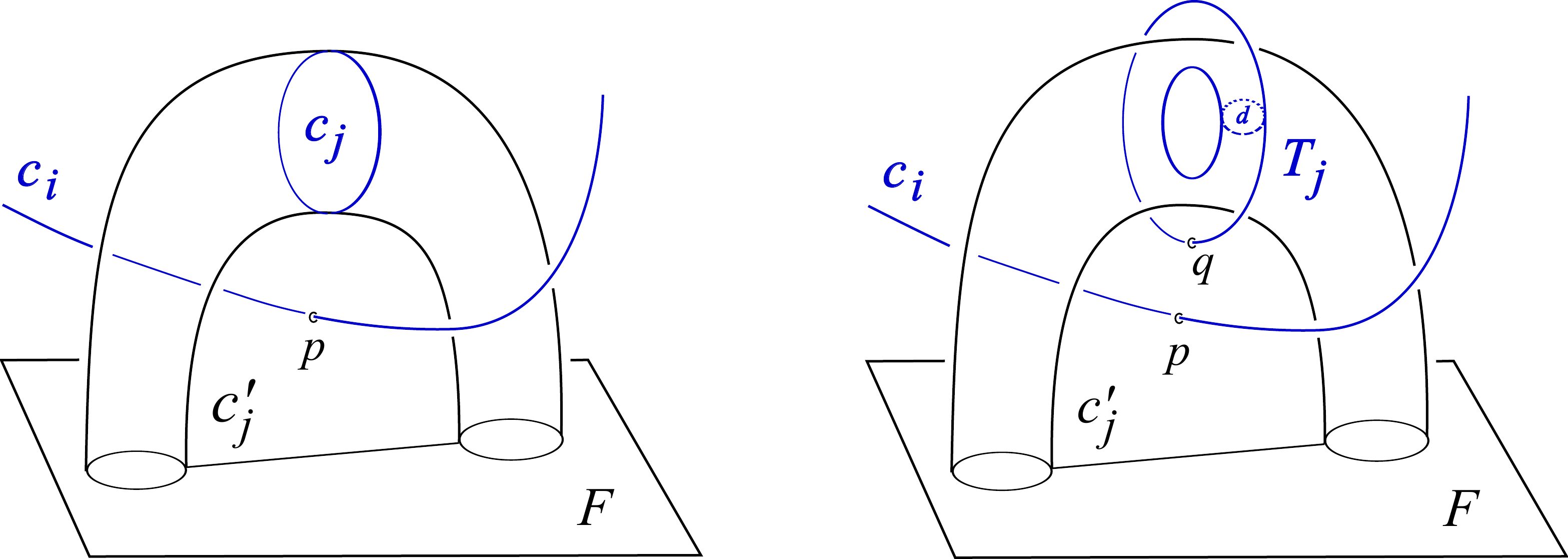}}
        \caption{Left: An intersection $p\in c_i\pitchfork c'_j$. Right:  A torus $T_j$ of normal circles over $\partial c_j$ with $T_j\cap c'_j=\{q\}$.}
        \label{fig:cap-clean-up-1}
\end{figure}

We first describe the easiest case where $\cC\pitchfork \cC'$ is a single interior intersection $p\in c_i\pitchfork c'_j$ for some $c_i\in\cC$ and $c'_j\in\cC'$ with $i\neq j$ (Figure~\ref{fig:cap-clean-up-1}, left). By assumption there exists a cap $c_j\in\cC$ whose boundary $\partial c_j$ intersects $\partial c'_j$ in a single point. A torus $T_j$ of normal circles to $F$ over $\partial c_j$ intersects the interior of $c'_j$ in a single point $q$ (Figure~\ref{fig:cap-clean-up-1}, right). 
Let $d$ be a meridional disk to $F$ bounded by a circle in $T_j$, and denote by $d_G$ the result of tubing $d$ into $G$ to eliminate the intersection between $d$ and $\partial c_j$ (as in Figure~\ref{fig:tube-caps-2} but here $\partial d\subset T_j$). Then surgering $T_j$ along $d_G$ yields a $0$-framed embedded sphere $S_j$ with $q=S_j\cap c'_j$, such that $S_j$ is disjoint from all other caps in $\cC'$, and $S_j$ is disjoint from all caps in $\cC$ (Figure~\ref{fig:cap-clean-up-2}, left). So the intersection $p$ can be eliminated by tubing $c_i$ into $S_j$ along a path between $p$ and $q$ in $c_j$ (Figure~\ref{fig:cap-clean-up-2}, right).

\begin{figure}[ht!]
        \centerline{\includegraphics[scale=.275]{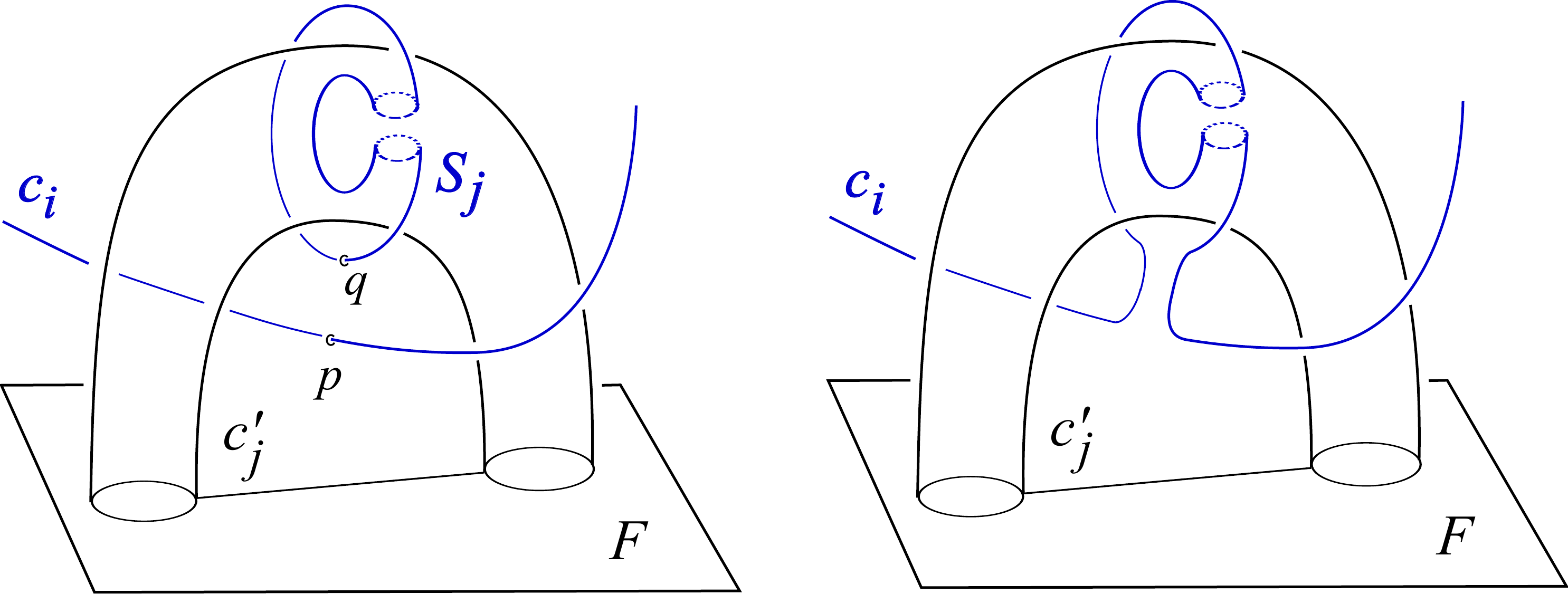}}
        \caption{Left: The sphere $S_j$ with $S_j\cap F=\{q\}$. Right: The result of tubing $c_i$ into $S_j$ to eliminate $p$ and $q$.}
        \label{fig:cap-clean-up-2}
\end{figure}
At this point we have eliminated $p\in c_i\pitchfork c'_j$ by replacing $c_i$ with the connected sum $c''_i:=c_i\# S_j$ of $c_i$ with $S_j$ to get a new collection of caps $\cC''$ with the same boundaries as $\cC$ but with interiors disjoint from $\cC'$. We want to check that $F_\cC$ is isotopic to $F_{\cC''}$.
Note that $T_j$ also admits a cap $\gamma_j$ formed from $c_j$ by deleting a small collar. (The boundary of $\gamma_j$ is visible in the right side of Figure~\ref{fig:cap-clean-up-1} as the ``inner longitude'' of $T_j$.) This cap $\gamma_j$ is disjoint from $F$ and is dual to $d_G$, so it follows from the capped surface isotopy lemma (Lemma~\ref{lem:capped-surface-isotopy}) that the sphere $S_j^\gamma$ formed by surgering $T_j$ along $\gamma_j$ is isotopic to $S_j$ in the complement of $F$. So it suffices to check that $F_\cC$ is isotopic to $F_{\cC^\gamma}$, where the collection of caps $\cC^\gamma$ differs from the original $\cC$ by
replacing $c_i$ with $c_i\# S_j^\gamma$.

The sphere $S_j^\gamma$ is contained in the boundary of a tubular neighborhood $\nu_{c_j}\cong D^2\times D^2$ of $c_j$, and $S_j^\gamma$ bounds an embedded $3$--ball $B_j^\gamma\subset\nu_{c_j}$ which is the union of the solid torus $\partial c_j\times D^2$ with a $1$-dimensional sub-bundle over the interior of $c_j$. Observe that the only intersections between $B_j^\gamma$ and $F$ are the circle $\partial c_j$.

Now surger $F$ along $\cC^\gamma$ to get $F_{\cC^\gamma}$. Since surgery has deleted a regular $\epsilon$-neighborhood of $\partial c_j$ from $F$, the $3$--ball $B^\gamma_j$ is now disjoint from $F_{\cC^\gamma}$. So there exists an isotopy from $F_{\cC^\gamma}$ to $F_\cC$ supported near $B^\gamma_j$ which isotopes the two parallel copies of $c_i\# S_j^\gamma$ in $F_{\cC^\gamma}$ to the two parallel copies of $c_i$ in $F_\cC$ by shrinking the parallels of $S_j^\gamma$ in $B^\gamma_j$.

The description of how this construction can be carried out in the general case to simultaneously eliminate any number of intersections $p\in c_i\pitchfork c'_j$ among all the $c_i\in\cC$ and $c'_j\in\cC'$ is straightforward: Consider some $c'_j$ which has multiple interior intersections with multiple $c_i$ (in the left of Figure~\ref{fig:cap-clean-up-1} imagine more $p$-intersections). We will not introduce sub-index notation to enumerate the interior intersections in each $c'_j$, nor for the subsequent tori and spheres created for each intersection.
Take a torus $T_j$ as in the right of Figure~\ref{fig:cap-clean-up-1}
around a parallel copy of $\partial c_j$ for each interior intersection. (Note that these parallels of $\partial c_j$ and their corresponding disjoint normal tori can be assumed to be supported arbitrarily close to $\partial c_j$, ie.~in the part of $F$ that will be deleted by surgery -- this observation is key to why the general case will present no new difficulties.) Just as above, these tori can be surgered to spheres $S_j$ disjoint from $F$ which are dual to $c'_j$ using caps $d_G$ on the $T_j$ in the complement of $F$ created by tubing meridional disks into $G$ along disjointly embedded arcs in $F$. These $S_j$ are all disjointly embedded by construction. Now all intersections between $c'_j$ and the $c_i$ can be eliminated by tubing the $c_i$ into the $S_j$ along disjointly embedded arcs in $c'_j$ between pairs of intersection points in $c_i\pitchfork c'_j$ and $S_j\cap c'_j$ (as in the right of Figure~\ref{fig:cap-clean-up-2}). Note that the case $i=j$ is allowed in this construction since the tori are supported near the parallel copies of $\partial c_j$ and the $S_j$ are disjoint from all $c_i$, so changing the interior of $c_j$ by tubing into an $S_j$ can be carried out just as for $c_i$ with $i\neq j$. Carrying out this construction for all $c'_j$ replaces $\cC$ with $\cC''$ such that $\cC''$ and $\cC'$ have disjoint interiors (with boundaries unchanged).

It remains to check that the argument from the easy case also applies to show that this construction which has changed the $c_i$ by multiple connected sums has not changed the result of surgery. As before, we can surger each of the $T_j$-tori along a cap $\gamma_j$ formed from a parallel of $c_j$ to get a sphere $S^\gamma_j$  which is isotopic in the complement of $F$ to the corresponding $S_j$. Here we are using parallels of the new $c''_j$ which may been tubed into some $S_k$'s, but the key properties of being framed, with interiors disjointly embedded in the complement of $F$ have been preserved. Since the $\gamma_j$-caps are dual to the $d_G$-caps, the $S_j^\gamma$-spheres are isotopic to the $S_j$-spheres in the complement of $F$, again
by the capped surface isotopy lemma (Lemma~\ref{lem:capped-surface-isotopy}). So again it suffices to check that $F_\cC$ is isotopic to $F_{\cC^\gamma}$
where the collection of caps $\cC^\gamma$ differs from the original $\cC$ by taking connected sums 
of the $c_i$ with multiple $S_j^\gamma$.

Similarly as before, the $S_j^\gamma$ are contained in the boundaries of disjoint tubular $D^2\times D^2$-neighborhoods of parallels of $c_j$, with each of these neighborhoods containing an embedded $3$--ball $B_j^\gamma$ bounded by $S_j^\gamma$ such that $B_j^\gamma$ and $F$ only intersect in the corresponding parallel copy of $\partial c_j$.
Surgering $F$ along $\cC^\gamma$ to get $F_{\cC^\gamma}$ deletes regular $\epsilon$-neighborhoods of all the $\partial c_j$ from $F$, 
and since we may assume that all the $T_j$-tori in the construction were supported near parallels of the $\partial c_j$ that lie inside these deleted 
$\epsilon$-neighborhoods,
all the $B^\gamma_j$-balls are disjoint from $F_{\cC^\gamma}$. So there exists an isotopy from $F_{\cC^\gamma}$ to $F_\cC$ supported near the $B^\gamma_j$ which isotopes the pairs of parallel copies of $c_i\# S_j^\gamma$ in $F_{\cC^\gamma}$ to the pairs of parallel copies of $c_i$ in $F_\cC$ by shrinking the parallels of $S_j^\gamma$ in $B^\gamma_j$.    
\end{proof}


\end{document}